\documentclass{amsart}

\usepackage{amssymb}
\usepackage{amsmath}
\usepackage{braket}
\usepackage{mathrsfs}
\usepackage[all]{xy}
\usepackage{graphicx}
\usepackage{color}

\usepackage[noabbrev,capitalise]{cleveref}

\newtheorem{thm}{Theorem}[section]
\newtheorem{cor}[thm]{Corollary}
\newtheorem{prop}[thm]{Proposition}
\theoremstyle{definition}
\newtheorem{dfn}[thm]{Definition}
\newtheorem{ex}[thm]{Example}

\newtheorem{lem}[thm]{Lemma}

\newtheorem{question}[thm]{Question}

\theoremstyle{remark}
\newtheorem{rem}[thm]{Remark}
\newcommand{\C}{\Csc}               
\newcommand{\D}{\Dsc}               
\newcommand{\E}{\Ebb}               

\newcommand{\Ebb}{\mathbb{E}}        

\newcommand{\Fun}{\mathrm{Fun}}      
\newcommand{\End}{\mathrm{End}}      

\newcommand{\Ab}{\mathit{Ab}}        
\newcommand{\Sets}{\mathit{Set}}     

\newcommand{\op}{^\mathrm{op}}       
\newcommand{\ssh}{_\sharp}           

\newcommand{\id}{\mathrm{id}}         
\newcommand{\se}{\subseteq}           
\newcommand{\ppr}{^{\prime}}          
\newcommand{\pprr}{^{\prime\prime}}   

\newcommand{\co}{\colon}              
\newcommand{\ci}{\circ}               

\newcommand{\uas}{^{\ast}}            
\newcommand{\sas}{_{\ast}}            

\newcommand{\lla}{\longleftarrow}     
\newcommand{\lra}{\longrightarrow}    

\newcommand{\tc}{\Rightarrow}         
\newcommand{\ltc}{\Longrightarrow}    

\newcommand{\EQ}{\Leftrightarrow}     

\newcommand{\dra}{\dashrightarrow}  

\newcommand{\afr}{\mathfrak{a}}  
\newcommand{\bfr}{\mathfrak{b}}  
\newcommand{\cfr}{\mathfrak{c}}  
\newcommand{\dfr}{\mathfrak{d}}  
\newcommand{\efr}{\mathfrak{e}}  
\newcommand{\ffr}{\mathfrak{f}}  
\newcommand{\gfr}{\mathfrak{g}}  
\newcommand{\hfr}{\mathfrak{h}}  
\newcommand{\vfr}{\mathfrak{v}}  
\newcommand{\rfr}{\mathfrak{r}}  
\newcommand{\sfr}{\mathfrak{s}}  
\newcommand{\wfr}{\mathfrak{w}}  
\newcommand{\xfr}{\mathfrak{x}}  
\newcommand{\yfr}{\mathfrak{y}}  

\newcommand{\Csc}{\mathscr{C}}   
\newcommand{\Dsc}{\mathscr{D}}   
\newcommand{\Esc}{\mathscr{E}}   

\newcommand{\Xcal}{\mathcal{X}} 
\newcommand{\Ycal}{\mathcal{Y}} 
\newcommand{\Zcal}{\mathcal{Z}} 
\newcommand{\und}{\underline}   
\newcommand{\ovl}{\overline}    
\newcommand{\ov}{\overset}      
\newcommand{\ti}{\times}        
\newcommand{\am}{\amalg}        

\newcommand{\spmatrix}[2]{\big[\!\!\!\begin{array}{c}{}_{#1}\\ {}^{#2}\end{array}\!\!\!\!\big]}

\newcommand{\SQ}[2]{({#1}\backslash {#2})}
\newcommand{\RT}[3]{\left[\begin{array}{c|c|c}{#1}&{#2}&{#3}\end{array}\right]}
\newcommand{\CB}[6]{\left[\begin{array}{c|c|c} {#1}&{#2}&{#3}\\{#4}&{#5}&{#6}\end{array}\right]}
\newcommand{\Penta}[9]{
\xy
(0,13)*+{#1}="0";
(-12.5,4.5)*+{#2}="1";
(-7.5,-10)*+{#3}="2";
(7.5,-10)*+{#4}="3";
(12.5,4.5)*+{#5}="4";
{\ar_{#6} "0";"1"};
{\ar^{} "0";"2"};
{\ar^{} "0";"3"};
{\ar^{} "0";"4"};
{\ar_{#7} "1";"2"};
{\ar^{} "1";"3"};
{\ar^{} "1";"4"};
{\ar_{#8} "2";"3"};
{\ar^{} "2";"4"};
{\ar_{#9} "3";"4"};
\endxy
}

\newcommand{\TwoSP}[7]{
\xy
(-7,7)*+{#1}="0";
(-7,-7)*+{#2}="4";
(7,-7)*+{#3}="6";
{\ar_{#4} "0";"4"};
{\ar_{#5} "4";"6"};
{\ar^{#6} "0";"6"};
(-3,-7)*+{}="10";
(-7,-3)*+{}="11";
{\ar@/_0.2pc/@{-}_{_{#7}} "10";"11"};
\endxy
}

\newcommand{\ThreeSP}[4]{
\xy
(-1,7)*+{#1}="0";
(-8,-2)*+{#2}="1";
(1,-10)*+{#3}="2";
(9,-1)*+{#4}="3";
{\ar_{} "0";"1"};
{\ar^{} "0";"2"};
{\ar_{} "0";"3"};
{\ar_{} "1";"2"};
{\ar^{}|!{"0";"2"}\hole "1";"3"};
{\ar_{} "2";"3"};
\endxy
}

\newcommand{\roundup}[2]{\lceil{#1}|{#2}\rceil}

\newcommand{\bsm}{\begin{smallmatrix}}
\newcommand{\esm}{\end{smallmatrix}}

\newcommand{\ff}{\mathrm{f}}     
\newcommand{\bb}{\mathrm{b}}     

\newcommand{\al}{\alpha}         
\newcommand{\be}{\beta}          
\newcommand{\gam}{\gamma}        
\newcommand{\vp}{\varphi}        
\newcommand{\vt}{\vartheta}      
\newcommand{\kap}{\kappa}        
\newcommand{\lam}{\lambda}       
\newcommand{\ups}{\upsilon}      
\newcommand{\om}{\omega}         
\newcommand{\thh}{\theta}        
\newcommand{\del}{\delta}        
\newcommand{\ze}{\zeta}          

\newcommand{\Thh}{\Theta}        

\newcommand{\Sig}{\Sigma}        
\newcommand{\Om}{\Omega}         

\newcommand{\Cf}{\mathbf{C}}     
\newcommand{\Df}{\mathbf{D}}     
\newcommand{\Ff}{\mathbf{F}}     
\newcommand{\If}{\mathbf{I}}     
\newcommand{\Nf}{\mathbf{N}}     
\newcommand{\Of}{\mathbf{O}}     
\newcommand{\Pf}{\mathbf{P}}     
\newcommand{\Qf}{\mathbf{Q}}     
\newcommand{\Rf}{\mathbf{R}}     
\newcommand{\Sf}{\mathbf{S}}     
\newcommand{\Tf}{\mathbf{T}}     
\newcommand{\Uf}{\mathbf{U}}     
\newcommand{\Vf}{\mathbf{V}}     
\newcommand{\Zf}{\mathbf{Z}}     

\newcommand{\infl}{\rightarrowtail}
\newcommand{\defl}{\twoheadrightarrow}

\numberwithin{equation}{section}

\begin{document}

\title[Homotopy category of exact quasi-category]{External triangulation of the homotopy category of exact quasi-category}

\author{Hiroyuki Nakaoka}
\address{Graduate School of Mathematics, Nagoya University, Furocho, Chikusaku, Nagoya 464-8602, Japan}
\email{nakaoka.hiroyuki@math.nagoya-u.ac.jp}
\author{Yann Palu}
\address{LAMFA, Universit\'e de Picardie Jules Verne, 33 rue Saint-Leu, Amiens, France}
\email{yann.palu@u-picardie.fr}
\urladdr{http://www.lamfa.u-picardie.fr/palu/}


\thanks{}

\thanks{The first author is supported by JSPS KAKENHI Grant Numbers JP19K03468, JP20K03532. Both authors were partially supported by the French ANR grant SC3A~(15\,CE40\,0004\,01). Both authors wish to thank Gustavo Jasso for his interest, and for introducing them a possible relation with a future work of his student.}

\begin{abstract}
Extriangulated categories axiomatize extension-closed subcategories of triangulated categories. We show that the homotopy category of an exact quasi-category can be equipped with a natural extriangulated structure.
\end{abstract}

\maketitle

\tableofcontents

\section{Introduction}
For triangulated categories, several possible higher-categorical enhancements are known, one of which is the celebrated notion of a \emph{stable $\infty$-category} by Lurie (\cite{L2}). 
Recently, as a wider class of $\infty$-categories, the notion of an \emph{exact $\infty$-category} has appeared in Barwick's \cite{B1}, where it is used in order to prove that the $K$-theory of an idempotent complete stable quasi-category endowed with a bounded t-structure is weakly equivalent to the $K$-theory of its heart. A non-additive version of exact quasi-categories also appeared in \cite{DK} as a source of examples of 2-Segal spaces, via Waldhausen's S-construction. Instances of exact quasi-categories are given by
\begin{itemize}
\item the nerve of any ordinary exact category,
\item any extension-closed full additive subcategory of a stable $\infty$-category
\end{itemize}
as stated in \cite[Examples~3.3 and 3.5]{B1}.

Regarding its importance and naturality, it should be natural to expect that the homotopy category of an exact $\infty$-category would have a nice structure very close to triangulation.
In this article, we give a positive answer to this expectation, with the notion of an \emph{extriangulated category} which has been introduced as a unification of exact categories and triangulated categories in our previous article \cite{NP}. It is defined as an additive category equipped with a structure called an \emph{external triangulation}, which is more flexible than triangulations in the sense that it is naturally inherited by relative theories, extension-closed subcategories, and ideal quotients by some projective-injectives (\cite[Section~3.2]{HLN}, \cite[Remark~2.18, Proposition~3.30]{NP}). In particular, extension-closed subcategories of triangulated categories are typical examples of extriangulated categories, which dovetails the above-mentioned feature of exact $\infty$-categories.

Let $\C$ be an exact quasi-category (the definition is recalled in \cref{Section_ExactSequences}).
We say that two exact sequences in $\C$
\[
\xy
(-12,0)*+{\Sf =};
(-6,6)*+{A}="0";
(6,6)*+{B}="2";
(-6,-6)*+{O}="4";
(6,-6)*+{C}="6";
{\ar@{>->}^i "0"+(3.5,0);"2"};
{\ar "0";"4"};
{\ar@{->>}^p "2";"6"};
{\ar "4";"6"};
(24,0)*+{\Sf\ppr =};
(30,6)*+{A}="10";
(42,6)*+{B\ppr}="12";
(30,-6)*+{O\ppr}="14";
(42,-6)*+{C}="16";
{\ar@{>->} "10"+(3.5,0);"12"};
{\ar "10";"14"};
{\ar@{->>} "12";"16"};
{\ar "14";"16"};
\endxy
\]
are equivalent if there is a cube in $\C$ from $\Sf$ to $\Sf\ppr$
\[
\Cf=\ 
\xy
(-8,9)*+{A}="0";
(8,9)*+{B}="2";
(1,3)*+{O}="4";
(17,3)*+{C}="6";
(-8,-7)*+{A}="10";
(8,-7)*+{B\ppr}="12";
(1,-13)*+{O\ppr}="14";
(17,-13)*+{C}="16";
{\ar "0";"2"};
{\ar "0";"4"};
{\ar "2";"6"};
{\ar "4";"6"};
{\ar_{a} "0";"10"};
{\ar "2";"12"};
{\ar "4";"14"};
{\ar^{c} "6";"16"};
{\ar "10";"12"};
{\ar "10";"14"};
{\ar "12";"16"};
{\ar "14";"16"};
\endxy
\]
where $\ovl{a}=\id_A$ and $\ovl{c}=\id_C$ in the homotopy category $h\C$.
This is an equivalence relation (\cref{PropEquivExSeq}), and we define $\E(C,A)$ to be the set of equivalence classes of exact sequences with fixed end terms $A$ and $C$.
Then $\E\co(h\C)\op\ti(h\C)\to\Ab$ is an additive bifunctor (\cref{CorEAddFtr}).
We also define an additive realization $\sfr$ (\cref{LemET}) by sending the equivalence class of $\Sf$ to the equivalence class (defined similarly as for short exact sequences) of $A\ov{\ovl{i}}{\lra}B\ov{\ovl{p}}{\lra}C$ in $h\C$.

\begin{thm}[\cref{ThmET,propExactFunctors}]
 Let $\C$ be an exact quasi-category. Then
 \begin{enumerate}
  \item its homotopy category $(h\C,\E,\sfr)$ is extriangulated.
  \item Moreover, the functor $\C\to h\C$ sends exact sequences $\Sf$ as above to extriangles 
  $A\ov{\ovl{i}}{\infl}B\ov{\ovl{p}}{\defl}C\ov{\und{\Sf}}{\dashrightarrow}$.
  \item If $F\co\C\to\D$ is an exact functor of exact quasi-categories, then the induced functor $hF\colon h\C\to h\D$ is an exact functor of extriangulated categories.
 \end{enumerate}
\end{thm}

When $\C$ is stable, this recovers the usual triangulated structure on $h\C$ (\cref{PropCompatStable}).

In Section~\ref{Section_NotationsAndTerminology}, we organize notations and terminology which will be used in the proceeding sections. In Section~\ref{Section_ExactSequences} we remind the definition of an exact $\infty$-category in \cite{B1}, and summarize properties of exact sequences.
In Section~\ref{Section_ExtriangulatedStructure}, after a brief introduction of the definition of an extriangulated category, we proceed to construct an external triangulation of the homotopy category. As the main theorem (Theorem~\ref{ThmET}), indeed we will show that the homotopy category of any exact $\infty$-category becomes an extriangulated category. In the last Subsection~\ref{Subsection_StableCase}, in the  stable case we will also see  that this extriangulated structure is compatible with the triangulation given in \cite{L2}.

This work came out of our exploration for an existing higher-categorical notion appropriate to enhance extriangulated categories. We hope that the existence of such a notion in turn guarantees the validity and naturality of the notion of an extriangulated category.

\section{Notations and terminology}\label{Section_NotationsAndTerminology}

Basically we follow the notations and terminology in \cite{L1} and \cite{B1}, while we use the term \emph{quasi-category} as in Joyal's \cite{J}, instead of $\infty$-category. Besides, we will also employ some ad hoc ones which we will introduce in this section, to facilitate the later argument dealing with exact sequences.

Throughout this article, $\C$ denotes an essentially small quasi-category (\cite[Notation~1.3]{B2}). By definition it is a simplicial set which admits every inner horn extensions (\cite[Definition~1.1.2.4.]{L1}). As being a simplicial set, it has the \emph{face maps} $d_k\co \C_{n+1}\to\C_n$ $(k=0,1,\ldots,n+1)$, and the \emph{degeneracy maps} $s_k\co\C_n\to\C_{n+1}$ $(k=0,1,\ldots,n)$ for any $n\ge0$, where each $\C_n$ denotes the set of its $n$-simplices. In particular $0$-simplices and $1$-simplices in $\C$ are called \emph{objects} and \emph{morphisms}. For any pair of objects $A,B\in\C_0$, we write $f\in\C_1(A,B)$ to indicate that $f$ is a morphism satisfying $d_1(f)=A$ and $d_0(f)=B$.
For an object $A\in\C_0$, we abbreviate $s_0(A)\in\C_1(A,A)$ to $1_A$ for simplicity.
The homotopy category of $\C$ will be denoted by $h\C$. For a morphism $a\in\C_1(A,B)$, its image in the homotopy category will be denoted by $\ovl{a}\in(h\C)(A,B)$. We use the symbol $\id_A=\ovl{1}_A$ exclusively for the identity morphisms in the homotopy category.

\subsection{Squares, rectangles and cubes}
As in \cite[Section~4.4.2]{L1}, a \emph{square} in $\C$ is a map $\Sf\co\Delta^1\ti\Delta^1\to\C$, namely an object in the quasi-category of functors $\Fun(\Delta^1\ti\Delta^1,\C)$ (see \cite[Notation~1.2.7.2, Proposition~1.2.7.3]{L1}). We will often denote $\Sf$ as a diagram in $\C$
\[
\xy
(-8,7)*+{\Sf(0,0)}="0";
(8,7)*+{\Sf(0,1)}="2";
(-8,-7)*+{\Sf(1,0)}="4";
(8,-7)*+{\Sf(1,1)}="6";
{\ar^{} "0";"2"};
{\ar_{} "0";"4"};
{\ar^{} "2";"6"};
{\ar_{} "4";"6"};
\endxy,
\]
or by labelling some of its simplices
\begin{equation}\label{Sq}
\xy
(-7,7)*+{A}="0";
(7,7)*+{B}="2";
(-7,-7)*+{D}="4";
(7,-7)*+{C}="6";
{\ar^{x} "0";"2"};
{\ar_{i} "0";"4"};
{\ar|*+{_z} "0";"6"};
{\ar^{y} "2";"6"};
{\ar_{j} "4";"6"};
(3,7)*+{}="00";
(7,3)*+{}="01";
{\ar@/_0.2pc/@{-}_{^{\xi}} "00";"01"};
(-3,-7)*+{}="10";
(-7,-3)*+{}="11";
{\ar@/_0.2pc/@{-}_{_{\eta}} "10";"11"};
\endxy
\ \ \text{or}\ \ 
\xy
(-7,7)*+{A}="0";
(7,7)*+{B}="2";
(-7,-7)*+{D}="4";
(7,-7)*+{C}="6";
{\ar^{x} "0";"2"};
{\ar_{i} "0";"4"};
{\ar_{z} "0";"6"};
{\ar^{y} "2";"6"};
{\ar_{j} "4";"6"};
\endxy
\ \ \text{or just}\ \ 
\xy
(-7,7)*+{A}="0";
(7,7)*+{B}="2";
(-7,-7)*+{D}="4";
(7,-7)*+{C}="6";
{\ar^{x} "0";"2"};
{\ar_{i} "0";"4"};
{\ar^{y} "2";"6"};
{\ar_{j} "4";"6"};
\endxy.
\end{equation}
It can be also identified with a pair of $2$-simplices $\xi$ and $\eta$ satisfying $d_1(\xi)=d_1(\eta)$, which correspond to
\[ \xi=\Sf|_{(\{0\}\ti\Delta^1)\ast\{(1,1)\}}\quad\text{and}
\quad
\eta=\Sf|_{\{(0,0)\}\ast(\{1\}\ti\Delta^1)}. \]
In this article we will also express a square as above by $\Sf=\SQ{\eta}{\xi}$. For a square $\Sf=\SQ{\eta}{\xi}$, its transpose $\Sf^t=\SQ{\xi}{\eta}$ is given by the below.
\[
\xy
(-7,7)*+{A}="0";
(7,7)*+{D}="2";
(-7,-7)*+{B}="4";
(7,-7)*+{C}="6";
{\ar^{i} "0";"2"};
{\ar_{x} "0";"4"};
{\ar|*+{_z} "0";"6"};
{\ar^{j} "2";"6"};
{\ar_{y} "4";"6"};
(3,7)*+{}="00";
(7,3)*+{}="01";
{\ar@/_0.2pc/@{-}_{^{\eta}} "00";"01"};
(-3,-7)*+{}="10";
(-7,-3)*+{}="11";
{\ar@/_0.2pc/@{-}_{_{\xi}} "10";"11"};
\endxy
\]

If we fix a map $F\co\Lambda^2_0\to\C$ denoted by a diagram
\begin{equation}\label{FixBD}
\xy
(-6,6)*+{A}="0";
(6,6)*+{B}="2";
(-6,-6)*+{D}="4";
{\ar^{x} "0";"2"};
{\ar_{i} "0";"4"};
%
\endxy
\end{equation}
in $\C$, then squares 
\[
\Sf=\ 
\xy
(-7,7)*+{A}="0";
(7,7)*+{B}="2";
(-7,-7)*+{D}="4";
(7,-7)*+{C}="6";
{\ar^{x} "0";"2"};
{\ar_{i} "0";"4"};
{\ar|*+{_z} "0";"6"};
{\ar^{y} "2";"6"};
{\ar_{j} "4";"6"};
(3,7)*+{}="00";
(7,3)*+{}="01";
{\ar@/_0.2pc/@{-}_{^{\xi}} "00";"01"};
(-3,-7)*+{}="10";
(-7,-3)*+{}="11";
{\ar@/_0.2pc/@{-}_{_{\eta}} "10";"11"};
\endxy
\quad\text{and}\quad
\Sf\ppr=\ 
\xy
(-7,7)*+{A}="0";
(7,7)*+{B}="2";
(-7,-7)*+{D}="4";
(7,-7)*+{C\ppr}="6";
{\ar^{x} "0";"2"};
{\ar_{i} "0";"4"};
{\ar|*+{_{z\ppr}} "0";"6"};
{\ar^{y\ppr} "2";"6"};
{\ar_{j\ppr} "4";"6"};
(3,7)*+{}="00";
(7,3)*+{}="01";
{\ar@/_0.2pc/@{-}_(0.4){^{\xi\ppr}} "00";"01"};
(-3,-7)*+{}="10";
(-7,-3)*+{}="11";
{\ar@/_0.2pc/@{-}_(0.4){_{\eta\ppr}} "10";"11"};
\endxy
\]
can be regarded as objects in the quasi-category under $(\ref{FixBD})$ (\cite[Remark 1.2.9.5.]{L1}).
In this undercategory, a morphism from $\Sf$ to $\Sf\ppr$ is nothing but a pair of $3$-simplices $\vp,\psi$ in $\C$ satisfying
\[ d_1(\vp)=d_1(\psi),d_2(\vp)=\eta\ppr,d_2(\psi)=\xi\ppr,d_3(\vp)=\eta,d_3(\psi)=\xi. \]
In this article, we denote this morphism by $\roundup{\vp}{\psi}\co\Sf\to\Sf\ppr$. We refer to such a pair simply as a morphism in \lq\lq the" undercategory, when $(\ref{FixBD})$ is clear from the context.

A \emph{rectangle} in this article means a map $\Rf\co\Delta^1\ti\Delta^2\to\C$ which will be denoted as a diagram
\[
\xy
(-16,7)*+{\Rf(0,0)}="0";
(0,7)*+{\Rf(0,1)}="2";
(16,7)*+{\Rf(0,2)}="4";
(-16,-7)*+{\Rf(1,0)}="10";
(0,-7)*+{\Rf(1,1)}="12";
(16,-7)*+{\Rf(1,2)}="14";
{\ar^{} "0";"2"};
{\ar_{} "2";"4"};
{\ar^{} "0";"10"};
{\ar_{} "2";"12"};
{\ar_{} "4";"14"};
{\ar^{} "10";"12"};
{\ar_{} "12";"14"};
\endxy
\]
often by labelling some of its simplices similarly as for squares.
It can be identified with a triplet of $3$-simplices $\Xcal,\Ycal,\Zcal$ satisfying $d_1(\Xcal)=d_1(\Ycal)$ and $d_2(\Ycal)=d_2(\Zcal)$, 
which correspond to
\[ \Xcal=\Rf|_{\{(0,0)\}\ast(\{1\}\ti\Delta^2)},\ \ \Ycal=\Rf|_{(\{0\}\ti [0,1])\ast(\{1\}\ti [1,2])},\ \ \Zcal=\Rf|_{(\{0\}\ti\Delta^2)\ast\{(1,2)\}}. \]
In this article, we express such a rectangle by $\Rf=\RT{\Xcal}{\Ycal}{\Zcal}$. The left, right, and outer squares in $\Rf$ will be denoted by $\Rf_{\mathrm{left}}=\Rf|_{\Delta^1\ti [0,1]}$, $\Rf_{\mathrm{right}}=\Rf|_{\Delta^1\ti [1,2]}$ and $\Rf_{\mathrm{out}}=\Rf|_{\Delta^1\ti \Lambda^2_1}$, respectively.
If $\Rf_{\mathrm{left}}$ is a push-out, then $\Rf_{\mathrm{out}}$ is a push-out if and only if $\Rf_{\mathrm{right}}$ is a push-out (\cite[Lemma~4.4.2.1]{L1}). Dually for pull-backs.

Remark that a square $(\ref{Sq})$ can be thought of as a morphism in the quasi-category $\Fun(\Delta^1,\C)$ from $A\ov{x}{\lra}B$ to $D\ov{j}{\lra}C$. In this view, a consecutive pair of squares can be filled into a $2$-simplex in $\Fun(\Delta^1,\C)$. Slightly more strictly, we may also designate the 2-simplices in $\C$ used as guides for this pasting operation, as follows.
\begin{prop}\label{PropComposeSquares}
Let
\[
\Sf=\ 
\xy
(-7,7)*+{A}="0";
(7,7)*+{B}="2";
(-7,-7)*+{A\ppr}="4";
(7,-7)*+{B\ppr}="6";
{\ar^{x} "0";"2"};
{\ar_{a} "0";"4"};
{\ar|*+{_z} "0";"6"};
{\ar^{b} "2";"6"};
{\ar_{x\ppr} "4";"6"};
(3,7)*+{}="00";
(7,3)*+{}="01";
{\ar@/_0.2pc/@{-}_{^{\xi}} "00";"01"};
(-3,-7)*+{}="10";
(-7,-3)*+{}="11";
{\ar@/_0.2pc/@{-}_{_{\eta}} "10";"11"};
\endxy
\quad\text{and}\quad
\Sf\ppr=\ 
\xy
(-7,7)*+{A\ppr}="0";
(7,7)*+{B\ppr}="2";
(-7,-7)*+{A\pprr}="4";
(7,-7)*+{B\pprr}="6";
{\ar^{x\ppr} "0";"2"};
{\ar_{a\ppr} "0";"4"};
{\ar|*+{_{z\ppr}} "0";"6"};
{\ar^{b\ppr} "2";"6"};
{\ar_{x\pprr} "4";"6"};
(3,7)*+{}="00";
(7,3)*+{}="01";
{\ar@/_0.2pc/@{-}_(0.4){^{\xi\ppr}} "00";"01"};
(-3,-7)*+{}="10";
(-7,-3)*+{}="11";
{\ar@/_0.2pc/@{-}_(0.4){_{\eta\ppr}} "10";"11"};
\endxy
\]
be any pair of squares such that $\Sf|_{\{1\}\ti\Delta^1}=\Sf\ppr|_{\{0\}\ti\Delta^1}=x\ppr$, and let
\[
\TwoSP{A}{A\ppr}{A\pprr}{a}{a\ppr}{a\pprr}{\al}
\ , \ 
\TwoSP{B}{B\ppr}{B\pprr}{b}{b\ppr}{b\pprr}{\be}
\]
be arbitrarily taken pair of $2$-simplices. Then there exists a rectangle
\[
\Rf=
\xy
(-16,7)*+{A}="0";
(0,7)*+{A\ppr}="2";
(16,7)*+{A\pprr}="4";
(-16,-7)*+{B}="10";
(0,-7)*+{B\ppr}="12";
(16,-7)*+{B\pprr}="14";
{\ar_{a} "0";"2"};
{\ar_{a\ppr} "2";"4"};
{\ar@/^1.0pc/^{a\pprr} "0";"4"};
{\ar_{x} "0";"10"};
{\ar_{x\ppr} "2";"12"};
{\ar^{x\pprr} "4";"14"};
{\ar^{b} "10";"12"};
{\ar^{b\ppr} "12";"14"};
{\ar@/_1.0pc/_{b\pprr} "10";"14"};
\endxy
\]
such that $\Rf_{\mathrm{left}}=\Sf^t$, $\Rf_{\mathrm{right}}=\Sf^{\prime t}$ and $\Rf|_{\{0\}\ti\Delta^2}=\al$, $\Rf|_{\{1\}\ti\Delta^2}=\be$.

Moreover, if $\Sf=\SQ{s_0(x\ppr)}{\xi}$ and $\al=s_0(a\ppr)$ (hence consequently $A\ppr=A, a=1_A, z=x\ppr, a\pprr=a\ppr$), then $\Rf$ can be chosen to be of the form $\Rf=\RT{\Xcal}{s_0(\xi\ppr)}{s_0(\eta\ppr)}$.
\end{prop}
\begin{proof}
This is just rephrasing the fact that $(\Delta^1\ti\Lambda^2_1)\cup(\partial\Delta^1\ti\Delta^2)\hookrightarrow\Delta^1\ti\Delta^2$ is an inner anodyne map. (\cite{J}. See also \cite[Proposition~2.3.2.1]{L1}, \cite[Proposition~3.2.3]{C}).
We also remark that as for the latter part, we only have to take a $3$-simplex
\[
\Xcal=\ 
\xy
(-2,11)*+{A}="0";
(-12,-2)*+{B}="2";
(10,0)*+{B\pprr}="4";
(2,-12)*+{B\ppr}="6";
{\ar_{x} "0";"2"};
{\ar^{z\ppr} "0";"4"};
{\ar_(0.3){z} "0";"6"};
{\ar_(0.35){b\pprr}|!{"0";"6"}\hole "2";"4"};
{\ar_{b} "2";"6"};
{\ar_{b\ppr} "6";"4"};
\endxy
\]
such that $d_0(\Xcal)=\be, d_1(\Xcal)=\xi\ppr, d_3(\Xcal)=\xi$.
\end{proof}

If we are given a rectangle $\Rf\co\Delta^1\ti\Delta^2\to\C$ and a morphism $f\in\C_1(\Rf((1,2)),C)$ to some $C\in\C_0$, then we can extend them to obtain a map $F\co(\Delta^1\ti\Delta^2)^{\vartriangleright}\to\C$. More strictly, we may designate some of the simplices involved, as below.
This ad hoc lemma will be used later in the proof of Lemma~\ref{LemExSeqPO}.
\begin{lem}\label{LemRectProl}
Let
\[
\Rf=\RT{\psi}{\vp}{\thh}=\ 
\xy
(-16,7)*+{A}="0";
(0,7)*+{A\ppr}="2";
(16,7)*+{A\pprr}="4";
(-16,-7)*+{B}="10";
(0,-7)*+{B\ppr}="12";
(16,-7)*+{B\pprr}="14";
{\ar_{a} "0";"2"};
{\ar_{a\ppr} "2";"4"};
{\ar@/^1.0pc/^{a\pprr} "0";"4"};
{\ar_{x} "0";"10"};
{\ar_{x\ppr} "2";"12"};
{\ar^{x\pprr} "4";"14"};
{\ar^{b} "10";"12"};
{\ar^{b\ppr} "12";"14"};
{\ar@/_1.0pc/_{b\pprr} "10";"14"};
\endxy
\]
be any rectangle, let
\[ \roundup{\ze}{\ze\ppr}\co\Rf_{\mathrm{out}}\to
\Sf=\ 
\xy
(-7,7)*+{A}="0";
(7,7)*+{A\pprr}="2";
(-7,-7)*+{B}="4";
(7,-7)*+{C}="6";
{\ar^{a\pprr} "0";"2"};
{\ar_{x} "0";"4"};
{\ar^{} "2";"6"};
{\ar_{} "4";"6"};
\endxy
\]
be any morphism in the undercategory, and let
\[
\wp=
\xy
(-2,11)*+{A}="0";
(-12,-2)*+{A\ppr}="2";
(10,0)*+{C}="4";
(2,-12)*+{A\pprr}="6";
{\ar_{a} "0";"2"};
{\ar^{} "0";"4"};
{\ar_(0.3){a\pprr} "0";"6"};
{\ar_(0.35){}|!{"0";"6"}\hole "2";"4"};
{\ar_{a\ppr} "2";"6"};
{\ar_{} "6";"4"};
\endxy
\]
be any $3$-simplex in $\C$ satisfying $d_1(\wp)=d_2(\ze\ppr)$ and $d_3(\wp)=d_3(\thh)$. Then we can extend this to a map $(\Delta^1\ti\Delta^2)^{\vartriangleright}\to\C$. Namely, there is a triplet of $4$-simplices
\[
\Psi=\Penta{A}{B}{B\ppr}{B\pprr}{C}{x}{b}{b\ppr}{},\ 
\Phi=\Penta{A}{A\ppr}{B\ppr}{B\pprr}{C}{a}{x\ppr}{b\ppr}{},\ 
\Theta=\Penta{A}{A\ppr}{A\pprr}{B\pprr}{C}{a}{a\ppr}{x\pprr}{}
\]
such that $d_1(\Psi)=d_1(\Phi)$, $d_2(\Phi)=d_2(\Thh)$, $\roundup{\ze}{\ze\ppr}=\roundup{d_2(\Psi)}{d_1(\Thh)}$, $d_3(\Thh)=\wp$, and $\RT{\psi}{\vp}{\thh}=\RT{d_4(\Phi)}{d_4(\Psi)}{d_4(\Thh)}$.
\end{lem}
\begin{proof}
Take a $2$-simplex
\[ \TwoSP{B\ppr}{B\pprr}{C}{b\ppr}{}{}{\varpi} \]
such that $d_0(\varpi)=d_0d_1(\ze)=d_0d_1(\ze\ppr)$ and $d_2(\varpi)=b\ppr$. Take $3$-simplices
\[
\al=\ThreeSP{A\ppr}{A\pprr}{B\pprr}{C},\ 
\be=\ThreeSP{A}{B\ppr}{B\pprr}{C},\ 
\gam=\ThreeSP{A\ppr}{B\ppr}{B\pprr}{C},\ 
\del=\ThreeSP{B}{B\ppr}{B\pprr}{C}
\]
such that
\begin{eqnarray*}
&d_0(\al)=d_0(\ze\ppr),\ d_2(\al)=d_0(\wp),\ d_3(\al)=d_0(\thh),&\\
&d_0(\be)=\varpi,\ d_1(\be)=d_1(\ze),\ d_3(\be)=d_1(\vp)=d_1(\psi),&\\
&d_0(\gam)=\varpi,\ d_1(\gam)=d_1(\al),\ d_3(\gam)=d_0(\vp),&\\
&d_0(\del)=\varpi,\ d_1(\del)=d_0(\ze),\ d_3(\del)=d_0(\psi).&
\end{eqnarray*}
Then we can take $\Theta,\Phi,\Psi$ so that
\begin{eqnarray*}
&d_0(\Theta)=\al,\ d_1(\Theta)=\ze\ppr,\ d_3(\Theta)=\wp,\ d_4(\Theta)=\thh,&\\
&d_0(\Phi)=\gam,\ d_1(\Phi)=\be,\ d_2(\Phi)=d_2(\Theta),\ d_4(\Phi)=\vp,&\\
&d_0(\Psi)=\del,\ d_1(\Psi)=\be,\ d_2(\Psi)=\ze,\ d_4(\Psi)=\psi&
\end{eqnarray*}
hold, as desired.
\end{proof}

A \emph{cube} in $\C$ is a map $\Cf\co\Delta^1\ti\Delta^1\ti\Delta^1\to\C$, 
namely an object in $\Fun(\Delta^1\ti\Delta^1\ti\Delta^1,\C)$, which we denote by a diagram 
\[
\xy
(-11,15)*+{\Cf(0,0,0)}="0";
(11,15)*+{\Cf(0,1,0)}="2";
(0,5)*+{\Cf(1,0,0)}="4";
(22,5)*+{\Cf(1,1,0)}="6";
(-11,-5)*+{\Cf(0,0,1)}="10";
(11,-5)*+{\Cf(0,1,1)}="12";
(0,-15)*+{\Cf(1,0,1)}="14";
(22,-15)*+{\Cf(1,1,1)}="16";
{\ar^{} "0";"2"};
{\ar_{} "0";"4"};
{\ar^{} "2";"6"};
{\ar_{} "4";"6"};
{\ar^{} "0";"10"};
{\ar_{}|!{(9,5);(13,5)}\hole "2";"12"};
{\ar^{} "4";"14"};
{\ar_{} "6";"16"};
{\ar^{}|!{(0,-3);(0,-7)}\hole "10";"12"};
{\ar_{} "10";"14"};
{\ar^{} "12";"16"};
{\ar_{} "14";"16"};
\endxy
\]
in $\C$ or by labelling some of its simplices for example as follows,
\begin{equation}\label{Cb}
\Cf=\ 
\xy
(-11,12)*+{A}="0";
(11,12)*+{B}="2";
(0,2)*+{D}="4";
(22,2)*+{C}="6";
(-11,-8)*+{A\ppr}="10";
(11,-8)*+{B\ppr}="12";
(0,-18)*+{D\ppr}="14";
(22,-18)*+{C\ppr}="16";
{\ar^{x} "0";"2"};
{\ar_{i} "0";"4"};
{\ar^{y} "2";"6"};
{\ar_(0.7){j} "4";"6"};
{\ar_(0.4){z} "0";"6"};
{\ar_{a} "0";"10"};
{\ar^(0.7){b}|!{(9,2);(13,2)}\hole|!{(9,6.2);(13,4.2)}\hole "2";"12"};
{\ar_(0.3){d} "4";"14"};
{\ar^{c} "6";"16"};
{\ar^(0.7){x\ppr}|!{(0,-6);(0,-10)}\hole "10";"12"};
{\ar_{i\ppr} "10";"14"};
{\ar^{y\ppr} "12";"16"};
{\ar_{j\ppr} "14";"16"};
{\ar_{z\ppr} "10";"16"};
\endxy,
\end{equation}
similarly as for squares. We may identify it with a pair of rectangles
\[
\xy
(-28,4)*+{\Rf_{\ff}=}="-";
(-16,7)*+{A}="0";
(0,7)*+{B}="2";
(16,7)*+{C}="4";
(-16,-7)*+{A\ppr}="10";
(0,-7)*+{B\ppr}="12";
(16,-7)*+{C\ppr}="14";
{\ar_{x} "0";"2"};
{\ar_{y} "2";"4"};
{\ar@/^1.0pc/^{z} "0";"4"};
{\ar_{a} "0";"10"};
{\ar_{b} "2";"12"};
{\ar^{c} "4";"14"};
{\ar^{x\ppr} "10";"12"};
{\ar^{y\ppr} "12";"14"};
{\ar@/_1.0pc/_{z\ppr} "10";"14"};
\endxy
\quad\text{and}\quad
\xy
(-28,4)*+{\Rf_{\bb}=}="-";
(-16,7)*+{A}="0";
(0,7)*+{D}="2";
(16,7)*+{C}="4";
(-16,-7)*+{A\ppr}="10";
(0,-7)*+{D\ppr}="12";
(16,-7)*+{C\ppr}="14";
{\ar_{i} "0";"2"};
{\ar_{j} "2";"4"};
{\ar@/^1.0pc/^{z} "0";"4"};
{\ar_{a} "0";"10"};
{\ar_{d} "2";"12"};
{\ar^{c} "4";"14"};
{\ar^{i\ppr} "10";"12"};
{\ar^{j\ppr} "12";"14"};
{\ar@/_1.0pc/_{z\ppr} "10";"14"};
\endxy
\]
sharing the outer square in common.
Each of these rectangles is expressed as $\Rf_{\ff}=\RT{\Xcal_{\ff}}{\Ycal_{\ff}}{\Zcal_{\ff}}$ and $\Rf_{\bb}=\RT{\Xcal_{\bb}}{\Ycal_{\bb}}{\Zcal_{\bb}}$. In summary, a cube $\Cf$ can be expressed by a $6$-tuple of $3$-simplices in $\C$
\begin{equation}\label{Exp_Cube}
\Cf=\CB{\Xcal_{\ff}}{\Ycal_{\ff}}{\Zcal_{\ff}}{\Xcal_{\bb}}{\Ycal_{\bb}}{\Zcal_{\bb}}
\end{equation}
satisfying $d_1(\Xcal_{\ff})=d_1(\Ycal_{\ff}), d_2(\Ycal_{\ff})=d_2(\Zcal_{\ff}), d_2(\Xcal_{\ff})=d_2(\Xcal_{\bb}), d_1(\Zcal_{\ff})=d_1(\Zcal_{\bb})$, $d_1(\Xcal_{\bb})=d_1(\Ycal_{\bb}), d_2(\Ycal_{\bb})=d_2(\Zcal_{\bb})$.

\begin{rem}
From a cube $\Cf$ as above, we may obtain cubes
\[
\CB{\Ycal_{\bb}}{\Xcal_{\bb}}{\Xcal_{\ff}}{\Zcal_{\bb}}{\Zcal_{\ff}}{\Ycal_{\ff}}=\ 
\xy
(-8,9)*+{A}="0";
(8,9)*+{A\ppr}="2";
(1,3)*+{B}="4";
(17,3)*+{B\ppr}="6";
(-8,-7)*+{D}="10";
(8,-7)*+{D\ppr}="12";
(1,-13)*+{C}="14";
(17,-13)*+{C\ppr}="16";
{\ar^{a} "0";"2"};
{\ar_{x} "0";"4"};
{\ar^{x\ppr} "2";"6"};
{\ar_(0.3){b} "4";"6"};
%
{\ar_{i} "0";"10"};
{\ar^(0.7){i\ppr}|!{(9,3);(13,3)}\hole "2";"12"};
{\ar_(0.3){y} "4";"14"};
{\ar^{y\ppr} "6";"16"};
{\ar^(0.3){d}|!{(1,-3);(1,-7)}\hole "10";"12"};
{\ar_{j} "10";"14"};
{\ar^{j\ppr} "12";"16"};
{\ar_{c} "14";"16"};
%
\endxy
\ ,\ 
\CB{\Zcal_{\ff}}{\Zcal_{\bb}}{\Ycal_{\bb}}{\Ycal_{\ff}}{\Xcal_{\ff}}{\Xcal_{\bb}}=\ 
\xy
(-8,9)*+{A}="0";
(8,9)*+{D}="2";
(1,3)*+{A\ppr}="4";
(17,3)*+{D\ppr}="6";
(-8,-7)*+{B}="10";
(8,-7)*+{C}="12";
(1,-13)*+{B\ppr}="14";
(17,-13)*+{C\ppr}="16";
{\ar^{i} "0";"2"};
{\ar_{a} "0";"4"};
{\ar^{d} "2";"6"};
{\ar_(0.3){i\ppr} "4";"6"};
%
{\ar_{x} "0";"10"};
{\ar^(0.7){j}|!{(9,3);(13,3)}\hole "2";"12"};
{\ar_(0.3){x\ppr} "4";"14"};
{\ar^{j\ppr} "6";"16"};
{\ar^(0.3){y}|!{(1,-3);(1,-7)}\hole "10";"12"};
{\ar_{b} "10";"14"};
{\ar^{c} "12";"16"};
{\ar_{y\ppr} "14";"16"};
%
\endxy \]
by transpositions, and also
$\CB{\Xcal_{\bb}}{\Ycal_{\bb}}{\Zcal_{\bb}}{\Xcal_{\ff}}{\Ycal_{\ff}}{\Zcal_{\ff}}$,
$\CB{\Ycal_{\ff}}{\Xcal_{\ff}}{\Xcal_{\bb}}{\Zcal_{\ff}}{\Zcal_{\bb}}{\Ycal_{\bb}}$,
$\CB{\Zcal_{\bb}}{\Zcal_{\ff}}{\Ycal_{\ff}}{\Ycal_{\bb}}{\Xcal_{\bb}}{\Xcal_{\ff}}$. 
\end{rem}

We often regard a cube $(\ref{Cb})$ as a morphism from $\Cf|_{\Delta^1\ti\Delta^1\ti\{0\}}$ to $\Cf|_{\Delta^1\ti\Delta^1\ti\{1\}}$ in $\Fun(\Delta^1\ti\Delta^1,\C)$. In this view, we have the following.
\begin{ex}\label{ExTrivCubes}
Let
\[
\Sf=\ 
\xy
(-7,7)*+{A}="0";
(7,7)*+{B}="2";
(-7,-7)*+{D}="4";
(7,-7)*+{C}="6";
{\ar^{x} "0";"2"};
{\ar_{i} "0";"4"};
{\ar|*+{_z} "0";"6"};
{\ar^{y} "2";"6"};
{\ar_{j} "4";"6"};
(3,7)*+{}="00";
(7,3)*+{}="01";
{\ar@/_0.2pc/@{-}_{^{\xi}} "00";"01"};
(-3,-7)*+{}="10";
(-7,-3)*+{}="11";
{\ar@/_0.2pc/@{-}_{_{\eta}} "10";"11"};

\endxy
\]
be any square.
\begin{enumerate}
\item There is a cube consisting of degenerated $3$-simplices
\[
\If_{\Sf}=\CB{s_0(\xi)}{s_1(\xi)}{s_2(\xi)}{s_0(\eta)}{s_1(\eta)}{s_2(\eta)}
\]
such that $\If_{\Sf}|_{\Delta^1\ti\Delta^1\ti\{0\}}=\If_{\Sf}|_{\Delta^1\ti\Delta^1\ti\{1\}}=\Sf$.
This is nothing but $1_{\Sf}=s_0(\Sf)$ in $\Fun(\Delta^1\ti\Delta^1,\C)$.
\item Suppose that $\C$ is pointed (\cite[Definition~2.1]{L2}). Then any square
\[
\Of=\ 
\xy
(-7,7)*+{O_A}="0";
(7,7)*+{O_B}="2";
(-7,-7)*+{O_D}="4";
(7,-7)*+{O_C}="6";
{\ar^{x_o} "0";"2"};
{\ar_{i_o} "0";"4"};
{\ar|*+{_{z_o}} "0";"6"};
{\ar^{y_o} "2";"6"};
{\ar_{j_o} "4";"6"};
\endxy
\]
in which $O_A,O_B,O_C,O_D$ are zero objects in $\C$, gives a zero object in $\Fun(\Delta^1\ti\Delta^1,\C)$. In particular, there exists a cube $\Cf$ such that $\Cf|_{\Delta^1\ti\Delta^1\ti\{0\}}=\Sf$ and $\Cf|_{\Delta^1\ti\Delta^1\ti\{1\}}=\Of$. Indeed, by using the definition of zero objects, such a cube can be taken compatibly with arbitrarily given $5$-tuple of squares
\[
\xy
(-5,5)*+{A}="0";
(5,5)*+{B}="2";
(-5,-5)*+{O_A}="4";
(5,-5)*+{O_B}="6";
{\ar^{x} "0";"2"};
{\ar_{a} "0";"4"};
{\ar^{b} "2";"6"};
{\ar_{x_o} "4";"6"};
\endxy\ ,\ 
\xy
(-5,5)*+{A}="0";
(5,5)*+{D}="2";
(-5,-5)*+{O_A}="4";
(5,-5)*+{O_D}="6";
{\ar^{i} "0";"2"};
{\ar_{a} "0";"4"};
{\ar^{d} "2";"6"};
{\ar_{i_o} "4";"6"};
\endxy\ ,\ 
\xy
(-5,5)*+{B}="0";
(5,5)*+{C}="2";
(-5,-5)*+{O_B}="4";
(5,-5)*+{O_C}="6";
{\ar^{y} "0";"2"};
{\ar_{b} "0";"4"};
{\ar^{c} "2";"6"};
{\ar_{y_o} "4";"6"};
\endxy\ ,\ 
\xy
(-5,5)*+{D}="0";
(5,5)*+{C}="2";
(-5,-5)*+{O_D}="4";
(5,-5)*+{O_C}="6";
{\ar^{j} "0";"2"};
{\ar_{c} "0";"4"};
{\ar^{d} "2";"6"};
{\ar_{j_o} "4";"6"};
\endxy\ ,\ 
\xy
(-5,5)*+{A}="0";
(5,5)*+{C}="2";
(-5,-5)*+{O_A}="4";
(5,-5)*+{O_C}="6";
{\ar^{z} "0";"2"};
{\ar_{a} "0";"4"};
{\ar^{c} "2";"6"};
{\ar_{z_o} "4";"6"};
\endxy
\]
sharing the morphisms with the same labels in common.
Similarly for the existence of a cube $\Cf\ppr$ with $\Cf\ppr|_{\Delta^1\ti\Delta^1\ti\{0\}}=\Of$ and $\Cf\ppr|_{\Delta^1\ti\Delta^1\ti\{1\}}=\Sf$.
\end{enumerate}
\end{ex}

Similarly as in Proposition~\ref{PropComposeSquares}, any consecutive pair of cubes can be filled into a $2$-simplex in $\Fun(\Delta^1\ti\Delta^1,\C)$. More strictly, the following holds.
\begin{prop}\label{PropComposeCubes}
Let
\[
\Cf=\ 
\xy
(-8,9)*+{A}="0";
(8,9)*+{B}="2";
(1,3)*+{D}="4";
(17,3)*+{C}="6";
(-8,-7)*+{A\ppr}="10";
(8,-7)*+{B\ppr}="12";
(1,-13)*+{D\ppr}="14";
(17,-13)*+{C\ppr}="16";
{\ar^{x} "0";"2"};
{\ar_{i} "0";"4"};
{\ar^{y} "2";"6"};
{\ar_(0.3){j} "4";"6"};
%
{\ar_{a} "0";"10"};
{\ar^(0.7){b}|!{(9,3);(13,3)}\hole "2";"12"};
{\ar_(0.3){d} "4";"14"};
{\ar^{c} "6";"16"};
{\ar^(0.3){x\ppr}|!{(1,-3);(1,-7)}\hole "10";"12"};
{\ar_{i\ppr} "10";"14"};
{\ar^{y\ppr} "12";"16"};
{\ar_{j\ppr} "14";"16"};
%
\endxy
\quad,\quad
\Cf\ppr=\ 
\xy
(-8,9)*+{A\ppr}="0";
(8,9)*+{B\ppr}="2";
(1,3)*+{D\ppr}="4";
(17,3)*+{C\ppr}="6";
(-8,-7)*+{A\pprr}="10";
(8,-7)*+{B\pprr}="12";
(1,-13)*+{D\pprr}="14";
(17,-13)*+{C\pprr}="16";
{\ar^{x\ppr} "0";"2"};
{\ar_{i\ppr} "0";"4"};
{\ar^{y\ppr} "2";"6"};
{\ar_(0.3){j\ppr} "4";"6"};
%
{\ar_{a\ppr} "0";"10"};
{\ar^(0.7){b\ppr}|!{(9,3);(13,3)}\hole "2";"12"};
{\ar_(0.3){d\ppr} "4";"14"};
{\ar^{c\ppr} "6";"16"};
{\ar^(0.3){x\pprr}|!{(1,-3);(1,-7)}\hole "10";"12"};
{\ar_{i\pprr} "10";"14"};
{\ar^{y\pprr} "12";"16"};
{\ar_{j\pprr} "14";"16"};
%
\endxy
\]
be a pair of cubes satisfying $\Cf|_{\Delta^1\ti\Delta^1\ti\{1\}}=\Cf\ppr|_{\Delta^1\ti\Delta^1\ti\{0\}}$, and let
\[
\TwoSP{A}{A\ppr}{A\pprr}{a}{a\ppr}{a\pprr}{\al},\ 
\TwoSP{B}{B\ppr}{B\pprr}{b}{b\ppr}{b\pprr}{\be},\ 
\TwoSP{C}{C\ppr}{C\pprr}{c}{c\ppr}{c\pprr}{\gam},\ 
\TwoSP{D}{D\ppr}{D\pprr}{d}{d\ppr}{d\pprr}{\del}
\]
be arbitrarily taken $2$-simplices.
Then there exists $F\co\Delta^1\ti\Delta^1\ti\Delta^2\to\C$ such that
$F|_{\Delta^1\ti\Delta^1\ti [0,1]}=\Cf, F|_{\Delta^1\ti\Delta^1\ti [1,2]}=\Cf\ppr$ and
\[ F|_{\{(0,0)\}\ti\Delta^2}=\al, F|_{\{(0,1)\}\ti\Delta^2}=\be, F|_{\{(1,1)\}\ti\Delta^2}=\gam, F|_{\{(1,0)\}\ti\Delta^2}=\del. \]
Especially we obtain a cube
\[
F|_{\Delta^1\ti\Delta^1\ti \Lambda^2_1}=\ 
\xy
(-8,9)*+{A}="0";
(8,9)*+{B}="2";
(1,3)*+{D}="4";
(17,3)*+{C}="6";
(-8,-7)*+{A\pprr}="10";
(8,-7)*+{B\pprr}="12";
(1,-13)*+{D\pprr}="14";
(17,-13)*+{C\pprr}="16";
{\ar^{x} "0";"2"};
{\ar_{i} "0";"4"};
{\ar^{y} "2";"6"};
{\ar_(0.3){j} "4";"6"};
%
{\ar_{a\pprr} "0";"10"};
{\ar^(0.7){b\pprr}|!{(9,3);(13,3)}\hole "2";"12"};
{\ar_(0.3){d\pprr} "4";"14"};
{\ar^{c\pprr} "6";"16"};
{\ar^(0.3){x\pprr}|!{(1,-3);(1,-7)}\hole "10";"12"};
{\ar_{i\pprr} "10";"14"};
{\ar^{y\pprr} "12";"16"};
{\ar_{j\pprr} "14";"16"};
%
\endxy
\]
compatible with the given data.
\end{prop}
\begin{proof}
Similarly as Proposition~\ref{PropComposeSquares}, this is just rephrasing the fact that $((\Delta^1\ti\Delta^1)\ti\Lambda^2_1)\cup(\{(0,0),(0,1),(1,0),(1,1)\}\ti\Delta^2)\hookrightarrow\Delta^1\ti\Delta^1\ti\Delta^2$ is an inner anodyne map (\cite{J}, \cite[Proposition~2.3.2.1]{L1}, \cite[Proposition~3.2.3]{C}).
\end{proof}

\begin{cor}\label{CorComposeCubes}
Let $\Sf, \Sf\ppr$ be any pair of squares such that $\Sf((0,0))=\Sf\ppr((0,0))=A$ and $\Sf((1,1))=\Sf\ppr((1,1))=C$. Assume that there is a cube $\Cf$ such that
\[ \Cf|_{\Delta^1\ti\Delta^1\ti\{0\}}=\Sf,\ \Cf|_{\Delta^1\ti\Delta^1\ti\{1\}}=\Sf\ppr \]
and that $a=\Cf|_{\{(0,0)\}\ti\Delta^1}$ and $c=\Cf|_{\{(1,1)\}\ti\Delta^1}$ satisfies $\ovl{a}=\id_A$ and $\ovl{c}=\id_C$.
Then there exists a cube $\Cf\ppr$ such that
\[ \Cf\ppr|_{\Delta^1\ti\Delta^1\ti\{0\}}=\Sf,\ \Cf\ppr|_{\Delta^1\ti\Delta^1\ti\{1\}}=\Sf\ppr  \]
and $\Cf\ppr|_{\{(0,0)\}\ti\Delta^1}=1_A$, $\Cf\ppr|_{\{(0,1)\}\ti\Delta^1}=\Cf|_{\{(0,1)\}\ti\Delta^1}$, $\Cf\ppr|_{\{(1,0)\}\ti\Delta^1}=\Cf|_{\{(1,0)\}\ti\Delta^1}$, $\Cf\ppr|_{\{(1,1)\}\ti\Delta^1}=1_C$.
\end{cor}
\begin{proof}
This follows from Proposition~\ref{PropComposeCubes} applied to the cubes $\Cf,\If_{\Sf\ppr}$ and $2$-simplices of the following form,
\[
\TwoSP{A}{A}{A}{a}{1_A}{1_A}{},\ 
\TwoSP{B}{B\ppr}{B\ppr}{b}{1_{B\ppr}}{b}{s_1(b)},\ 
\TwoSP{C}{C}{C}{c}{1_C}{1_C}{},\ 
\TwoSP{D}{D\ppr}{D\ppr}{d}{1_{D\ppr}}{d}{s_1(d)},\ 
\]
where we denote $\Cf|_{\{(0,1)\}\ti\Delta^1}$ and $\Cf|_{\{(1,0)\}\ti\Delta^1}$ by $B\ov{b}{\lra}B\ppr$ and $D\ov{d}{\lra}D\ppr$, respectively.
\end{proof}

\section{Exact sequences in exact quasi-categories}\label{Section_ExactSequences}
First we recall the definition of an exact quasi-category. For the detail, see \cite{B1},\cite{B2}. 
A quasi-category $\C$ is called \emph{additive} (\cite[Definition~2.2]{B1}) if it is pointed, has all finite products and finite coproducts, and moreover if the homotopy category $h\C$ is additive as an ordinary category. In particular, for any pair of objects $X_1,X_2\in h\C$, the unique morphism $\vfr\in (h\C)(X_1\am X_2,X_1\ti X_2)$ which satisfies
\[ p_{k\ppr}\ci \vfr\ci i_k\co X_k\to X_{k\ppr}=
\begin{cases}
\id& \text{if}\ k=k\ppr\\
0  & \text{if}\ k\ne k\ppr
\end{cases}
\]
should become an isomorphism. Here, $X_1\ov{p_1}{\lla}X_1\ti X_2\ov{p_2}{\lra}X_2$ and $X_1\ov{i_1}{\lra}X_1\am X_2\ov{i_2}{\lla}X_2$ are product and coproduct of $X_1,X_2$ in $h\C$, respectively.

Let $\C$ be a quasi-category, and let $\C_{\dag},\C^{\dag}$ be two subcategories of $\C$ containing all homotopy equivalences. A morphism in $\C_{\dag}$ is called \emph{ingressive}, and a morphism in $\C^{\dag}$ is called \emph{egressive}.
A square
\begin{equation}\label{AmbPB'}
\xy
(-6,6)*+{A}="0";
(6,6)*+{B}="2";
(-6,-6)*+{D}="4";
(6,-6)*+{C}="6";
{\ar^{x} "0";"2"};
{\ar_{i} "0";"4"};
{\ar^{y} "2";"6"};
{\ar_{j} "4";"6"};
\endxy
\end{equation}
is called an \emph{ambigressive pull-back} if it is a pull-back square in which $j$ is ingressive and $y$ is egressive.
Dually, $(\ref{AmbPB'})$ is an \emph{ambigressive push-out} if it is a push-out square in which $x$ is ingressive and $i$ is egressive.
\begin{dfn}\label{DefExactQuasiCat}(\cite[Definition~3.1]{B1})
 The triplet $(\C,\C_{\dag},\C^{\dag})$ is an \emph{exact quasi-category} ($=$ exact $\infty$-category) if it satisfies the following conditions.
\begin{itemize}
\item[{\rm (Ex0)}] $\C$ is additive.
\item[{\rm (Ex1)}] If $O\in \C_0$ is a zero object, then any morphism $O\to X$ in $\C$ is ingressive. Dually, any morphism $X\to O$ in $\C$ is egressive.
\item[{\rm (Ex2)}] Push-outs of ingressive morphisms exist and are ingressive. Dually, pull-backs of egressive morphisms exist and are egressive.
\item[{\rm (Ex3)}] A square in $\C$ is an ambigressive pull-back if and only if it is an ambigressive push-out.
\end{itemize}
\end{dfn}

In the rest of this article, let $(\C,\C_{\dag},\C^{\dag})$ be an exact quasi-category defined as above. We also simply say that $\C$ is an exact quasi-category.
\begin{dfn}\label{DefExSeq}(\cite[Definition~3.1]{B1})
$\ \ $
\begin{enumerate}
\item
An \emph{exact sequence} starting from $A$ and ending in $C$ is an ambigressive pull-back (hence an ambigressive push-out)
\begin{equation}\label{ExSeq}
\Sf=\ 
\xy
(-7,7)*+{A}="0";
(7,7)*+{B}="2";
(-7,-7)*+{O}="4";
(7,-7)*+{C}="6";
{\ar^{x} "0";"2"};
{\ar_{i} "0";"4"};
{\ar|*+{_z} "0";"6"};
{\ar^{y} "2";"6"};
{\ar_{j} "4";"6"};
(3,7)*+{}="00";
(7,3)*+{}="01";
{\ar@/_0.2pc/@{-}_{^{\xi}} "00";"01"};
(-3,-7)*+{}="10";
(-7,-3)*+{}="11";
{\ar@/_0.2pc/@{-}_{_{\eta}} "10";"11"};
\endxy
\end{equation}
in which $O$ is a zero object in $\C$. When we emphasize the end-objects $A$ and $C$, we write ${}_A\Sf_C$ in this article.
\end{enumerate}
\end{dfn}

\begin{ex}\label{ExTrivExSeq}
If a morphism $x\in\C_1(X,X)$ satisfies $\ovl{x}=\id_X$, then any square
\[
\xy
(-6,6)*+{X}="0";
(6,6)*+{X}="2";
(-6,-6)*+{O}="4";
(6,-6)*+{O\ppr}="6";
{\ar^{x} "0";"2"};
{\ar_{} "0";"4"};
{\ar^{} "2";"6"};
{\ar_{} "4";"6"};
\endxy,\ \ 
\Bigg(\text{similarly,}\ 
\xy
(-6,6)*+{O}="0";
(6,6)*+{X}="2";
(-6,-6)*+{O\ppr}="4";
(6,-6)*+{X}="6";
{\ar^{} "0";"2"};
{\ar_{} "0";"4"};
{\ar^{x} "2";"6"};
{\ar_{} "4";"6"};
\endxy
\Bigg)
\]
in which $O,O\ppr$ are zero objects, is an exact sequence.
\end{ex}

\subsection{Morphisms of exact sequences}
\begin{dfn}\label{DefMorphExSeq}
Let $\Esc\se\Fun(\Delta^1\ti\Delta^1,\C)$ be the full subcategory spanned by exact sequences. In particular, for a pair of exact sequences
\begin{equation}\label{TwoExSeq}
\Sf=\ 
\xy
(-7,7)*+{A}="0";
(7,7)*+{B}="2";
(-7,-7)*+{O}="4";
(7,-7)*+{C}="6";
{\ar^{x} "0";"2"};
{\ar_{i} "0";"4"};
{\ar|*+{_z} "0";"6"};
{\ar^{y} "2";"6"};
{\ar_{j} "4";"6"};
(3,7)*+{}="00";
(7,3)*+{}="01";
{\ar@/_0.2pc/@{-}_{^{\xi}} "00";"01"};
(-3,-7)*+{}="10";
(-7,-3)*+{}="11";
{\ar@/_0.2pc/@{-}_{_{\eta}} "10";"11"};
\endxy
\quad
\text{and}
\quad
\Sf\ppr=\ 
\xy
(-7,7)*+{A\ppr}="0";
(7,7)*+{B\ppr}="2";
(-7,-7)*+{O\ppr}="4";
(7,-7)*+{C\ppr}="6";
{\ar^{x\ppr} "0";"2"};
{\ar_{i\ppr} "0";"4"};
{\ar|*+{_{z\ppr}} "0";"6"};
{\ar^{y\ppr} "2";"6"};
{\ar_{j\ppr} "4";"6"};
(3,7)*+{}="00";
(7,3)*+{}="01";
{\ar@/_0.2pc/@{-}_(0.4){^{\xi\ppr}} "00";"01"};
(-3,-7)*+{}="10";
(-7,-3)*+{}="11";
{\ar@/_0.2pc/@{-}_(0.4){_{\eta\ppr}} "10";"11"};
\endxy,
\end{equation}
a \emph{morphism of exact sequences from $\Sf$ to $\Sf\ppr$} means a cube
\begin{equation}\label{MorphExSeq}
\Cf=\ 
\xy
(-8,9)*+{A}="0";
(8,9)*+{B}="2";
(1,3)*+{O}="4";
(17,3)*+{C}="6";
(-8,-7)*+{A\ppr}="10";
(8,-7)*+{B\ppr}="12";
(1,-13)*+{O\ppr}="14";
(17,-13)*+{C\ppr}="16";
{\ar^{x} "0";"2"};
{\ar_{i} "0";"4"};
{\ar^{y} "2";"6"};
{\ar_(0.3){j} "4";"6"};
%
{\ar_{a} "0";"10"};
{\ar^(0.7){b}|!{(9,3);(13,3)}\hole "2";"12"};
{\ar_(0.3){o} "4";"14"};
{\ar^{c} "6";"16"};
{\ar^(0.3){x\ppr}|!{(1,-3);(1,-7)}\hole "10";"12"};
{\ar_{i\ppr} "10";"14"};
{\ar^{y\ppr} "12";"16"};
{\ar_{j\ppr} "14";"16"};
%
\endxy
\end{equation}
which satisfies $\Cf|_{\Delta^1\ti\Delta^1\ti\{0\}}=\Sf$ and $\Cf|_{\Delta^1\ti\Delta^1\ti\{1\}}=\Sf\ppr$. We abbreviate this to $\Cf\co\Sf\to\Sf\ppr$. When we indicate morphisms $a$ and $c$, we will also write ${}_a\Cf_c\co\Sf\to\Sf\ppr$ in this article.
\end{dfn}

\begin{prop}\label{PropWIsom}
Let ${}_A\Sf_C,{}_A\Sf\ppr_C,$ be any pair of exact sequences starting from $A$ and ending in $C$, and let ${}_a\Cf_c\co\Sf\to\Sf\ppr$ be a morphism such that $\ovl{a}=\id_A$ and $\ovl{c}=\id_C$, as depicted in $(\ref{CbPair})$ below.
Suppose that there also exists a morphism in the opposite direction ${}_{a\ppr}\Cf\ppr_{c\ppr}\co\Sf\ppr\to\Sf$ as below, such that $\ovl{a\ppr}=\id_A$ and $\ovl{c\ppr}=\id_C$.
Then $b$ is a homotopy equivalence.
\begin{equation}\label{CbPair}
\Cf=\ \xy
(-8,9)*+{A}="0";
(8,9)*+{B}="2";
(1,3)*+{O}="4";
(17,3)*+{C}="6";
(-8,-7)*+{A}="10";
(8,-7)*+{B\ppr}="12";
(1,-13)*+{O\ppr}="14";
(17,-13)*+{C}="16";
{\ar^{x} "0";"2"};
{\ar_{i} "0";"4"};
{\ar^{y} "2";"6"};
{\ar_(0.3){j} "4";"6"};
%
{\ar_{a} "0";"10"};
{\ar^(0.7){b}|!{(9,3);(13,3)}\hole "2";"12"};
{\ar_(0.3){} "4";"14"};
{\ar^{c} "6";"16"};
{\ar^(0.3){x\ppr}|!{(1,-3);(1,-7)}\hole "10";"12"};
{\ar_{i\ppr} "10";"14"};
{\ar^{y\ppr} "12";"16"};
{\ar_{j\ppr} "14";"16"};
%
\endxy
\quad,\quad
\Cf\ppr=\ \xy
(-8,9)*+{A}="0";
(8,9)*+{B\ppr}="2";
(1,3)*+{O\ppr}="4";
(17,3)*+{C}="6";
(-8,-7)*+{A}="10";
(8,-7)*+{B}="12";
(1,-13)*+{O}="14";
(17,-13)*+{C}="16";
{\ar^{x\ppr} "0";"2"};
{\ar_{i\ppr} "0";"4"};
{\ar^{y\ppr} "2";"6"};
{\ar_(0.3){j\ppr} "4";"6"};
%
{\ar_{a\ppr} "0";"10"};
{\ar^(0.7){b\ppr}|!{(9,3);(13,3)}\hole "2";"12"};
{\ar_(0.3){} "4";"14"};
{\ar^{c\ppr} "6";"16"};
{\ar^(0.3){x}|!{(1,-3);(1,-7)}\hole "10";"12"};
{\ar_{i} "10";"14"};
{\ar^{y} "12";"16"};
{\ar_{j} "14";"16"};
%
\endxy
\end{equation}
\end{prop}
\begin{proof}
It suffices to show that $\ovl{b}$ is an isomorphism in $h\C$.
Remark that $\Cf$ and $\Cf\ppr$ yield commutative diagrams
\[
\xy
(-12,6)*+{A}="2";
(0,6)*+{B}="4";
(12,6)*+{C}="6";
(-12,-6)*+{A}="12";
(0,-6)*+{B\ppr}="14";
(12,-6)*+{C}="16";
{\ar^{\ovl{x}} "2";"4"};
{\ar^{\ovl{y}} "4";"6"};
{\ar@{=} "2";"12"};
{\ar^{\ovl{b}} "4";"14"};
{\ar@{=} "6";"16"};
{\ar_{\ovl{x\ppr}} "12";"14"};
{\ar_{\ovl{y\ppr}} "14";"16"};
{\ar@{}|\circlearrowright "2";"14"};
{\ar@{}|\circlearrowright "4";"16"};
\endxy
\ \ ,\ \ 
\xy
(-12,6)*+{A}="2";
(0,6)*+{B\ppr}="4";
(12,6)*+{C}="6";
(-12,-6)*+{A}="12";
(0,-6)*+{B}="14";
(12,-6)*+{C}="16";
{\ar^{\ovl{x\ppr}} "2";"4"};
{\ar^{\ovl{y\ppr}} "4";"6"};
{\ar@{=} "2";"12"};
{\ar^{\ovl{b\ppr}} "4";"14"};
{\ar@{=} "6";"16"};
{\ar_{\ovl{x}} "12";"14"};
{\ar_{\ovl{y}} "14";"16"};
{\ar@{}|\circlearrowright "2";"14"};
{\ar@{}|\circlearrowright "4";"16"};
\endxy
\]
in $h\C$, whose rows are both weak kernel and weak cokernel. For $\bfr=\ovl{b\ppr}\ci\ovl{b}\in(h\C)(B,B)$, since we have $(\id_B-\bfr)\ci\ovl{x}=0$ and $\ovl{y}\ci(\id_B-\bfr)=0$, there are $\dfr\in(h\C)(B,A)$, $\efr\in(h\C)(C,B)$ such that $\ovl{x}\ci\dfr=\id_B-\bfr$ and $\efr\ci\ovl{y}=\id_B-\bfr$.
It follows that $(\id_B-\bfr)\ci(\id_B-\bfr)=\efr\ci\ovl{y}\ci\ovl{x}\ci\dfr=0$, equivalently $\bfr\ci(2\cdot\id_B-\bfr)=(2\cdot\id_B-\bfr)\ci\bfr=\id_B$. In particular $\bfr=\ovl{b\ppr}\ci\ovl{b}$ is an isomorphism.
The same argument shows that $\ovl{b}\ci\ovl{b\ppr}$ is also an isomorphism. Hence so is $\ovl{b}$. 
\end{proof}

\begin{rem}
Later in Proposition~\ref{PropInvert} we will see (in combination with Corollary~\ref{CorComposeCubes}) that the existence of $\Cf\ppr$ is not necessary to assume, and that it always holds automatically.
\end{rem}

\begin{prop}\label{PropMorphExSeq}
Let $\Sf=\SQ{\eta}{\xi},\Sf\ppr=\SQ{\eta\ppr}{\xi\ppr}$ be any pair of exact sequences as in $(\ref{TwoExSeq})$, and let
\begin{equation}\label{Square_AB}
\Tf=\ 
\xy
(-7,7)*+{A}="0";
(7,7)*+{B}="2";
(-7,-7)*+{A\ppr}="4";
(7,-7)*+{B\ppr}="6";
{\ar^{x} "0";"2"};
{\ar_{a} "0";"4"};
{\ar^{b} "2";"6"};
{\ar_{x\ppr} "4";"6"};
\endxy
\end{equation}
be any square. Then there exists a morphism of exact sequences $\Cf\co\Sf\to\Sf\ppr$ such that $\Cf|_{\Delta^1\ti\{0\}\ti\Delta^1}=\Tf$. More precisely, the following holds.
\begin{enumerate}
\item There exist rectangles of the form
\begin{eqnarray*}
&\Rf=\RT{s_2(\xi)}{s_2(\eta)}{\Zcal}=\ 
\xy
(-16,7)*+{A}="0";
(0,7)*+{O}="2";
(16,7)*+{O\ppr}="4";
(-16,-7)*+{B}="10";
(0,-7)*+{C}="12";
(16,-7)*+{C}="14";
{\ar_{i} "0";"2"};
{\ar_{} "2";"4"};
{\ar@/^1.0pc/^{i_p} "0";"4"};
{\ar_{x} "0";"10"};
{\ar_{j} "2";"12"};
{\ar^{} "4";"14"};
{\ar^{y} "10";"12"};
{\ar^{1_C} "12";"14"};
{\ar@/_1.0pc/_{y} "10";"14"};
\endxy,&\\
&\Rf\ppr=\RT{\Xcal\ppr}{\Ycal\ppr}{\Zcal\ppr}=\ 
\xy
(-16,7)*+{A}="0";
(0,7)*+{A\ppr}="2";
(16,7)*+{O\ppr}="4";
(-16,-7)*+{B}="10";
(0,-7)*+{B\ppr}="12";
(16,-7)*+{C\ppr}="14";
{\ar_{a} "0";"2"};
{\ar_{i\ppr} "2";"4"};
{\ar@/^1.0pc/^{i_p} "0";"4"};
{\ar_{x} "0";"10"};
{\ar_{x\ppr} "2";"12"};
{\ar^{j\ppr} "4";"14"};
{\ar^{b} "10";"12"};
{\ar^{y\ppr} "12";"14"};
{\ar@/_1.0pc/_{} "10";"14"};
\endxy&
\end{eqnarray*}
sharing some common morphism $i_p\in\C_1(A,O\ppr)$,
such that $\Sf=(\Rf_{\mathrm{left}})^t$, $\Sf\ppr=(\Rf\ppr_{\mathrm{right}})^t$ and $\Tf=(\Rf\ppr_{\mathrm{left}})^t$.

Moreover, if $A=A\ppr,a=1_A$ and $\Tf=\SQ{s_0(x\ppr)}{\xi}$, then we may take $i_p=i\ppr$, and $\Rf\ppr$ can be chosen to be of the form $\RT{\Xcal\ppr}{s_0(\xi\ppr)}{s_0(\eta\ppr)}$.
\item For any $\Rf=\RT{s_2(\xi)}{s_2(\eta)}{\Zcal}$ and $\Rf\ppr=\RT{\Xcal\ppr}{\Ycal\ppr}{\Zcal\ppr}$ as in {\rm (1)}, there exist three $3$-simplices $\Phi,\Psi_{\ff},\Psi_{\bb}$ which gives a morphism
\[ \Cf=\CB{\Ycal\ppr}{\Xcal\ppr}{\Psi_{\ff}}{\Zcal\ppr}{\Phi}{\Psi_{\bb}}\co\Sf\to\Sf\ppr \]
satisfying $d_3(\Phi)=d_3(\Zcal)$.
\end{enumerate}
\end{prop}
\begin{proof}
{\rm (1)} By using the definition of a zero object, we may take a $3$-simplex
\[
\Zcal=\ 
\xy
(-2,11)*+{A}="0";
(-12,-2)*+{O}="2";
(10,0)*+{O\ppr}="4";
(2,-12)*+{C}="6";
{\ar_{i} "0";"2"};
{\ar^{i_p} "0";"4"};
{\ar_(0.3){z} "0";"6"};
{\ar_(0.35){}|!{"0";"6"}\hole "2";"4"};
{\ar_{j} "2";"6"};
{\ar_{} "4";"6"};
\endxy
\]
with arbitrarily chosen morphisms $A\ov{i_p}{\lra}O\ppr$, $O\to O\ppr$, $O\ppr\to C$, such that $d_2(\Zcal)=\eta$. This gives $\Rf$ as desired. By the definition of a zero object, there also exists a $2$-simplex $\al$ as below. Also, we may take a $2$-simplex $\be$ arbitrarily as below, with some $y_p\in\C_1(B,C\ppr)$.
\[ \TwoSP{A}{A\ppr}{O\ppr}{a}{i\ppr}{i_p}{\al},\ \ \TwoSP{B}{B\ppr}{C\ppr}{b}{y\ppr}{y_p}{\be} \]
Then the existence of $\Rf\ppr$ follows from Proposition~\ref{PropComposeSquares} applied to $\Tf,\Sf\ppr,\al,\be$. The latter part is also obvious from the construction, in which case we use $\al=s_0(i\ppr)$.

{\rm (2)} We may take a $3$-simplex
\[
\Phi=\ 
\xy
(-2,11)*+{A}="0";
(-12,-2)*+{O}="2";
(10,0)*+{C\ppr}="4";
(2,-12)*+{O\ppr}="6";
{\ar_{i} "0";"2"};
{\ar^{} "0";"4"};
{\ar_(0.3){i_p} "0";"6"};
{\ar_(0.35){}|!{"0";"6"}\hole "2";"4"};
{\ar_{} "2";"6"};
{\ar_{j\ppr} "6";"4"};
\endxy
\]
such that $d_1(\Phi)=d_1(\Zcal\ppr)$ and $d_3(\Phi)=d_3(\Zcal)$, with arbitrarily taken compatible $d_0(\Phi)$ by using the definition of a zero object.
Put $\Sf^p=(\Rf_{\mathrm{out}})^t=\SQ{d_1(\Zcal)}{\xi}$ and
\[ \Sf^q=\SQ{d_1(\Zcal\ppr)}{d_2(\Xcal\ppr)}=\ 
\xy
(-6,6)*+{A}="0";
(6,6)*+{B}="2";
(-6,-6)*+{O\ppr}="4";
(6,-6)*+{C\ppr}="6";
{\ar^{x} "0";"2"};
{\ar_{i_p} "0";"4"};
{\ar^{} "2";"6"};
{\ar_{j\ppr} "4";"6"};
\endxy.
\]
Since $(\Rf_{\mathrm{left}})^t,(\Rf_{\mathrm{right}})^t$ are exact sequences, so is $\Sf^p$. Thus by the definition of push-out, there exists a morphism $\roundup{\Gamma}{\Psi_{\ff}}$ from $\Sf^p$ to $\Sf^q$ in the undercategory, namely a pair of $3$-simplices $\Psi_{\ff},\Gamma$ in $\C$ satisfying
$d_2(\Psi_{\ff})=d_2(\Xcal\ppr)$, $d_3(\Psi_{\ff})=\xi$, $d_2(\Gamma)=d_1(\Zcal\ppr)$, $d_3(\Gamma)=d_1(\Zcal)$, and $d_1(\Gamma)=d_1(\Psi_{\ff})$.
 We may take a $3$-simplex $\nu$
such that $d_0(\nu)=d_0(\Gamma)$, $d_2(\nu)=d_0(\Phi)$, $d_3(\nu)=d_0(\Zcal)$ and then a $4$-simplex
\[
\Thh=\Penta{A}{O}{O\ppr}{C}{C\ppr}{i}{}{}{}
\]
such that $d_0(\Thh)=\nu$, $d_1(\Thh)=\Gamma$, $d_3(\Thh)=\Phi$, $d_4(\Thh)=\Zcal$. If we put $\Psi_{\bb}=d_2(\Thh)$, then $\Phi,\Psi_{\ff},\Psi_{\bb}$ satisfy the required conditions.
\end{proof}

The following will be used later to replace the zero objects appearing in exact sequences. 
\begin{lem}\label{LemReplZero}
Let
\[
\nu=\ \xy
(-1,7)*+{A}="0";
(-8,-2)*+{O}="1";
(1,-10)*+{O\ppr}="2";
(9,-1)*+{C}="3";
{\ar_{} "0";"1"};
{\ar^{} "0";"2"};
{\ar_{} "0";"3"};
{\ar_{} "1";"2"};
{\ar^{}|!{"0";"2"}\hole "1";"3"};
{\ar_{} "2";"3"};
\endxy
\quad\text{and}\quad
\nu\ppr=\ 
\xy
(-1,7)*+{A}="0";
(-8,-2)*+{O}="1";
(1,-10)*+{O\pprr}="2";
(9,-1)*+{C}="3";
{\ar_{} "0";"1"};
{\ar^{} "0";"2"};
{\ar_{} "0";"3"};
{\ar_{} "1";"2"};
{\ar^{}|!{"0";"2"}\hole "1";"3"};
{\ar_{} "2";"3"};
\endxy
\]
be any pair of $3$-simplices satisfying $d_2(\nu)=d_2(\nu\ppr)$, in which $O,O\ppr,O\pprr$ are zero objects. Let $o\in\C_1(O\pprr,O\ppr)$ be arbitrarily taken morphism. Then there exists a $3$-simplex
\[
\mu=\ 
\xy
(-1,7)*+{A}="0";
(-8,-2)*+{O\pprr}="1";
(1,-10)*+{O\ppr}="2";
(9,-1)*+{C}="3";
{\ar_{} "0";"1"};
{\ar^{} "0";"2"};
{\ar_{} "0";"3"};
{\ar_{o} "1";"2"};
{\ar^{}|!{"0";"2"}\hole "1";"3"};
{\ar_{} "2";"3"};
\endxy
\]
such that $d_1(\mu)=d_1(\nu)$ and $d_2(\mu)=d_1(\nu\ppr)$.
\end{lem}
\begin{proof}
We can show the existence of a $4$-simplex
\[ \Phi=\Penta{A}{O}{O\pprr}{O\ppr}{C}{}{}{o}{} \]
such that $d_2(\Phi)=\nu$ and $d_3(\Phi)=\nu\ppr$, with arbitrarily chosen compatible $3$-simplices $d_0(\Phi)$ and $d_4(\Phi)$ by using the definition of zero objects. This gives $\mu=d_1(\Phi)$ as desired.
\end{proof}

\begin{prop}\label{PropReplZero}
Let
\[
{}_A\Sf_C=\SQ{\eta}{\xi}=\ 
\xy
(-7,7)*+{A}="0";
(7,7)*+{B}="2";
(-7,-7)*+{O}="4";
(7,-7)*+{C}="6";
{\ar^{x} "0";"2"};
{\ar_{i} "0";"4"};
{\ar|*+{_z} "0";"6"};
{\ar^{y} "2";"6"};
{\ar_{j} "4";"6"};
(3,7)*+{}="00";
(7,3)*+{}="01";
{\ar@/_0.2pc/@{-}_{^{\xi}} "00";"01"};
(-3,-7)*+{}="10";
(-7,-3)*+{}="11";
{\ar@/_0.2pc/@{-}_{_{\eta}} "10";"11"};
\endxy
\quad\text{and}\quad
{}_A\Sf\ppr_C=\SQ{\eta\ppr}{\xi}=\ 
\xy
(-7,7)*+{A}="0";
(7,7)*+{B}="2";
(-7,-7)*+{O\ppr}="4";
(7,-7)*+{C}="6";
{\ar^{x} "0";"2"};
{\ar_{i\ppr} "0";"4"};
{\ar|*+{_z} "0";"6"};
{\ar^{y} "2";"6"};
{\ar_{j\ppr} "4";"6"};
(3,7)*+{}="00";
(7,3)*+{}="01";
{\ar@/_0.2pc/@{-}_{^{\xi}} "00";"01"};
(-3,-7)*+{}="10";
(-7,-3)*+{}="11";
{\ar@/_0.2pc/@{-}_{_{\eta\ppr}} "10";"11"};
\endxy
\]
be any pair of squares sharing common $\xi$ in the upper triangles. Suppose that there is a $3$-simplex $\nu$ such that $d_1(\nu)=\eta\ppr$ and $d_2(\nu)=\eta$.
Then the following holds.
\begin{enumerate}
\item There exists a $3$-simplex $\nu\ppr$ such that $d_1(\nu\ppr)=\eta$ and $d_2(\nu\ppr)=\eta\ppr$.
\item $\Sf$ is an exact sequence if and only if $\Sf\ppr$ is an exact sequence. 
\end{enumerate}
We also remark that in case {\rm (2)}, the given $\nu$ and $\nu\ppr$ obtained in {\rm (1)} are nothing but morphisms of exact sequences
\[ \CB{s_0(\xi)}{s_1(\xi)}{s_2(\xi)}{s_0(\eta\ppr)}{\nu}{s_2(\eta)}\co\Sf\to\Sf\ppr
\quad\text{and}\quad
\CB{s_0(\xi)}{s_1(\xi)}{s_2(\xi)}{s_0(\eta)}{\nu\ppr}{s_2(\eta\ppr)}\co\Sf\ppr\to\Sf. \]
\end{prop}
\begin{proof}
{\rm (1)} is immediate from Lemma~\ref{LemReplZero}, applied to $\nu$ and $s_1(\eta)$. {\rm (2)} follows from the existence of a rectangle $\Rf=\RT{\nu}{s_0(\eta\ppr)}{s_0(\xi)}$, in which $\Rf_{\mathrm{left}}$ is an exact sequence and $\Rf_{\mathrm{out}}=\Sf$, $\Rf_{\mathrm{right}}=\Sf\ppr$.
\end{proof}

\subsection{Decompositions of morphisms}

\begin{lem}\label{LemExSeqPO}
Let ${}_A\Sf_C=\SQ{\eta}{\xi}$ be any exact sequence as in $(\ref{ExSeq})$, and let $a\in\C_1(A,A\ppr)$ be any morphism. Let
\begin{equation}\label{SquareABPO}
\Pf=\ 
\xy
(-7,7)*+{A}="0";
(7,7)*+{B}="2";
(-7,-7)*+{A\ppr}="4";
(7,-7)*+{B_0}="6";
{\ar^{x} "0";"2"};
{\ar_{a} "0";"4"};
{\ar^{b_0} "2";"6"};
{\ar_{x_0} "4";"6"};
\endxy
\end{equation}
be arbitrarily chosen push-out square, which always exists by the definition of exact quasi-category. Then, we have the following.
\begin{enumerate}
\item There exists an exact sequence
\begin{equation}\label{ExSeqPO}
\Sf_0=\ 
\xy
(-7,7)*+{A\ppr}="0";
(7,7)*+{B_0}="2";
(-7,-7)*+{O}="4";
(7,-7)*+{C}="6";
{\ar^{x_0} "0";"2"};
{\ar_{i_0} "0";"4"};
{\ar|*+{_{z_0}} "0";"6"};
{\ar^{y_0} "2";"6"};
{\ar_{j} "4";"6"};
(3,7)*+{}="00";
(7,3)*+{}="01";
{\ar@/_0.2pc/@{-}_(0.4){^{\xi_0}} "00";"01"};
(-3,-7)*+{}="10";
(-7,-3)*+{}="11";
{\ar@/_0.2pc/@{-}_(0.4){_{\eta_0}} "10";"11"};
\endxy
\end{equation}
and a morphisms of exact sequences $\Df\co \Sf\to\Sf_0$ of the form
\[
\Df=\CB{\vp}{\psi}{s_2(\xi)}{\thh}{s_1(\eta)}{s_2(\eta)}\ =
\xy
(-8,9)*+{A}="0";
(8,9)*+{B}="2";
(1,3)*+{O}="4";
(17,3)*+{C}="6";
(-8,-7)*+{A\ppr}="10";
(8,-7)*+{B_0}="12";
(1,-13)*+{O}="14";
(17,-13)*+{C}="16";
{\ar^{x} "0";"2"};
{\ar_{i} "0";"4"};
{\ar^{y} "2";"6"};
{\ar_(0.3){j} "4";"6"};
%
{\ar_{a} "0";"10"};
{\ar^(0.7){b_0}|!{(9,3);(13,3)}\hole "2";"12"};
{\ar_(0.3){1_O} "4";"14"};
{\ar^{1_C} "6";"16"};
{\ar^(0.3){x_0}|!{(1,-3);(1,-7)}\hole "10";"12"};
{\ar_{i_0} "10";"14"};
{\ar^{y_0} "12";"16"};
{\ar_{j} "14";"16"};
%
\endxy
\]
such that $\Df|_{\Delta^1\ti\{0\}\ti\Delta^1}=\Pf$. We denote any such square $\Sf_0$ admitting $\Df$ by $a\sas\Sf$.

\item Let $\Df\co \Sf\to\Sf_0=a\sas\Sf$ be arbitrarily as in {\rm (1)}. For any morphism of exact sequences
\[ {}_a\Cf_c=\CB{\Xcal_{\ff}}{\Ycal_{\ff}}{\Zcal_{\ff}}{\Xcal_{\bb}}{\Ycal_{\bb}}{\Zcal_{\bb}}\co {}_A\Sf_C\to {}_{A\ppr}\Sf\ppr_{C\ppr}=\SQ{\eta\ppr}{\xi\ppr} \]
as in $(\ref{MorphExSeq})$, there exists a morphism of exact sequences ${}_{1_{A\ppr}}\Cf\ppr_{c}\co a\sas\Sf\to\Sf\ppr$ of the form
\begin{equation}\label{CubeFactorPO}
\Cf\ppr=\CB{s_0(\xi\ppr)}{\Ycal\ppr_{\ff}}{\Zcal\ppr_{\ff}}{s_0(\eta\ppr)}{\Ycal\ppr_{\bb}}{\Zcal\ppr_{\bb}}=\ 
\xy
(-8,9)*+{A\ppr}="0";
(8,9)*+{B_0}="2";
(1,3)*+{O}="4";
(17,3)*+{C}="6";
(-8,-7)*+{A\ppr}="10";
(8,-7)*+{B\ppr}="12";
(1,-13)*+{O\ppr}="14";
(17,-13)*+{C\ppr}="16";
{\ar^{x_0} "0";"2"};
{\ar_{i_0} "0";"4"};
{\ar^{y_0} "2";"6"};
{\ar_(0.3){j} "4";"6"};
%
{\ar_{1_{A\ppr}} "0";"10"};
{\ar^(0.7){b\ppr}|!{(9,3);(13,3)}\hole "2";"12"};
{\ar_(0.3){o} "4";"14"};
{\ar^{c} "6";"16"};
{\ar^(0.3){x\ppr}|!{(1,-3);(1,-7)}\hole "10";"12"};
{\ar_{i\ppr} "10";"14"};
{\ar^{y\ppr} "12";"16"};
{\ar_{j\ppr} "14";"16"};
%
\endxy
\end{equation}
for some $\Ycal\ppr_{\ff},\Ycal\ppr_{\bb},\Zcal\ppr_{\ff},\Zcal\ppr_{\bb}$, such that $\ovl{b}=\ovl{b\ppr}\ci\ovl{b}_0$ holds in $h\C$.
\end{enumerate}

Dually, for any ${}_A\Sf_C=\SQ{\eta}{\xi}$ as above and any $c\in\C(C\ppr,C)$, we may obtain a morphism of exact sequences of the form
\[ {}_{1_A}\Df\ppr_c=\CB{s_0(\xi)}{s_1(\xi)}{\Zcal_{\ff}}{s_0(\eta)}{\Ycal_{\bb}}{\Zcal_{\bb}}\co\Sf_0\ppr\to\Sf \]
in which $\Df\ppr|_{\Delta^1\ti\{1\}\ti\Delta^1}$ is a pull-back. In this case, we denote such $\Sf_0\ppr$ by $c\uas\Sf$. It has a property dual to {\rm (2)}.
\end{lem}
\begin{proof}
{\rm (1)} By using the definition of a zero object, we can take a $3$-simplex
\[
\thh=\ 
\xy
(-3,7)*+{A}="0";
(-12,-6)*+{A\ppr}="2";
(10,-4)*+{C}="4";
(2,-16)*+{O}="6";
{\ar_{a} "0";"2"};
{\ar^{z} "0";"4"};
{\ar_(0.3){i} "0";"6"};
{\ar_(0.7){z_0}|!{"0";"6"}\hole "2";"4"};
{\ar_{i_0} "2";"6"};
{\ar_{j} "6";"4"};
\endxy
\]
such that $d_1(\thh)=\eta$, with some morphisms $i_0\in\C_1(A\ppr,O)$ and $z_0\in\C_1(A\ppr,C)$. Put $\nu=d_2(\thh)$. Since $\Pf$ is a push-out, we obtain a morphism
\[ \roundup{\vp}{\psi}\co\Pf\to
\xy
(-7,7)*+{A}="0";
(7,7)*+{B}="2";
(-7,-7)*+{A\ppr}="4";
(7,-7)*+{C}="6";
{\ar^{x} "0";"2"};
{\ar_{a} "0";"4"};
{\ar|*+{_z} "0";"6"};
{\ar^{y} "2";"6"};
{\ar_{z_0} "4";"6"};
(3,7)*+{}="00";
(7,3)*+{}="01";
{\ar@/_0.2pc/@{-}_{^{\xi}} "00";"01"};
(-3,-7)*+{}="10";
(-7,-3)*+{}="11";
{\ar@/_0.2pc/@{-}_{_{\nu}} "10";"11"};
\endxy
\]
in the undercategory.
Then this gives a rectangle $\Rf=\RT{\psi}{\vp}{\thh}$ with
$\Rf_{\mathrm{left}}=\Pf^t$, $\Rf_{\mathrm{out}}=\Sf^t$ and $\Rf_{\mathrm{right}}=(\Sf_0)^t$.
Since $\Pf$, $\Sf^t$ are push-outs, it follows that $x_0$ is ingressive and that $\Sf_0=(\Rf_{\mathrm{right}})^t$ is also a push-out. Thus $\Sf_0$ is an exact sequence, and $\Df=\CB{\vp}{\psi}{s_2(\xi)}{\thh}{s_1(\eta)}{s_2(\eta)}\co\Sf\to\Sf_0$ indeed gives a morphism with the stated property.

{\rm (2)} We may take a $4$-simplex
\[ \Om=\Penta{A}{A\ppr}{O}{O\ppr}{C\ppr}{}{}{}{}{} \]
such that $d_1(\Om)=\Ycal_{\bb}$ and $d_2(\Om)=\Xcal_{\bb}$, for arbitrarily chosen compatible $d_0(\Om)$ and $d_4(\Om)$ satisfying $d_3d_4(\Om)=d_3(\thh)$ using the definition of zero object.
Applying Lemma~\ref{LemRectProl} to
\[ \Rf=\RT{\psi}{\vp}{\thh},\ \ \roundup{\Zcal_{\ff}}{\Zcal_{\bb}}\co\Rf_{\mathrm{out}}\to\xy
(-6,6)*+{A}="0";
(6,6)*+{O}="2";
(-6,-6)*+{B}="4";
(6,-6)*+{C\ppr}="6";
{\ar^{i} "0";"2"};
{\ar_{x} "0";"4"};
{\ar^{} "2";"6"};
{\ar_{} "4";"6"};
%
%
\endxy,\ \ \wp=d_3(\Om), \]
we obtain $4$-simplices $\Psi,\Phi,\Theta$ such that
$d_1(\Psi)=d_1(\Phi)$, $d_2(\Phi)=d_2(\Thh)$, $\roundup{\Zcal_{\ff}}{\Zcal_{\bb}}=\roundup{d_2(\Psi)}{d_1(\Thh)}$, $d_3(\Thh)=\wp$, and $\RT{\psi}{\vp}{\thh}=\RT{d_4(\Phi)}{d_4(\Psi)}{d_4(\Thh)}$.

If we put
\begin{eqnarray*}
&\Tf=\Cf|_{\{0\}\ti\Delta^1\ti\Delta^1}=\SQ{d_3(\Xcal_{\ff})}{d_3(\Ycal_{\ff})},&\\
&\Tf\ppr=\SQ{d_2(\Xcal_{\bb})}{d_2(\Zcal_{\ff})}=\SQ{d_2(\Xcal_{\ff})}{d_2(\Ycal_{\ff})},&
\end{eqnarray*}
Then $\Tf,\Tf\ppr$ and $\Pf$ can be regarded as objects in the quasi-category under
\[
\xy
(-6,6)*+{A}="0";
(6,6)*+{B}="2";
(-6,-6)*+{A\ppr}="4";
{\ar^{x} "0";"2"};
{\ar_{a} "0";"4"};
\endxy
\]
and there is a morphism $\roundup{\Xcal_{\ff}}{\Ycal_{\ff}}\co\Tf\to\Tf\ppr$.
By the construction so far, we also have $\roundup{d_3(\Phi)}{d_3(\Psi)}\co\Pf\to\Tf\ppr$ in this undercategory.
Since $\Pf$ is a push-out, we obtain another morphism $\roundup{\kap}{\lam}\co\Pf\to\Tf$.
Then, again by the definition of push-out, we obtain a $2$-simplex in this undercategory, namely a pair of $4$-simplices in $\C$
\[
\Gamma=\Penta{A}{A\ppr}{B_0}{B\ppr}{C\ppr}{a}{x_0}{}{}\ ,\ \ 
\Xi=\Penta{A}{B}{B_0}{B\ppr}{C\ppr}{x}{b_0}{}{}
\]
such that
\begin{eqnarray*}
&d_1(\Gamma)=d_1(\Xi), d_2(\Gamma)=\Xcal_{\ff}, d_2(\Xi)=\Ycal_{\ff},&\\
&d_3(\Gamma)=d_3(\Phi), d_3(\Xi)=d_3(\Psi), d_4(\Gamma)=\kap, d_4(\Xi)=\lam.&
\end{eqnarray*}
Now
\[ \Cf\ppr=\CB{s_0(\xi\ppr)}{d_0(\Gamma)}{d_0(\Phi)}{s_0(\eta\ppr)}{d_0(\Om)}{d_0(\Thh)}\co\Sf_0\to\Sf\ppr \]
gives a morphism of exact sequences with the stated properties.
\end{proof}

\begin{lem}\label{LemForSym}
Let ${}_A\Sf_C,{}_{A\ppr}\Sf\ppr_{C\ppr}$ be two exact sequences as in $(\ref{TwoExSeq})$, and let
\[ {}_a\Cf_c=\CB{\Xcal_{\ff}}{\Ycal_{\ff}}{\Zcal_{\ff}}{\Xcal_{\bb}}{\Ycal_{\bb}}{\Zcal_{\bb}}\co\Sf\to\Sf\ppr \]
be a morphism as in $(\ref{MorphExSeq})$. If $C=C\ppr$ and $\ovl{c}=\id_C$, then the following holds.
\begin{enumerate}
\item The square
\begin{equation}\label{CPB}
\Cf|_{\{0\}\ti\Delta^1\ti\Delta^1}=\ 
\xy
(-6,6)*+{A}="0";
(6,6)*+{B}="2";
(-6,-6)*+{A\ppr}="4";
(6,-6)*+{B\ppr}="6";
{\ar^{x} "0";"2"};
{\ar_{a} "0";"4"};
{\ar^{b} "2";"6"};
{\ar_{x\ppr} "4";"6"};
\endxy
\end{equation}
is a pull-back.
\item If moreover $b$ is egressive, then $(\ref{CPB})$ is an ambigressive push-out.
\end{enumerate}
\end{lem}
\begin{proof}
{\rm (1)} $\Cf$ contains rectangles
\begin{eqnarray*}
&\Rf=\RT{\Ycal_{\ff}}{\Xcal_{\ff}}{\Xcal_{\bb}}=\ 
\xy
(-12,6)*+{A}="0";
(0,6)*+{A\ppr}="2";
(12,6)*+{O\ppr}="4";
(-12,-6)*+{B}="10";
(0,-6)*+{B\ppr}="12";
(12,-6)*+{C}="14";
{\ar^{a} "0";"2"};
{\ar^{i\ppr} "2";"4"};
%
{\ar_{x} "0";"10"};
{\ar_{x\ppr} "2";"12"};
{\ar^{j\ppr} "4";"14"};
{\ar_{b} "10";"12"};
{\ar_{y\ppr} "12";"14"};
%
\endxy,
&\\
&\Rf\ppr=\RT{\Zcal_{\ff}}{\Zcal_{\bb}}{\Ycal_{\bb}}=\ 
\xy
(-12,6)*+{A}="0";
(0,6)*+{O}="2";
(12,6)*+{O\ppr}="4";
(-12,-6)*+{B}="10";
(0,-6)*+{C}="12";
(12,-6)*+{C}="14";
{\ar^{i} "0";"2"};
{\ar^{o} "2";"4"};
%
{\ar_{x} "0";"10"};
{\ar_{j} "2";"12"};
{\ar^{j\ppr} "4";"14"};
{\ar_{y} "10";"12"};
{\ar_{c} "12";"14"};
%
\endxy&
\end{eqnarray*}
which satisfies $(\Rf_{\mathrm{right}})^t=\Sf\ppr$, $\Rf_{\mathrm{out}}=\Rf\ppr_{\mathrm{out}}$ and $(\Rf\ppr_{\mathrm{left}})^t=\Sf$. Since $(\Rf_{\mathrm{right}})^t$ is an exact sequence as in Example~\ref{ExTrivExSeq}, it follows that $\Cf|_{\{0\}\ti\Delta^1\ti\Delta^1}=(\Rf_{\mathrm{left}})^t$ is a pull-back.
{\rm (2)} is immediate from the definition of exact quasi-category.
\end{proof}

\begin{prop}\label{PropForSym}
Let
\[
\Sf=\SQ{\eta}{\xi}=
\xy
(-7,7)*+{A}="0";
(7,7)*+{B}="2";
(-7,-7)*+{O}="4";
(7,-7)*+{C}="6";
{\ar^{x} "0";"2"};
{\ar_{i} "0";"4"};
{\ar|*+{_z} "0";"6"};
{\ar^{y} "2";"6"};
{\ar_{j} "4";"6"};
(3,7)*+{}="00";
(7,3)*+{}="01";
{\ar@/_0.2pc/@{-}_{^{\xi}} "00";"01"};
(-3,-7)*+{}="10";
(-7,-3)*+{}="11";
{\ar@/_0.2pc/@{-}_{_{\eta}} "10";"11"};
\endxy
\quad\text{and}\quad
\Sf\ppr=\SQ{\eta\ppr}{\xi\ppr}=
\xy
(-7,7)*+{A}="0";
(7,7)*+{B\ppr}="2";
(-7,-7)*+{O\ppr}="4";
(7,-7)*+{C}="6";
{\ar^{x\ppr} "0";"2"};
{\ar_{i\ppr} "0";"4"};
{\ar|*+{_{z\ppr}} "0";"6"};
{\ar^{y\ppr} "2";"6"};
{\ar_{j\ppr} "4";"6"};
(3,7)*+{}="00";
(7,3)*+{}="01";
{\ar@/_0.2pc/@{-}_(0.4){^{\xi\ppr}} "00";"01"};
(-3,-7)*+{}="10";
(-7,-3)*+{}="11";
{\ar@/_0.2pc/@{-}_(0.4){_{\eta\ppr}} "10";"11"};
\endxy
\]
be two exact sequences starting from $A$ and ending in $C$. Let
\begin{equation}\label{CubeForSym}
{}_{1_A}\Cf_{1_C}=\CB{\Xcal_{\ff}}{\Ycal_{\ff}}{\Zcal_{\ff}}{\Xcal_{\bb}}{\Ycal_{\bb}}{\Zcal_{\bb}}=\ 
\xy
(-8,9)*+{A}="0";
(8,9)*+{B}="2";
(1,3)*+{O}="4";
(17,3)*+{C}="6";
(-8,-7)*+{A}="10";
(8,-7)*+{B\ppr}="12";
(1,-13)*+{O\ppr}="14";
(17,-13)*+{C}="16";
{\ar^{x} "0";"2"};
{\ar_{i} "0";"4"};
{\ar^{y} "2";"6"};
{\ar_(0.3){j} "4";"6"};
%
{\ar_{1_A} "0";"10"};
{\ar^(0.7){b}|!{(9,3);(13,3)}\hole "2";"12"};
{\ar_(0.3){o} "4";"14"};
{\ar^{1_C} "6";"16"};
{\ar^(0.3){x\ppr}|!{(1,-3);(1,-7)}\hole "10";"12"};
{\ar_{i\ppr} "10";"14"};
{\ar^{y\ppr} "12";"16"};
{\ar_{j\ppr} "14";"16"};
%
\endxy\ \ \co\Sf\to\Sf\ppr
\end{equation}
be any morphism. 

If moreover $b=\Cf|_{\{(0,1)\}\ti\Delta^1}$ is a homotopy equivalence in $\C$, then there also exists a morphism in the opposite direction ${}_{1_A}(\Cf^{\prime})_{1_C}\co\Sf\ppr\to\Sf$. In addition, $b\ppr=\Cf\ppr|_{\{(0,1)\}\ti\Delta^1}$ becomes a homotopy equivalence in $\C$.
\end{prop}
\begin{proof}
Put $\xi^{(1)}=d_2(\Ycal_{\ff}),\eta^{(1)}=d_2(\Ycal_{\bb}),\eta^{(2)}=d_1(\Xcal_{\bb})$, and put
\[
\Sf^{(1)}=\SQ{\eta^{(1)}}{\xi^{(1)}}=
\xy
(-6,6)*+{A}="0";
(6,6)*+{B}="2";
(-6,-6)*+{O}="4";
(6,-6)*+{C}="6";
{\ar^{x} "0";"2"};
{\ar_{i} "0";"4"};
{\ar^{} "2";"6"};
{\ar_{} "4";"6"};
\endxy,\ \ 
\Sf^{(2)}=\SQ{\eta^{(2)}}{\xi^{(1)}}=
\xy
(-6,6)*+{A}="0";
(6,6)*+{B}="2";
(-6,-6)*+{O\ppr}="4";
(6,-6)*+{C}="6";
{\ar^{x} "0";"2"};
{\ar_{} "0";"4"};
{\ar^{} "2";"6"};
{\ar_{j\ppr} "4";"6"};
\endxy.
\]
In the following rectangle, since $(\Rf_{\mathrm{left}})^t=\Sf$ and $(\Rf_{\mathrm{right}})^t$ are exact sequences,
\[
\Rf=\RT{\Zcal_{\ff}}{\Zcal_{\bb}}{s_1(\eta^{(1)})}=\ 
\xy
(-12,6)*+{A}="0";
(0,6)*+{O}="2";
(12,6)*+{O}="4";
(-12,-6)*+{B}="10";
(0,-6)*+{C}="12";
(12,-6)*+{C}="14";
{\ar^{i} "0";"2"};
{\ar^{1_O} "2";"4"};
%
{\ar_{x} "0";"10"};
{\ar_{j} "2";"12"};
{\ar^{} "4";"14"};
{\ar_{y} "10";"12"};
{\ar_{1_C} "12";"14"};
%
\endxy
\]
so is $\Sf^{(1)}=(\Rf_{\mathrm{out}})^t$. Moreover,
\[ {}_{1_A}\Cf^{(1)}_{1_C}=
\CB{s_0(\xi^{(1)})}{s_1(\xi^{(1)})}{\Zcal_{\ff}}{s_0(\eta^{(1)})}{s_1(\eta^{(1)})}{\Zcal_{\bb}}
\co {}_A\Sf_C\to {}_A\Sf^{(1)}_C \]
is a morphism of exact sequences. By Proposition~\ref{PropReplZero}, the existence of the $3$-simplex $\Ycal_{\bb}$ shows that $\Sf^{(2)}$ is also an exact sequence, and that 
\[ {}_{1_A}\Cf^{(2)}_{1_C}=
\CB{s_0(\xi^{(1)})}{s_1(\xi^{(1)})}{s_2(\xi^{(1)})}{s_0(\eta^{(2)})}{\Ycal_{\bb}}{s_2(\eta^{(1)})}
\co {}_A\Sf^{(1)}_C\to {}_A\Sf^{(2)}_C \]
is a morphisms of exact sequences. Besides, there is another morphism
\[ {}_{1_A}\Cf^{(3)}_{1_C}=
\CB{\Xcal_{\ff}}{\Ycal_{\ff}}{s_2(\xi^{(1)})}{\Xcal_{\bb}}{s_1(\eta^{(2)})}{s_2(\eta^{(2)})}
\co {}_A\Sf^{(2)}_C\to {}_A\Sf\ppr_C, \]
hence we have a sequence of morphisms
$\Sf\ov{\Cf^{(1)}}{\lra}\Sf^{(1)}\ov{\Cf^{(2)}}{\lra}\Sf^{(2)}\ov{\Cf^{(3)}}{\lra}\Sf\ppr$.

By Lemma~\ref{LemForSym}, it follows that $\Cf^{(3)}|_{\{0\}\ti\Delta^1\ti\Delta^1}$ is an ambigressive push-out. Thus $\Cf^{(3)}$ satisfies the conditions required in Lemma~\ref{LemExSeqPO} {\rm (1)}. Thus by {\rm (2)} of the same lemma applied to $\Cf^{(3)}$ and $\If_{\Sf^{(2)}}\co\Sf^{(2)}\to\Sf^{(2)}$, we obtain a morphism ${}_{1_A}(\Cf^{(3)\prime})_{1_C}\co\Sf\ppr\to\Sf^{(2)}$. Similarly, by using the duals of Lemmas~\ref{LemExSeqPO} and \ref{LemForSym}, we obtain a morphism ${}_{1_A}(\Cf^{(1)\prime})_{1_C}\co\Sf^{(1)}\to\Sf$.
By Proposition~\ref{PropReplZero}, we also have a morphism ${}_{1_A}(\Cf^{(2)\prime})_{1_C}\co\Sf^{(2)}\to\Sf^{(1)}$.
Now by Proposition~\ref{PropComposeCubes} applied iteratively to the sequence of morphisms
\[ \Sf\ov{\Cf^{(1)\prime}}{\lla}\Sf^{(1)}\ov{\Cf^{(2)\prime}}{\lla}\Sf^{(2)}\ov{\Cf^{(3)\prime}}{\lla}\Sf\ppr, \]
we obtain a morphism $\Cf^{\prime}\co\Sf\ppr\to\Sf$ as desired.
The last assertion for $b\ppr$ follows from Proposition~\ref{PropWIsom}.
\end{proof}

\begin{rem}\label{RemCi}
Proposition~\ref{PropForSym} also follows immediately from \cite[Corollary~3.5.12]{C}. In fact, the morphism $\Cf\in\Fun(\Delta^1\ti\Delta^1,\C)(\Sf,\Sf\ppr)$ satisfying the assumptions in Proposition~\ref{PropForSym} becomes a homotopy equivalence in $\Fun(\Delta^1\ti\Delta^1,\C)(\Sf,\Sf\ppr)$ (hence in $\Esc$).
\end{rem}

\begin{rem}\label{RemForSym}
The assumption on $b$ in Proposition~\ref{PropForSym} is redundant. Indeed in Proposition~\ref{PropInvert} we will see that it holds automatically.
\end{rem}

\subsection{Splitting exact sequences}

\begin{dfn}\label{DefForSplit}
Let $A,C\in\C_0$ be any pair of objects, and let
\begin{equation}\label{2SimplexETA}
\TwoSP{A}{O}{C}{i}{j}{z}{\eta}
\end{equation}
be any $2$-simplex in $\C$, in which $O\in\C_0$ is a zero object. Let $A\ov{p_A}{\lla}A\ti C\ov{p_C}{\lra}C$ and $A\ov{i_A}{\lra}A\am C\ov{i_C}{\lla}C$ be a product and a coproduct of $A$ and $C$, respectively.
\begin{enumerate}
\item By the definition of product, we obtain a pair of $2$-simplices as below.
\begin{equation}\label{ProdAC}
\xy
(0,7)*+{A}="0";
(-18,-7)*+{A}="2";
(0,-7)*+{A\ti C}="4";
(18,-7)*+{C}="6";
{\ar_{1_A} "0";"2"};
{\ar^{u_A} "0";"4"};
{\ar^{z} "0";"6"};
{\ar^{p_A} "4";"2"};
{\ar_{p_C} "4";"6"};
(-13,-2)*+{}="10";
(-11,-7.5)*+{}="11";
{\ar@/^0.2pc/@{-}^{_{\mu_{A,C}}} "10";"11"};
(13,-2)*+{}="20";
(11,-7.5)*+{}="21";
{\ar@/_0.2pc/@{-}_{_{\pi_{A,C}}} "20";"21"};
\endxy
\end{equation}
\item Dually, we obtain a pair of $2$-simplices as below.
\[
\xy
(0,-7)*+{C}="0";
(-18,7)*+{A}="2";
(0,7)*+{A\am C}="4";
(18,7)*+{C}="6";
{\ar_{z} "2";"0"};
{\ar^{j_C} "4";"0"};
{\ar^{1_C} "6";"0"};
{\ar^(0.4){i_A} "2";"4"};
{\ar_(0.4){i_C} "6";"4"};
(-13,2)*+{}="10";
(-11,7.5)*+{}="11";
{\ar@/_0.2pc/@{-}_{_{\iota_{A,C}}} "10";"11"};
(13,2)*+{}="20";
(11,7.5)*+{}="21";
{\ar@/^0.2pc/@{-}^{_{\nu_{A,C}}} "20";"21"};
\endxy
\]
\end{enumerate}
We will use these symbols in the rest, often abbreviating $\pi_{A,C}$ to $\pi$ and so on.
\end{dfn}

\begin{prop}\label{PropForSplit}
Let $(\ref{2SimplexETA})$ be any $2$-simplex, in which $O\in\C_0$ is a zero object. Then the squares
\begin{equation}\label{SplitExSeq}
{}_A\Nf_C=\ 
\xy
(-7,7)*+{A}="0";
(7,7)*+{A\ti C}="2";
(-7,-7)*+{O}="4";
(7,-7)*+{C}="6";
{\ar^(0.4){u_A} "0";"2"};
{\ar_{i} "0";"4"};
{\ar|*+{_z} "0";"6"};
{\ar^{p_C} "2";"6"};
{\ar_{j} "4";"6"};
(3,7)*+{}="00";
(7,3)*+{}="01";
{\ar@/_0.2pc/@{-}_{^{\pi}} "00";"01"};
(-3,-7)*+{}="10";
(-7,-3)*+{}="11";
{\ar@/_0.2pc/@{-}_{_{\eta}} "10";"11"};
\endxy
\ \ ,\ \ 
{}_A\Nf\ppr_C=\ \xy
(-7,7)*+{A}="0";
(7,7)*+{A\am C}="2";
(-7,-7)*+{O}="4";
(7,-7)*+{C}="6";
{\ar^(0.4){i_A} "0";"2"};
{\ar_{i} "0";"4"};
{\ar|*+{_z} "0";"6"};
{\ar^{j_C} "2";"6"};
{\ar_{j} "4";"6"};
(3,7)*+{}="00";
(7,3)*+{}="01";
{\ar@/_0.2pc/@{-}_{^{\iota}} "00";"01"};
(-3,-7)*+{}="10";
(-7,-3)*+{}="11";
{\ar@/_0.2pc/@{-}_{_{\eta}} "10";"11"};
\endxy
\end{equation}
are exact sequences.
\end{prop}
\begin{proof}
By using the definition of a zero object we may obtain a rectangle $\Rf$ of the following form
\[
\xy
(-16,7)*+{A}="0";
(0,7)*+{A\ti C}="2";
(16,7)*+{A}="4";
(-16,-7)*+{O}="10";
(0,-7)*+{C}="12";
(16,-7)*+{O}="14";
{\ar_(0.4){u_A} "0";"2"};
{\ar_(0.6){p_A} "2";"4"};
{\ar@/^1.0pc/^{1_A} "0";"4"};
{\ar_{i} "0";"10"};
{\ar_{p_C} "2";"12"};
{\ar^{i} "4";"14"};
{\ar^{j} "10";"12"};
{\ar^{} "12";"14"};
{\ar@/_1.0pc/_{1_O} "10";"14"};
\endxy
\]
in which $\Rf_{\mathrm{left}}=\Nf$, 
and $\Rf_{\mathrm{right}}$ is a pull-back. As in Example~\ref{ExTrivExSeq} we know that $\Rf_{\mathrm{right}}$ is an exact sequence, thus it follows that $\Nf$ is also a pull-back and $p_C$ is egressive. This shows that $\Nf$ is an exact sequence. Dually for $\Nf\ppr$.
\end{proof}

\begin{rem}
The above definition of ${}_A\Nf_C$ and ${}_A\Nf\ppr_C$ depends on the choice of $(\ref{2SimplexETA})$ and simplices taken. Later we will introduce an equivalence relation for exact sequences in Section~\ref{Section_EquivalenceRelation}, to show that their equivalence classes ${}_A\und{\Nf}_C, {}_A\und{\Nf\ppr}_C$ do not depend on these choices, and moreover satisfy ${}_A\und{\Nf}_C={}_A\und{\Nf\ppr}_C$.
\end{rem}

\begin{lem}\label{LemRelSplit}
Let ${}_A\Sf_C=\SQ{\eta}{\xi}$ be any exact sequence as in $(\ref{ExSeq})$. Let ${}_A\Nf_C=\SQ{\eta}{\pi},{}_A\Nf\ppr_C=\SQ{\eta}{\iota}$ be as in Proposition~\ref{PropForSplit}, obtained by using the $2$-simplex $\eta$. The following holds.
\begin{enumerate}
\item If $\ovl{x}$ is a split monomorphism in $h\C$, then there is a $3$-simplex
\[
\vp=\ 
\xy
(-3,7)*+{B}="0";
(-12,-6)*+{A}="2";
(10,-4)*+{C}="4";
(2,-16)*+{A\ti C}="6";
{\ar^{x} "2";"0"};
{\ar^{y} "0";"4"};
{\ar_{} "0";"6"};
{\ar_{}|!{"0";"6"}\hole "2";"4"};
{\ar_{u_A} "2";"6"};
{\ar_{p_C} "6";"4"};
\endxy \]
such that $d_1(\vp)=\pi_{A,C}$ and $d_2(\vp)=\xi$.
\item Dually, if $\ovl{y}$ is a split epimorphism in $h\C$, then there is a $3$-simplex $\psi$ such that $d_1(\psi)=\xi$ and $d_2(\psi)=\iota_{A,C}$.
\end{enumerate}
\end{lem}
\begin{proof}
Since {\rm (2)} can be shown dually, it is enough to show {\rm (1)}. Suppose that $\ovl{x}$ is a split monomorphism in $h\C$. This means that there is a $2$-simplex $\om$ of the following form.
\[ \TwoSP{A}{B}{A}{x}{r}{1_A}{\om} \]
Then by the definition of product, we obtain a pair of $2$-simplices as below.
\begin{equation}\label{PAeKS1}
\xy
(0,7)*+{B}="0";
(-18,-7)*+{A}="2";
(0,-7)*+{A\ti C}="4";
(18,-7)*+{C}="6";
{\ar_{r} "0";"2"};
{\ar^{b} "0";"4"};
{\ar^{y} "0";"6"};
{\ar^{p_A} "4";"2"};
{\ar_{p_C} "4";"6"};
(-13,-2)*+{}="10";
(-11,-7.5)*+{}="11";
{\ar@/^0.2pc/@{-}^{} "10";"11"};
(13,-2)*+{}="20";
(11,-7.5)*+{}="21";
{\ar@/_0.2pc/@{-}_{} "20";"21"};
\endxy
\end{equation}
Also, $\om$ and $\xi$ gives the following pair of $2$-simplices.
\begin{equation}\label{PAeKS2}
\xy
(0,7)*+{A}="0";
(-18,-7)*+{A}="2";
(0,-7)*+{B}="4";
(18,-7)*+{C}="6";
{\ar_{1_A} "0";"2"};
{\ar^{x} "0";"4"};
{\ar^{z} "0";"6"};
{\ar^{r} "4";"2"};
{\ar_{y} "4";"6"};
(-13,-2)*+{}="10";
(-11,-7.5)*+{}="11";
{\ar@/^0.2pc/@{-}^{\om} "10";"11"};
(13,-2)*+{}="20";
(11,-7.5)*+{}="21";
{\ar@/_0.2pc/@{-}_{\xi} "20";"21"};
\endxy
\end{equation}
By the definition of product, $(\ref{PAeKS1})$ and $(\ref{PAeKS2})$ give a pair of $3$-simplices as below,
\[ 
\vp\ppr=\ 
\xy
(-2,7)*+{A}="0";
(-12,-6)*+{A}="2";
(10,-4)*+{B}="4";
(2,-16)*+{A\ti C}="6";
{\ar_{1_A} "0";"2"};
{\ar^{x} "0";"4"};
{\ar_(0.7){u_A} "0";"6"};
{\ar_(0.3){r}|!{"0";"6"}\hole "4";"2"};
{\ar^{p_A} "6";"2"};
{\ar^{b} "4";"6"};
\endxy\ \ ,\ \ \vp=\ 
\xy
(2,7)*+{A}="0";
(12,-6)*+{C}="2";
(-10,-4)*+{B}="4";
(-2,-16)*+{A\ti C}="6";
{\ar^{z} "0";"2"};
{\ar_{x} "0";"4"};
{\ar^(0.7){u_A} "0";"6"};
{\ar^(0.3){y}|!{"0";"6"}\hole "4";"2"};
{\ar_{p_C} "6";"2"};
{\ar_{b} "4";"6"};
\endxy
\]
such that $d_3(\vp)=d_3(\vp\ppr)$ holds and the pairs $(d_k(\vp\ppr),d_0(\vp))$ $(k=0,1,2)$ agree with $(\ref{PAeKS1}),(\ref{ProdAC}),(\ref{PAeKS2})$, respectively. This $\vp$ satisfies the stated properties.
\end{proof}

\begin{prop}\label{PropRelSplit}
Let  $(\ref{2SimplexETA})$ be any $2$-simplex in which $O\in\C_0$ is a zero object, and let ${}_A\Nf_C,{}_A\Nf\ppr_C$ be as in Proposition~\ref{PropForSplit}.
Then there exists a $3$-simplex
\[
\aleph_{A,C}=\ 
\xy
(-3,8)*+{A\am C}="0";
(-15,-9)*+{A}="2";
(13,-5)*+{C}="4";
(2,-19)*+{A\ti C}="6";
{\ar^{i_A} "2";"0"};
{\ar^{j_C} "0";"4"};
{\ar_(0.37){\ups_{A,C}} "0";"6"};
{\ar_(0.4){z}|!{"0";"6"}\hole "2";"4"};
{\ar_{u_A} "2";"6"};
{\ar_{p_C} "6";"4"};
\endxy
\]
such that $d_1(\aleph_{A,C})=\pi_{A,C}$, $d_2(\aleph_{A,C})=\iota_{A,C}$ and that $\ovl{\ups_{A,C}}$ is an isomorphism in $h\C$.
\end{prop}
\begin{proof}
Take any $2$-simplex as below.
\[
\xy
(-7,7)*+{C}="0";
(-7,-7)*+{O}="4";
(7,-7)*+{A}="6";
{\ar_{i\ppr} "0";"4"};
{\ar^{z\ppr} "0";"6"};
{\ar_{j\ppr} "4";"6"};
(-3,-7)*+{}="10";
(-7,-3)*+{}="11";
{\ar@/_0.2pc/@{-}_{_{\eta\ppr}} "10";"11"};
\endxy
\]
By the definition of a coproduct, we obtain a pair of $2$-simplices
\[
\xy
(0,-7)*+{A}="0";
(-18,7)*+{A}="2";
(0,7)*+{A\am C}="4";
(18,7)*+{C}="6";
{\ar_{1_A} "2";"0"};
{\ar^{j_A} "4";"0"};
{\ar^{z\ppr} "6";"0"};
{\ar^(0.4){i_A} "2";"4"};
{\ar_(0.4){i_C} "6";"4"};
(-13,2)*+{}="10";
(-11,7.5)*+{}="11";
{\ar@/_0.2pc/@{-}_{} "10";"11"};
(13,2)*+{}="20";
(11,7.5)*+{}="21";
{\ar@/^0.2pc/@{-}^{} "20";"21"};
\endxy
\]
similarly as in Definition~\ref{DefForSplit} {\rm (2)}.
Then we have $\ovl{j_A}\ci\ovl{i_A}=\id_A$ in $h\C$, hence by Lemma~\ref{LemRelSplit} we obtain a $2$-simplex $\aleph_{A,C}$ with $d_1(\aleph_{A,C})=\pi_{A,C},d_2(\aleph_{A,C})=\iota_{A,C}$. As in the proof of Lemma~\ref{LemRelSplit}, morphism $\ups_{A,C}=d_0d_3(\aleph_{A,C})$ is obtained as a morphism appearing in the following
\[
\xy
(0,8)*+{A\am C}="0";
(-18,-7)*+{A}="2";
(0,-7)*+{A\ti C}="4";
(18,-7)*+{C}="6";
{\ar_{j_A} "0";"2"};
{\ar_{\ups_{A,C}} "0";"4"};
{\ar^{j_C} "0";"6"};
{\ar^{p_A} "4";"2"};
{\ar_{p_C} "4";"6"};
(-13,-2)*+{}="10";
(-11,-7.5)*+{}="11";
{\ar@/^0.2pc/@{-}^{} "10";"11"};
(13,-2)*+{}="20";
(11,-7.5)*+{}="21";
{\ar@/_0.2pc/@{-}_{} "20";"21"};
\endxy
\]
by construction. In the homotopy category $h\C$, morphism $\vfr=\ovl{\ups_{A,C}}$ is the unique one which makes
\begin{equation}\label{UniqueComm}
\xy
(0,8)*+{A\am C}="0";
(-18,-7)*+{A}="2";
(0,-7)*+{A\ti C}="4";
(18,-7)*+{C}="6";
(12,2)*+{}="3";
(-12,2)*+{}="5";
{\ar_{\ovl{j_A}} "0";"2"};
{\ar_{\vfr} "0";"4"};
{\ar^{\ovl{j_C}} "0";"6"};
{\ar^{\ovl{p_A}} "4";"2"};
{\ar_{\ovl{p_C}} "4";"6"};
{\ar@{}|\circlearrowright "3";"4"};
{\ar@{}|\circlearrowright "4";"5"};
\endxy
\end{equation}
commutative, since the bottom row gives a product of $A$ and $C$ in $h\C$. By the additivity of $h\C$, this forces $\vfr$ to be an isomorphism.
\end{proof}

\begin{rem}\label{Remv}
In the sequel we will continue to use the symbol $\ups_{A,C}$ for a morphism which gives the unique (iso)morphism $\vfr=\ovl{\ups_{A,C}}$ in $h\C$ making $(\ref{UniqueComm})$ commutative.
\end{rem}

\begin{cor}\label{CorRelSplit}
Let ${}_A\Sf_C=\SQ{\eta}{\xi}, {}_A\Nf_C=\SQ{\eta}{\pi}$ and ${}_A\Nf\ppr_C=\SQ{\eta}{\iota}$ be as in Lemma~\ref{LemRelSplit}. Then, the following are equivalent.
\begin{itemize}
\item[{\rm (i)}] There exists a morphism
\begin{equation}\label{MorphSN1}
{}_a\Cf_c=\ 
\xy
(-9,10)*+{A}="0";
(9,10)*+{B}="2";
(2,4)*+{O}="4";
(20,4)*+{C}="6";
(-9,-8)*+{A}="10";
(9,-8)*+{A\ti C}="12";
(2,-14)*+{O}="14";
(20,-14)*+{C}="16";
{\ar^{x} "0";"2"};
{\ar_{i} "0";"4"};
{\ar^{y} "2";"6"};
{\ar_(0.3){j} "4";"6"};
%
{\ar_{a} "0";"10"};
{\ar^(0.7){b}|!{(10,4);(14,4)}\hole "2";"12"};
{\ar_(0.3){o} "4";"14"};
{\ar^{c} "6";"16"};
{\ar^(0.3){u_A}|!{(2,-4);(2,-8)}\hole "10";"12"};
{\ar_{i} "10";"14"};
{\ar^(0.6){p_C} "12";"16"};
{\ar_{j} "14";"16"};
%
\endxy\ \ \co\Sf\to\Nf
\end{equation}
such that $\ovl{a}=\id_A$, $\ovl{c}=\id_C$.
\item[{\rm (i)$\ppr$}] There exists a morphism ${}_{a\ppr}\Cf\ppr_{c\ppr}\co\Nf\ppr\to\Sf$ such that $\ovl{a\ppr}=\id_A$, $\ovl{c\ppr}=\id_C$.
\item[{\rm (ii)}] There exists a morphism ${}_{1_A}\Cf_{1_C}\co\Sf\to\Nf$ in which $b=\Cf|_{\{(0,1)\}\ti\Delta^1}$ is a homotopy equivalence.
\item[{\rm (ii)$\ppr$}] There exists a morphism ${}_{1_A}\Cf\ppr_{1_C}\co\Nf\ppr\to\Sf$ in which $b\ppr=\Cf\ppr|_{\{(0,1)\}\ti\Delta^1}$ is a homotopy equivalence.
\item[{\rm (iii)}] $\ovl{x}$ is a split monomorphism in $h\C$.
\item[{\rm (iii)$\ppr$}] $\ovl{y}$ is a split epimorphism in $h\C$.
\end{itemize}
Remark that the conditions {\rm (iii)} and {\rm (iii)$\ppr$} are independent of the choices of $\Nf$ and $\Nf\ppr$.
\end{cor}
\begin{proof}
Equivalence $\mathrm{(ii)}\EQ\mathrm{(ii)\ppr}$ follows from Proposition~\ref{PropComposeCubes}, Propositions~\ref{PropForSym} and \ref{PropRelSplit}. Since $\mathrm{(i)\ppr}\EQ\mathrm{(ii)\ppr}\EQ\mathrm{(iii)\ppr}$ can be shown dually, it is enough to show $\mathrm{(i)}\EQ\mathrm{(ii)}\EQ\mathrm{(iii)}$.
Remark that $\mathrm{(ii)}\tc\mathrm{(i)}$ is trivial, and that $\mathrm{(iii)}\tc\mathrm{(i)}$ follows from in Lemma~\ref{LemRelSplit}. 
Also, $\mathrm{(i)}\tc\mathrm{(iii)}$ is immediate. Indeed if there is a morphism $(\ref{MorphSN1})$ as in {\rm (i)}, then we have $\ovl{p_A}\ci\ovl{b}\ci\ovl{x}=\ovl{p_A}\ci\ovl{u_A}\ci\ovl{a}=\id_A$, which in particular means that $\ovl{x}$ is a split monomorphism.

It remains to show $\mathrm{(i)}\tc\mathrm{(ii)}$.
Suppose that we have a morphism $(\ref{MorphSN1})$ satisfying $\ovl{a}=\id_A$ and $\ovl{c}=\id_C$. By Corollary~\ref{CorComposeCubes}, we may assume $a=1_A$ and $c=1_C$ from the beginning. Let us show that $\ovl{b}$ becomes an isomorphism in $h\C$. Remark that the $2$-simplex $\pi_{A,C}$ in $(\ref{ProdAC})$ induces a split short exact sequence
\[ A\ov{\ovl{u_A}}{\lra}A\ti C\ov{\ovl{p_C}}{\lra}C \]
in $h\C$. Replacing it by the isomorphic one
$A\ov{\left[\bsm1\\0\esm\right]}{\lra}A\oplus C\ov{[0\ 1]}{\lra}C$,
we see that $\Cf$ induces the following commutative diagram in $h\C$,
\begin{equation}\label{PAeKSComm}
\xy
(-14,6)*+{A}="2";
(0,6)*+{B}="4";
(14,6)*+{C}="6";
(-14,-6)*+{A}="12";
(0,-6)*+{A\oplus C}="14";
(14,-6)*+{C}="16";
{\ar^{\ovl{x}} "2";"4"};
{\ar^{\ovl{y}} "4";"6"};
{\ar@{=} "2";"12"};
{\ar^{\left[\bsm \rfr\\\ovl{y}\esm\right]} "4";"14"};
{\ar@{=} "6";"16"};
{\ar_(0.4){\left[\bsm1\\0\esm\right]} "12";"14"};
{\ar_(0.6){[0\ 1]} "14";"16"};
{\ar@{}|\circlearrowright "2";"14"};
{\ar@{}|\circlearrowright "4";"16"};
\endxy
\end{equation}
where we put $\rfr=\ovl{p_A}\ci\ovl{b}\in (h\C)(B,A)$. It suffices to show that $\left[\bsm\rfr\\\ovl{y}\esm\right]$ is an isomorphism.

Since the top row of $(\ref{PAeKSComm})$ is a weak cokernel sequence, there is some $\wfr\in(h\C)(C,B)$ such that $\id_B-\ovl{x}\ci\rfr=\wfr\ci\ovl{y}$. This shows that $[\ovl{x}\ \wfr]\co A\oplus C\to B$ gives a left inverse of $\left[\bsm\rfr\\\ovl{y}\esm\right]$. On the other hand, since the right square of $(\ref{PAeKSComm})$ is a weak push-out by the dual of Lemma~\ref{LemForSym} {\rm (1)}, it can be verified that $\left[\bsm\rfr\\\ovl{y}\esm\right]$ is an epimorphism, by a straightforward argument in the ordinary additive category $h\C$. Thus $\left[\bsm\rfr\\\ovl{y}\esm\right]$ is an isomorphism.
\end{proof}

\begin{dfn}\label{DefSplit}
An exact sequence $\Sf$ is said to \emph{split}, if it satisfies condition {\rm (iii)} (or other equivalent conditions) in Corollary~\ref{CorRelSplit}.
\end{dfn}

\begin{cor}\label{CorSplitEachOther}
Let $A,C\in\C_0$ be any pair of objects, and let ${}_A\Sf_C,{}_A\Sf\ppr_C$ be exact sequences starting from $A$ and ending in $C$. Suppose that there exists a morphism ${}_{1_A}\Cf_{1_C}\co\Sf\to\Sf\ppr$. If one of ${}_A\Sf_C,{}_A\Sf\ppr_C$ splits, then so does the other.
\end{cor}
\begin{proof}
This can be easily checked by using {\rm (iii)} and {\rm (iii)$\ppr$} in Corollary~\ref{CorRelSplit}.
\end{proof}

\subsection{Equivalence of exact sequences}\label{Section_EquivalenceRelation}

\begin{lem}\label{LemInvert}
Let ${}_A\Sf_C$ be any exact sequence as in $(\ref{ExSeq})$. Then as $x\sas\Sf$ of Lemma~\ref{LemExSeqPO}, we may take a splitting one. Dually, $y\uas\Sf$ can be chosen to split.
\end{lem}
\begin{proof}
By \cite[Lemma~4.6]{B1}, we have a pull-back of the following form.
\[
\xy
(-7,7)*+{A\ti B}="0";
(7,7)*+{B}="2";
(-7,-7)*+{B}="4";
(7,-7)*+{C}="6";
{\ar^(0.7){p_B} "0";"2"};
{\ar_{} "0";"4"};
{\ar^{y} "2";"6"};
{\ar_{y} "4";"6"};
\endxy
\]
If we use this pull-back, the resulting $y\uas\Sf$ becomes of the form
\[
\xy
(-7,7)*+{A}="0";
(7,7)*+{A\ti B}="2";
(-7,-7)*+{O}="4";
(7,-7)*+{B}="6";
{\ar^{} "0";"2"};
{\ar_{} "0";"4"};
{\ar^{p_B} "2";"6"};
{\ar_{} "4";"6"};
\endxy,
\]
which obviously satisfies {\rm (iii)$\ppr$} of Corollary~\ref{CorRelSplit}. Dually for $x\sas\Sf$.
\end{proof}

\begin{rem}\label{RemInvert}
Later in Corollary~\ref{CorFunctoriality} {\rm (1)} we will see that all $x\sas\Sf$ become \emph{equivalent}, and hence in particular always split, independently of the choices made. Similarly for $y\uas\Sf$.
\end{rem}

\begin{prop}\label{PropInvert}
For any morphism of exact sequences ${}_{1_A}\Cf_{1_C}\co\Sf\to\Sf\ppr$, the morphism $b=\Cf|_{\{(0,1)\}\ti\Delta^1}$ becomes a homotopy equivalence.

Thus the assumption on $b$ in Proposition~\ref{PropForSym} is always satisfied, hence there exists a morphism in the opposite direction ${}_{1_A}\Cf\ppr_{1_C}\co\Sf\ppr\to\Sf$.
\end{prop}
\begin{proof}
Label $\Cf$ as in $(\ref{CubeForSym})$.
Take $x\sas\Sf$ and $x\sas\Sf\ppr$. By Lemma~\ref{LemInvert}, we may choose $x\sas\Sf$ to split. There are morphisms
\[ {}_x\Df_{1_C}\co\Sf\to x\sas\Sf
\quad\text{and}\quad
{}_x\Df\ppr_{1_C}\co\Sf\ppr\to x\sas\Sf\ppr \]
as in Lemma~\ref{LemExSeqPO} {\rm (1)}. By Proposition~\ref{PropComposeCubes} applied to $\Sf\ov{{}_{1_A}\Cf_{1_C}}{\lra}\Sf\ppr\ov{{}_x(\Df\ppr)_{1_C}}{\lra}x\sas\Sf\ppr$, we obtain a morphism ${}_{x}\Cf\ppr_{1_C}\co\Sf\to x\sas\Sf\ppr$. 
Thus by Lemma~\ref{LemExSeqPO} {\rm (2)}, we obtain a morphism ${}_{1_B}\Cf\pprr_{1_C}\co x\sas\Sf\to x\sas\Sf\ppr$. 
Since $x\sas\Sf$ splits, it follows that $x\sas\Sf\ppr$ also splits by Corollary~\ref{CorSplitEachOther}. 

Let us show that $\ovl{b}$ is an isomorphism in $h\C$. If we label simplices of $\Df\ppr$ as below,
\[
\xy
(-8,9)*+{A}="0";
(8,9)*+{B\ppr}="2";
(1,3)*+{O\ppr}="4";
(17,3)*+{C}="6";
(-8,-7)*+{B}="10";
(8,-7)*+{M}="12";
(1,-13)*+{O\ppr}="14";
(17,-13)*+{C}="16";
{\ar^{x\ppr} "0";"2"};
{\ar_{i\ppr} "0";"4"};
{\ar^{y\ppr} "2";"6"};
{\ar_(0.3){j\ppr} "4";"6"};
%
{\ar_{x} "0";"10"};
{\ar^(0.7){b_M}|!{(9,3);(13,3)}\hole "2";"12"};
{\ar_(0.3){1_{O\ppr}} "4";"14"};
{\ar^{1_C} "6";"16"};
{\ar^(0.3){x_M}|!{(1,-3);(1,-7)}\hole "10";"12"};
{\ar_{} "10";"14"};
{\ar^{y_M} "12";"16"};
{\ar_{j\ppr} "14";"16"};
%
\endxy
\]
then we have the following commutative diagram
\[
\xy
(-12,6)*+{A}="2";
(0,6)*+{B\ppr}="4";
(12,6)*+{C}="6";
(-12,-6)*+{B}="12";
(0,-6)*+{M}="14";
(12,-6)*+{C}="16";
{\ar^{\ovl{x\ppr}} "2";"4"};
{\ar^{\ovl{y\ppr}} "4";"6"};
{\ar_{\ovl{x}} "2";"12"};
{\ar^{\ovl{b_M}} "4";"14"};
{\ar@{=} "6";"16"};
{\ar_{\ovl{x_M}} "12";"14"};
{\ar_{\ovl{y_M}} "14";"16"};
{\ar@{}|\circlearrowright "2";"14"};
{\ar@{}|\circlearrowright "4";"16"};
\endxy
\]
in $h\C$. Since $x\sas\Sf\ppr$ splits, there is some $\rfr\in(h\C)(M,B)$ such that $\rfr\ci\ovl{x_M}=\id_B$. Since $A\ov{\ovl{x}}{\lra}B\ov{\ovl{y}}{\lra}C$ is a weak cokernel sequence, there exists $\dfr\in(h\C)(C,B)$ such that
$\id_B-\rfr\ci\ovl{b_M}\ci\ovl{b}=\dfr\ci\ovl{y}$.
If we put $\bfr=\rfr\ci\ovl{b_M}+\dfr\ci\ovl{y\ppr}\in(h\C)(B\ppr,B)$, then it satisfies
\[ \bfr\ci\ovl{b}=\rfr\ci\ovl{b_M}\ci\ovl{b}+\dfr\ci\ovl{y\ppr}\ci\ovl{b}=\id_B, \]
which shows that $\bfr$ is a left inverse of $\ovl{b}$.

A dual argument using $y^{\prime\ast}\Sf$ and $y^{\prime\ast}\Sf\ppr$ shows that $\ovl{b}$ has also a right inverse, and thus it is an isomorphism.
\end{proof}

\begin{dfn}\label{DefEquivExSeq}
Let $A,C\in\C_0$ be any pair of objects. Let
\begin{equation}\label{TwoExSeqEquiv}
{}_A\Sf_C=\ 
\xy
(-7,7)*+{A}="0";
(7,7)*+{B}="2";
(-7,-7)*+{O}="4";
(7,-7)*+{C}="6";
{\ar^{x} "0";"2"};
{\ar_{i} "0";"4"};
{\ar^{y} "2";"6"};
{\ar_{j} "4";"6"};
\endxy
\ ,\ \ 
{}_A\Sf\ppr_C=\ 
\xy
(-7,7)*+{A}="0";
(7,7)*+{B\ppr}="2";
(-7,-7)*+{O\ppr}="4";
(7,-7)*+{C}="6";
{\ar^{x\ppr} "0";"2"};
{\ar_{i\ppr} "0";"4"};
{\ar^{y\ppr} "2";"6"};
{\ar_{j\ppr} "4";"6"};
\endxy
\end{equation}
be two exact sequences starting from $A$ and ending in $C$. We say \emph{$\Sf$ is equivalent to $\Sf\ppr$} if there is a morphism of exact sequences ${}_{1_A}\Cf_{1_C}\co\Sf\to\Sf\ppr$. In this case we write $\Sf\sim\Sf\ppr$.
By Corollary~\ref{CorComposeCubes}, this is equivalent to the existence of a morphism ${}_a\Cf\ppr_c\co \Sf\to\Sf\ppr$ satisfying $\ovl{a}=\id_A$ and $\ovl{c}=\id_C$.
\end{dfn}

\begin{prop}\label{PropEquivExSeq}
For any $A,C\in\C_0$, relation $\sim$ is an equivalence relation.
\end{prop}
\begin{proof}
Reflexive law is obvious, since there is always $\If_{\Sf}\co\Sf\to\Sf$ for any exact sequence $\Sf$ as in Example~\ref{ExTrivCubes}. Symmetric law and transitive law follow from Propositions~\ref{PropInvert} and \ref{PropComposeCubes}, respectively.
\end{proof}

\subsection{(Co)products of exact sequences}
We begin with the following consequence of \cite[Corollary~5.1.2.3]{L1}.
Let
\begin{equation}\label{PairOfExSeqForProd}
\Sf_k=\ 
\xy
(-7,7)*+{A_k}="0";
(7,7)*+{B_k}="2";
(-7,-7)*+{O_k}="4";
(7,-7)*+{C_k}="6";
{\ar^{x_k} "0";"2"};
{\ar_{i_k} "0";"4"};
{\ar|*+{_{z_k}} "0";"6"};
{\ar^{y_k} "2";"6"};
{\ar_{j_k} "4";"6"};
(3,7)*+{}="00";
(7,3)*+{}="01";
{\ar@/_0.2pc/@{-}_(0.4){^{\xi_k}} "00";"01"};
(-3,-7)*+{}="10";
(-7,-3)*+{}="11";
{\ar@/_0.2pc/@{-}_(0.4){_{\eta_k}} "10";"11"};
\endxy
\quad(k=1,2)
\end{equation}
be any pair of squares. By \cite[Corollary~5.1.2.3]{L1}, a product $\Sf_1\ti\Sf_2$ of $\Sf_1,\Sf_2$ in $\Fun(\Delta^1\ti\Delta^1,\C)$ can be obtained by taking products on each vertex of $\Delta^1\ti\Delta^1$ as follows,
\begin{equation}\label{ProdExSeqPair}
\Sf_1\ti\Sf_2=\ 
\xy
(-11,11)*+{A_1\ti A_2}="0";
(11,11)*+{B_1\ti B_2}="2";
(-11,-11)*+{O_1\ti O_2}="4";
(11,-11)*+{C_1\ti C_2}="6";
{\ar^{x_1\ti x_2} "0";"2"};
{\ar_{i_1\ti i_2} "0";"4"};
{\ar|*+{_{z_1\ti z_2}} "0";"6"};
{\ar^{y_1\ti y_2} "2";"6"};
{\ar_{j_1\ti j_2} "4";"6"};
(3,11)*+{}="00";
(11,3)*+{}="01";
{\ar@/_0.2pc/@{-}_{^{\xi_1\ti \xi_2}} "00";"01"};
(-3,-11)*+{}="10";
(-11,-3)*+{}="11";
{\ar@/_0.2pc/@{-}_{_{\eta_1\ti \eta_2}} "10";"11"};
\endxy
\end{equation}
together with projections $\Qf_k\in\Fun(\Delta^1\ti\Delta^1,\C)_1(\Sf_1\ti\Sf_2,\Sf_k)$ $(k=1,2)$. In particular, for any square $\Sf$ and any pair of morphisms $\Cf_k\in\Fun(\Delta^1\ti\Delta^1,\C)_1(\Sf,\Sf_k)$ $(k=1,2)$, we obtain a morphism $(\Cf_1,\Cf_2)$ and a pair of $2$-simplices
\[
\xy
(0,12)*+{\Sf}="0";
(-22,-7)*+{\Sf_1}="2";
(0,-7)*+{\Sf_1\ti \Sf_2}="4";
(22,-7)*+{\Sf_2}="6";
{\ar_{\Cf_1} "0";"2"};
{\ar|*+{_{(\Cf_1,\Cf_2)}} "0";"4"};
{\ar^{\Cf_2} "0";"6"};
{\ar^{\Qf_1} "4";"2"};
{\ar_{\Qf_2} "4";"6"};
(-16,-0.5)*+{}="10";
(-13,-7.5)*+{}="11";
{\ar@/^0.2pc/@{-}^{} "10";"11"};
(16,-0.5)*+{}="20";
(13,-7.5)*+{}="21";
{\ar@/_0.2pc/@{-}_{} "20";"21"};
\endxy
\]
in $\Fun(\Delta^1\ti\Delta^1,\C)$.
Dually, we have a coproduct
\[
\Sf_1\am\Sf_2=\ 
\xy
(-11,11)*+{A_1\am A_2}="0";
(11,11)*+{B_1\am B_2}="2";
(-11,-11)*+{O_1\am O_2}="4";
(11,-11)*+{C_1\am C_2}="6";
{\ar^{x_1\am x_2} "0";"2"};
{\ar_{i_1\am i_2} "0";"4"};
{\ar|*+{_{z_1\am z_2}} "0";"6"};
{\ar^{y_1\am y_2} "2";"6"};
{\ar_{j_1\am j_2} "4";"6"};
(3,11)*+{}="00";
(11,3)*+{}="01";
{\ar@/_0.2pc/@{-}_{^{\xi_1\am \xi_2}} "00";"01"};
(-3,-11)*+{}="10";
(-11,-3)*+{}="11";
{\ar@/_0.2pc/@{-}_{_{\eta_1\am \eta_2}} "10";"11"};
\endxy.
\]
For any square $\Sf$ and any $\Cf_k\in\Fun(\Delta^1\ti\Delta^1,\C)_1(\Sf_k,\Sf)$ $(k=1,2)$, we have $\Cf_1\cup\Cf_2\in\Fun(\Delta^1\ti\Delta^1,\C)_1(\Sf_1\am\Sf_2,\Sf)$ with a property dual to the above.

Again by \cite[Corollary~5.1.2.3]{L1}, if $\Sf_1,\Sf_2$ are pull-backs, then so is $\Sf_1\ti\Sf_2$. Similarly, if $\Sf_1,\Sf_2$ are push-outs, then so is $\Sf_1\am\Sf_2$.

\begin{prop}\label{PropPCrod}
If $\Sf_1,\Sf_2$ are exact sequences, then so are $\Sf_1\ti\Sf_2$ and $\Sf_1\am\Sf_2$.
\end{prop}
\begin{proof}
Let $\Sf_1,\Sf_2$ be as in $(\ref{PairOfExSeqForProd})$. Since egressive morphisms are stable under taking pull-backs, it follows that $y_1\ti 1_{A_2}$ and $1_{B_1}\ti y_2$ are egressive morphisms. Since there is a $2$-simplex
\[
\xy
(-12,11)*+{A_1\ti A_2}="0";
(-12,-6)*+{B_1\ti A_2}="4";
(12,-6)*+{B_1\ti B_2}="6";
{\ar_{y_1\ti 1_{A_2}} "0";"4"};
{\ar_{1_{B_1}\ti y_2} "4";"6"};
{\ar^{y_1\ti y_2} "0";"6"};
(-4,-6.2)*+{}="10";
(-12.2,1.8)*+{}="11";
{\ar@/_0.3pc/@{-}_{_{}} "10";"11"};
\endxy
\]
in $\C$, it follows that $y_1\ti y_2$ is also egressive. Since $j_1\ti j_2$ in $(\ref{ProdExSeqPair})$ is ingressive because $O_1\ti O_2$ is a zero object, we see that $\Sf_1\ti\Sf_2$ is an exact sequence. Dually for $\Sf_1\am\Sf_2$.
\end{proof}

\begin{ex}\label{ExProdExSeq}
Let ${}_A\Sf_C$ be any exact sequence as in $(\ref{ExSeq})$. From $\If_{\Sf}\co\Sf\to\Sf$, we obtain the following morphisms of exact sequences, which we denote by $\nabla_{\Sf}=\If_{\Sf}\cup\If_{\Sf}\co\Sf\am\Sf\to\Sf$ and $\Delta_{\Sf}=(\If_{\Sf},\If_{\Sf})\co\Sf\to\Sf\ti\Sf$, respectively.
\[
\xy
(-11,18)*+{A\am A}="0";
(11,18)*+{B\am B}="2";
(0,8)*+{O\am O}="4";
(22,8)*+{C\am C}="6";
(-11,-2)*+{A}="10";
(11,-2)*+{B}="12";
(0,-12)*+{O}="14";
(22,-12)*+{C}="16";
{\ar^{x\am x} "0";"2"};
{\ar^(0.6){i\am i} "0";"4"};
{\ar^(0.6){y\am y} "2";"6"};
{\ar_{} "4";"6"};
{\ar_{\nabla_A} "0";"10"};
{\ar^(0.75){\nabla_B}|!{(9,8);(13,8)}\hole "2";"12"};
{\ar_(0.3){\nabla_O} "4";"14"};
{\ar^{\nabla_C} "6";"16"};
{\ar_(0.7){x}|!{(0,0);(0,-4)}\hole "10";"12"};
{\ar_{i} "10";"14"};
{\ar_(0.4){y} "12";"16"};
{\ar_{j} "14";"16"};
\endxy
\quad,\quad
\xy
(-11,18)*+{A}="0";
(11,18)*+{B}="2";
(0,8)*+{O}="4";
(22,8)*+{C}="6";
(-11,-2)*+{A\ti A}="10";
(11,-2)*+{B\ti B}="12";
(0,-12)*+{O\ti O}="14";
(22,-12)*+{C\ti C}="16";
{\ar^{x} "0";"2"};
{\ar^(0.6){i} "0";"4"};
{\ar^(0.6){y} "2";"6"};
{\ar^(0.3){j} "4";"6"};
{\ar_{\Delta_A} "0";"10"};
{\ar^(0.75){\Delta_B}|!{(9,8);(13,8)}\hole "2";"12"};
{\ar_(0.25){\Delta_O} "4";"14"};
{\ar^{\Delta_C} "6";"16"};
{\ar_(0.6){}|!{(0,0);(0,-4)}\hole "10";"12"};
{\ar_{i\ti i} "10";"14"};
{\ar_(0.4){y\ti y} "12";"16"};
{\ar_{j\ti j} "14";"16"};
\endxy
\]
\end{ex}

\begin{prop}\label{PropProdExSeq1}
Let
\[
\Cf_k=\ \xy
(-8,9)*+{A_k}="0";
(8,9)*+{B_k}="2";
(1,3)*+{O_k}="4";
(17,3)*+{C_k}="6";
(-8,-7)*+{A_k\ppr}="10";
(8,-7)*+{B_k\ppr}="12";
(1,-13)*+{O_k\ppr}="14";
(17,-13)*+{C_k\ppr}="16";
{\ar^{x_k} "0";"2"};
{\ar_{i_k} "0";"4"};
{\ar^{y_k} "2";"6"};
{\ar_(0.3){j_k} "4";"6"};
%
{\ar_{a_k} "0";"10"};
{\ar^(0.7){b_k}|!{(9,3);(13,3)}\hole "2";"12"};
{\ar_(0.3){o_k} "4";"14"};
{\ar^{c_k} "6";"16"};
{\ar^(0.3){x_k\ppr}|!{(1,-3);(1,-7)}\hole "10";"12"};
{\ar_{i_k\ppr} "10";"14"};
{\ar^{y_k\ppr} "12";"16"};
{\ar_{j_k\ppr} "14";"16"};
%
\endxy
\ \ \co\Sf_k\to\Sf_k\ppr\quad(k=1,2)
\]
be any pair of morphisms of exact sequences.
Then we have the following morphism of exact sequences $\Sf_1\am\Sf_2\to\Sf_1\ppr\ti\Sf_2\ppr$
\begin{equation}\label{MorphUps}
\xy
(-12,20)*+{A_1\am A_2}="0";
(12,20)*+{B_1\am B_2}="2";
(3,10)*+{O_1\am O_2}="4";
(27,10)*+{C_1\am C_2}="6";
(-12,-2)*+{A_1\ppr\ti A_2\ppr}="10";
(12,-2)*+{B_1\ppr\ti B_2\ppr}="12";
(3,-12)*+{O_1\ppr\ti O_2\ppr}="14";
(27,-12)*+{C_1\ppr\ti C_2\ppr}="16";
{\ar^{x_1\am x_2} "0";"2"};
{\ar^(0.6){i_1\am i_2} "0";"4"};
{\ar^(0.6){y_1\am y_2} "2";"6"};
{\ar_(0.6){j_1\am j_2} "4";"6"};
{\ar_{\ups_{a_1,a_2}} "0";"10"};
{\ar^(0.75){\ups_{b_1,b_2}}|!{(9,10);(13,10)}\hole "2";"12"};
{\ar_(0.25){\ups_{o_1,o_2}} "4";"14"};
{\ar^{\ups_{c_1,c_2}} "6";"16"};
{\ar^(0.4){x_1\ppr\ti x_2\ppr}|!{(4,0);(4,-4)}\hole "10";"12"};
{\ar_{i_1\ppr\ti i_2\ppr} "10";"14"};
{\ar^{y_1\ppr\ti y_2\ppr} "12";"16"};
{\ar_{j_1\ppr\ti j_2\ppr} "14";"16"};
\endxy
\end{equation}
in which $\ups_{a_1,a_2}$ gives the unique morphism $\ovl{\ups_{a_1,a_2}}=\ovl{a_1}\oplus\ovl{a_2}\in(h\C)(A_1\oplus A_2,A_1\ppr\oplus A_2\ppr)$ in the homotopy category, and similarly for $\ups_{b_1,b_2},\ups_{c_1,c_2},\ups_{o_1,o_2}$.
\end{prop}
\begin{proof}
As in Example~\ref{ExTrivCubes}, we have morphisms
\[ \If_k\ppr=\If_{\Sf\ppr_k}\co\Sf\ppr_k\to\Sf\ppr_k,\quad\ \Cf_k\ppr\co\Sf\ppr_k\to\Of,\quad\Cf\pprr_k\co\Of\to\Sf\ppr_k. \]
By Proposition~\ref{PropComposeCubes} applied to $\Sf\ppr_1\ov{\Cf_1\ppr}{\lra}\Of\ov{\Cf_2\pprr}{\lra}\Sf\ppr_2$, we obtain a morphism $\Zf_1\ppr\co\Sf_1\ppr\to\Sf_2\ppr$. Similarly for $\Zf_2\ppr\co\Sf_2\ppr\to\Sf_1\ppr$.

Then we have morphisms
\[ (\If_1\ppr,\Zf_1\ppr)\co\Sf_1\ppr\to\Sf_1\ppr\ti\Sf_2\ppr\quad\text{and}\quad(\Zf_2\ppr,\If_2\ppr)\co\Sf_2\ppr\to\Sf_1\ppr\ti\Sf_2\ppr. \]
By Proposition~\ref{PropComposeCubes} applied to $\Sf_1\ov{\Cf_1}{\lra}\Sf\ppr_1\ov{(\If_1\ppr,\Zf_1\ppr)}{\lra}\Sf\ppr_1\ti\Sf\ppr_2$ and $\Sf_2\ov{\Cf_2}{\lra}\Sf\ppr_2\ov{(\Zf_2\ppr,\If_2\ppr)}{\lra}\Sf\ppr_1\ti\Sf\ppr_2$ respectively, we obtain morphisms $\Ff_1\co\Sf_1\to\Sf\ppr_1\ti\Sf\ppr_2$ and $\Ff_2\co\Sf_2\to\Sf\ppr_1\ti\Sf\ppr_2$. 
Then we obtain a morphism $\Ff_1\cup\Ff_2\co\Sf_1\am\Sf_2\to\Sf\ppr_1\ti\Sf\ppr_2$ as in $(\ref{MorphUps})$ with the stated property.
\end{proof}

\begin{cor}\label{CorProdExSeq1}
For any pair of squares $\Sf_k$ $(k=1,2)$ as in $(\ref{PairOfExSeqForProd})$, we obtain the following morphism $\Sf_1\am\Sf_2\to\Sf_1\ti\Sf_2$
\[
\xy
(-12,20)*+{A_1\am A_2}="0";
(12,20)*+{B_1\am B_2}="2";
(3,10)*+{O_1\am O_2}="4";
(27,10)*+{C_1\am C_2}="6";
(-12,-2)*+{A_1\ti A_2}="10";
(12,-2)*+{B_1\ti B_2}="12";
(3,-12)*+{O_1\ti O_2}="14";
(27,-12)*+{C_1\ti C_2}="16";
{\ar^{x_1\am x_2} "0";"2"};
{\ar^(0.6){i_1\am i_2} "0";"4"};
{\ar^(0.6){y_1\am y_2} "2";"6"};
{\ar_(0.6){j_1\am j_2} "4";"6"};
{\ar_{\ups_{A_1,A_2}} "0";"10"};
{\ar^(0.75){\ups_{B_1,B_2}}|!{(9,10);(13,10)}\hole "2";"12"};
{\ar_(0.25){\ups_{O_1,O_2}} "4";"14"};
{\ar^{\ups_{C_1,C_2}} "6";"16"};
{\ar^(0.4){x_1\ti x_2}|!{(4,0);(4,-4)}\hole "10";"12"};
{\ar_{i_1\ti i_2} "10";"14"};
{\ar^{y_1\ti y_2} "12";"16"};
{\ar_{j_1\ti j_2} "14";"16"};
\endxy
\]
in which, the vertical morphisms are homotopy equivalences as in Remark~\ref{Remv}.
\end{cor}
\begin{proof}
This follows from Proposition~\ref{PropProdExSeq1} applied to $\If_{\Sf_1}\co\Sf_1\to\Sf_1$ and $\If_{\Sf_2}\co\Sf_2\to\Sf_2$.
\end{proof}

\section{Extriangulated structure on the homotopy category}\label{Section_ExtriangulatedStructure}

By definition, an extriangulated category is an additive category equipped with a structure $(\E,\sfr)$ satisfying some conditions. For the convenience of the reader, we write them down in the case of the homotopy category $h\C$. For the detail, see \cite{NP}.

\begin{dfn}\label{DefEquivSeq}$($\cite[Definition~2.7]{NP}$)$
For a pair of objects $A,C\in h\C$, two consecutive morphisms $A\ov{\xfr}{\lra}B\ov{\yfr}{\lra}C$ and $A\ov{\xfr\ppr}{\lra}B\ppr\ov{\yfr\ppr}{\lra}C$ are said to be \emph{equivalent} if there is an isomorphism $\bfr\in(h\C)(B,B\ppr)$ such that $\bfr\ci\xfr=\xfr\ppr$ and $\yfr\ppr\ci\bfr=\yfr$.

In the sequel, the equivalence class to which a sequence $A\ov{\xfr}{\lra}B\ov{\yfr}{\lra}C$ belongs will be denoted by $[A\ov{\xfr}{\lra}B\ov{\yfr}{\lra}C]$. 
\end{dfn}

\begin{dfn}\label{DefAddReal}$($\cite[Definitions~2.9 and 2.10]{NP}$)$
Suppose that there is a biadditive functor $\E\co(h\C)\op\ti(h\C)\to\Ab$ to the category of abelian groups $\Ab$. Let $\sfr$ be a correspondence which associates an equivalence class $\mathfrak{s}(\delta)=[A\ov{\xfr}{\lra}B\ov{\yfr}{\lra}C]$ to each element $\del\in\E(C,A)$. Such $\sfr$ is called an \emph{additive realization of $\E$} if it satisfies the following conditions.
\begin{itemize}
\item[{\rm (i)}] Let $\del\in\E(C,A)$ and $\del\ppr\in\E(C\ppr,A\ppr)$ be any pair of elements, with
\begin{equation}\label{PairReal}
\sfr(\del)=[A\ov{\xfr}{\lra}B\ov{\yfr}{\lra}C],\ \ \sfr(\del\ppr)=[A\ppr\ov{\xfr\ppr}{\lra}B\ppr\ov{\yfr\ppr}{\lra}C\ppr].
\end{equation}
Then, for any pair of morphisms $\afr\in(h\C)(A,A\ppr),\cfr\in(h\C)(C,C\ppr)$ satisfying $\E(C,\afr)(\del)=\E(\cfr,A\ppr)(\del\ppr)$, there exists a morphism $\bfr\in(h\C)(B,B\ppr)$ such that $\bfr\ci\xfr=\xfr\ppr\ci\afr$ and $\yfr\ppr\ci\bfr=\cfr\ci\yfr$.
\item[{\rm (ii)}] For any $A,C\in h\C$, the zero element $0\in\E(C,A)$ satisfies
\[ \sfr(0)=[A\ov{\left[\bsm 1\\0\esm\right]}{\lra}A\oplus C\ov{[0\ 1]}{\lra}C]. \]
\item[{\rm (iii)}] For any pair of elements $\del\in\E(C,A)$ and $\del\ppr\in\E(C\ppr,A\ppr)$ with $(\ref{PairReal})$, we have
\[ \sfr(\del\oplus\del\ppr)=[A\oplus A\ppr\ov{\xfr\oplus\xfr\ppr}{\lra}B\oplus B\ppr\ov{\yfr\oplus\yfr\ppr}{\lra}C\oplus C\ppr]. \]
Here $\del\oplus\del\ppr\in\E(C\oplus C\ppr,A\oplus A\ppr)$ denotes the element which corresponds to $(\delta,0,0,\delta^{\prime})$ through the natural isomorphism
\[ \E(C\oplus C\ppr,A\oplus A\ppr)\cong\E(C,A)\oplus\E(C,A\ppr)\oplus\E(C\ppr,A)\oplus\E(C\ppr,A\ppr) \]
induced from the biadditivity of $\E$.
\end{itemize}

If $\sfr$ is an additive realization, any pair $(A\ov{\xfr}{\lra}B\ov{\yfr}{\lra}C,\del)$ of a sequence $A\ov{\xfr}{\lra}B\ov{\yfr}{\lra}C$ and an element $\del\in\E(C,A)$ satisfying $\sfr(\del)=[A\ov{\xfr}{\lra}B\ov{\yfr}{\lra}C]$ is called an \emph{$\sfr$-triangle} (or simply an \emph{extriangle} if $\sfr$ is obvious from the context), and abbreviately expressed by $A\ov{\xfr}{\lra}B\ov{\yfr}{\lra}C\ov{\del}{\dra}$.
\end{dfn}

\begin{dfn}\label{DefExtriangulation}$($\cite[Definition~2.12]{NP}$)$
A triplet $(h\C,\E,\sfr)$ is called an \emph{extriangulated category} if it satisfies the following conditions. In this case, the pair $(\E,\sfr)$ is called an \emph{external triangulation} of $h\C$.
\begin{itemize}
\item[{\rm (ET1)}] $\E\co(h\C)\op\ti(h\C)\to\Ab$ is a biadditive functor.
\item[{\rm (ET2)}] $\sfr$ is an additive realization of $\E$.
\item[{\rm (ET3)}] Let $\del\in\E(C,A)$ and $\del\ppr\in\E(C\ppr,A\ppr)$ be any pair of elements with $(\ref{PairReal})$.
Then, for any commutative square
\begin{equation}\label{SquareForET3}
\xy
(-12,6)*+{A}="0";
(0,6)*+{B}="2";
(12,6)*+{C}="4";
(-12,-6)*+{A\ppr}="10";
(0,-6)*+{B\ppr}="12";
(12,-6)*+{C\ppr}="14";
{\ar^{\xfr} "0";"2"};
{\ar^{\yfr} "2";"4"};
{\ar_{\afr} "0";"10"};
{\ar^{\bfr} "2";"12"};
{\ar_{\xfr\ppr} "10";"12"};
{\ar_{\yfr\ppr} "12";"14"};
{\ar@{}|{\circlearrowright} "0";"12"};
\endxy
\end{equation}
in $h\C$, there exists $\cfr\in(h\C)(C,C\ppr)$ such that $\cfr\ci\yfr=\yfr\ppr\ci\bfr$ and $\E(C,\afr)(\del)=\E(\cfr,A\ppr)(\del\ppr)$.
\item[{\rm (ET3)$\op$}] Dual of {\rm (ET3)}.
\item[{\rm (ET4)}] Let $\del\in\E(D,A)$ and $\del\ppr\in\E(F,B)$ be any pair of elements, with
\[ \sfr(\del)=[A\ov{\ffr}{\lra}B\ov{\ffr\ppr}{\lra}D]\ \ \text{and}\ \ \sfr(\del\ppr)=[B\ov{\gfr}{\lra}C\ov{\gfr\ppr}{\lra}F]. \]
Then there exist an object $E\in h\C$, a commutative diagram
\begin{equation}\label{DiagET4}
\xy
(-18,6)*+{A}="0";
(-6,6)*+{B}="2";
(6,6)*+{D}="4";
(-18,-6)*+{A}="10";
(-6,-6)*+{C}="12";
(6,-6)*+{E}="14";
(-6,-18)*+{F}="22";
(6,-18)*+{F}="24";
{\ar^{\ffr} "0";"2"};
{\ar^{\ffr\ppr} "2";"4"};
{\ar@{=} "0";"10"};
{\ar_{\gfr} "2";"12"};
{\ar^{\dfr} "4";"14"};
{\ar_{\hfr} "10";"12"};
{\ar_{\hfr\ppr} "12";"14"};
{\ar_{\gfr\ppr} "12";"22"};
{\ar^{\efr} "14";"24"};
{\ar@{=} "22";"24"};
{\ar@{}|{\circlearrowright} "0";"12"};
{\ar@{}|{\circlearrowright} "2";"14"};
{\ar@{}|{\circlearrowright} "12";"24"};
\endxy
\end{equation}
in $h\C$, and an element $\del\pprr\in\E(E,A)$ such that $\sfr(\del\pprr)=[A\ov{\hfr}{\lra}C\ov{\hfr\ppr}{\lra}E]$, which satisfy the following compatibilities.
\begin{itemize}
\item[{\rm (i)}] $\sfr(\E(F,\ffr\ppr)(\del\ppr))=[D\ov{\dfr}{\lra}E\ov{\efr}{\lra}F]$,
\item[{\rm (ii)}] $\E(\dfr,A)(\del\pprr)=\delta$,

\item[{\rm (iii)}] $\E(E,\ffr)(\del\pprr)=\E(\efr,B)(\del\ppr)$. 
\end{itemize}

\item[{\rm (ET4)$\op$}] Dual of {\rm (ET4)}.
\end{itemize}
\end{dfn}

In the rest of this section, we will show that $h\C$ can be equipped with an external triangulation $(\E,\sfr)$, for any exact quasi-category $\C$.

\subsection{Construction of the functor $\Ebb$}

Firstly, in this subsection we construct a functor $\E\co(h\C)\op\ti(h\C)\to\Sets$ to the category $\Sets$ of sets.
\begin{dfn}\label{DefEofCA}
Let $A,C\in\C_0$. Define $\Ebb(C,A)$ to be the set of equivalence classes of exact sequences $\Sf$ starting from $A$ and ending in $C$, modulo the equivalence relation $\sim$ defined in Definition~\ref{DefEquivExSeq}.

For each exact sequence ${}_A\Sf_C$, let $\del=\und{\Sf}$ denote its equivalence class with respect to $\sim$. When we emphasize the end-objects $A$ and $C$, we write ${}_A\del_C={}_A\und{\Sf}_C$.
\end{dfn}

We make use of the notation $a\sas\Sf$ from \cref{LemExSeqPO}.

\begin{prop}\label{PropFunctoriality}
Suppose that we have ${}_A\Sf_C\sim {}_A\Sf\ppr_C$, and let
\[
\xy
(-7,7)*+{A}="0";
(-7,-7)*+{A\ppr}="4";
(7,-7)*+{A\pprr}="6";
{\ar_{a_1} "0";"4"};
{\ar^{a_3} "0";"6"};
{\ar_{a_2} "4";"6"};
(-3,-7)*+{}="10";
(-7,-3)*+{}="11";
{\ar@/_0.2pc/@{-}_{_{\al}} "10";"11"};
\endxy
\]
be any $2$-simplex in $\C$. Then we have $a_{2\ast}a_{1\ast}\Sf\ppr\sim a_{3\ast}\Sf$.
\end{prop}
\begin{proof}
By the assumption, there is a morphism ${}_{1_A}\Cf_{1_C}\co\Sf\to\Sf\ppr$.
By the definition of $a_{1\ast}\Sf\ppr$, $a_{2\ast}a_{1\ast}\Sf\ppr$ and $a_{3\ast}\Sf$, there are morphisms of exact sequences
\[
{}_{a_1}\Df^{(1)}_{1_C}\co\Sf\ppr\to a_{1\ast}\Sf\ppr,\ \ 
{}_{a_2}\Df^{(2)}_{1_C}\co a_{1\ast}\Sf\ppr\to a_{2\ast}a_{1\ast}\Sf\ppr,\ \ 
{}_{a_3}\Df^{(3)}_{1_C}\co\Sf\to a_{3\ast}\Sf
\]
in which $\Df^{(k)}|_{\{0\}\ti\Delta^1\ti\Delta^1}$ $(k=1,2,3)$ are push-outs.
By Proposition~\ref{PropComposeCubes} applied iteratively to $\Sf\ov{\Cf}{\lra}\Sf\ppr\ov{\Df^{(1)}}{\lra}a_{1\ast}\Sf\ppr\ov{\Df^{(2)}}{\lra}a_{2\ast}a_{1\ast}\Sf\ppr$, we obtain a morphism ${}_{a_3}\Cf\ppr_{1_C}\co\Sf\to a_{2\ast}a_{1\ast}\Sf\ppr$.
Then by Lemma~\ref{LemExSeqPO} {\rm (2)}, we obtain a morphism ${}_{1_{A\pprr}}\Cf\pprr_{1_C}\co a_{3\ast}\Sf\to a_{2\ast}a_{1\ast}\Sf\ppr$.
This shows $a_{3\ast}\Sf\sim a_{2\ast}a_{1\ast}\Sf\ppr$.
\end{proof}

\begin{cor}\label{CorFunctoriality}
We have the following.
\begin{enumerate}
\item If morphisms $a,a\ppr\in\C_1(A,A\ppr)$ and exact sequences ${}_A\Sf_C,{}_A\Sf\ppr_C$ satisfy $\ovl{a}=\ovl{a\ppr}$ and $\Sf\sim\Sf\ppr$, then $a\sas\Sf\sim a\ppr\sas\Sf\ppr$. Thus for any $\afr=\ovl{a}\in(h\C)(A,A\ppr)$ and any $C\in\C_0$, the map
\[ \afr\sas\co\E(C,A)\to\E(C,A\ppr)\ ;\ \del=\und{\Sf}\mapsto\afr\sas\del=\und{a\sas\Sf} \]
is well-defined.
\item The maps obtained in {\rm (1)} satisfy the following.
\begin{itemize}
\item[{\rm (i)}] For any $A,C\in\C_0$, we have $(\id_A)\sas=\id_{\E(C,A)}\co\E(C,A)\to\E(C,A)$.
\item[{\rm (ii)}] For any $\afr\in (h\C)(A,A\ppr),\afr\ppr\in (h\C)(A\ppr,A\pprr)$ and any $C\in\C_0$, we have $(\afr\ppr\ci\afr)\sas=\afr\ppr\sas\ci\afr\sas\co\E(C,A)\to\E(C,A\pprr)$.
\end{itemize}
\end{enumerate}
\end{cor}
\begin{proof}
{\rm (1)} and {\rm  (ii)} of {\rm (2)} are immediate from Proposition~\ref{PropFunctoriality}. {\rm  (i)} of {\rm (2)} is also obvious by the existence of $\If_{\Sf}$ for any ${}_A\Sf_C$.
\end{proof}

\begin{lem}\label{LemET2}
Let ${}_A\Sf_C,{}_{A\ppr}\Sf\ppr_{C\ppr}$ be any pair of exact sequences, and put $\del=\und{\Sf},\del\ppr=\und{\Sf\ppr}$. Let $\afr\in (h\C)(A,A\ppr)$ and $\cfr\in (h\C)(C,C\ppr)$ be any pair of morphisms in the homotopy category. The following are equivalent.
\begin{enumerate}
\item $\afr\sas\del=\cfr\uas\del\ppr$.
\item There exists a morphism ${}_a\Cf_c\co\Sf\to\Sf\ppr$ 
such that $\ovl{a}=\afr$ and $\ovl{c}=\cfr$.
\end{enumerate}
\end{lem}
\begin{proof}
$(1)\Rightarrow(2)$.
Take any $a\in\C_1(A,A\ppr)$ and $c\in\C_1(C,C\ppr)$ such that $\afr=\ovl{a}$ and $\cfr=\ovl{c}$. By definition, there exist morphisms $\Sf\ov{{}_a\Df_{1_C}}{\lra}a\sas\Sf\ov{{}_{1_{A\ppr}}\Cf^{(1)}_{1_C}}{\lra}c\uas\Sf\ppr\ov{{}_{1_{A\ppr}}\Df\ppr_c}{\lra}\Sf\ppr$ of exact sequences. By Proposition~\ref{PropComposeCubes}, we obtain a morphism ${}_a\Cf_c\co\Sf\to\Sf\ppr$ as desired.

$(2)\Rightarrow(1)$. Let ${}_a\Cf_c\co\Sf\to\Sf\ppr$ be a morphism as in $(\ref{MorphExSeq})$ such that $\ovl{a}=\afr$, $\ovl{c}=\cfr$. By Lemma~\ref{LemExSeqPO}, we obtain a morphism ${}_{1_{A\ppr}}\Cf\ppr_{c}\co a\sas\Sf\to\Sf\ppr$. 
Then by the dual of Lemma~\ref{LemExSeqPO} applied to $\Cf\ppr$ and a pull-back along $c$, we obtain a morphism ${}_{1_{A\ppr}}\Cf\pprr_{1_C}\co a\sas\Sf\to c\uas\Sf\ppr$.
Thus $a\sas\Sf\sim c\uas\Sf\ppr$ follows.
\end{proof}

\begin{prop}\label{PropEFtr}
For any $\afr\in(h\C)(A,A\ppr)$ and $\cfr\in(h\C)(C\ppr,C)$, we have
\[ \afr\sas\cfr\uas=\cfr\uas\afr\sas\co\E(C,A)\to\E(C\ppr,A\ppr). \]
\end{prop}
\begin{proof}
Take any $a\in\C_1(A,A\ppr)$ and $c\in\C_1(C\ppr,C)$ so that $\afr=\ovl{a}$ and $\cfr=\ovl{c}$ hold. Let ${}_A\Sf_{C}$ be any exact sequence starting from $A$ and ending in $C$. By Lemma~\ref{LemExSeqPO} {\rm (1)} and its dual, we obtain morphisms of exact sequences
\[ c\uas\Sf\ov{{}_{1_A}\Df\ppr_{c}}{\lra}\Sf\ov{{}_{a}\Df_{1_{C}}}{\lra}a\sas\Sf. \]
By Proposition~\ref{PropComposeCubes}, we obtain a morphism ${}_{a}\Cf\pprr_{c}\co c\uas\Sf\to a\sas\Sf$. Lemma~\ref{LemET2} shows
$\afr\sas(\cfr\uas\und{\Sf})=\cfr\uas(\afr\sas\und{\Sf})$.
\end{proof}

The argument so far gives the following functor.
\begin{dfn}\label{DefEFtr}
Functor $\Ebb\co(h\C)\op\ti h\C\to\Sets$ is defined by the following.
\begin{itemize}
\item For an object $(C,A)\in(h\C)\op\ti(h\C)$, the set $\E(C,A)$ is the one defined in Definition~\ref{DefEofCA}.
\item For a morphism $(\cfr,\afr)\in\big((h\C)\op\ti(h\C)\big)\big((C,A),(C\ppr,A\ppr)\big)$, the map
\[ \E(\cfr,\afr)\co\E(C,A)\to\E(C\ppr,A\ppr) \]
is the one defined by $\E(\cfr,\afr)=\cfr\uas\afr\sas=\afr\sas\cfr\uas$.
\end{itemize}
\end{dfn}

\subsection{Biadditivity of $\E$}
In this subsection we will show that the functor $\E$ obtained in Definition~\ref{DefEFtr} indeed factors through $\Ab$.
\begin{dfn}\label{DefZeroElement}
By Corollary~\ref{CorSplitEachOther}, the equivalence class of the splitting exact sequences forms one element in $\E(C,A)$ for each $A,C\in h\C$. We denote this particular element by ${}_A0_C$. Especially we have ${}_A0_C={}_A\und{\Nf}_C= {}_A\und{\Nf\ppr}_C$. 
\end{dfn}

\begin{rem}\label{RemZeroElement}
It is not hard to see that $\Ebb(O,O)$ consists of only one object for each zero object $O\in\C_0$, which necessarily agrees with ${}_O0_O$. Remark that a morphism $a\in\C_1(A,A\ppr)$ satisfies $\ovl{a}=0$ in $h\C$ if and only if it factors through some (equivalently, any) zero object. For any $\del\in\Ebb(C,A)$ and any $a\in\C_1(A,A\ppr)$ such that $\ovl{a}=0$, we have $a\sas\del={}_{A\ppr}0_C$ by Lemma~\ref{LemInvert}. Similarly we have $0\uas\del={}_A0_{C\ppr}$ for the zero morphism $0\in (h\C)(C\ppr,C)$.
\end{rem}

\begin{prop}\label{PropSumExtension}
The following holds.
\begin{enumerate}
\item For any $A,A\ppr\in\C_0$, there is a natural isomorphism
\[ (\ups_{A,A\ppr})\sas\co\Ebb(-,A\am A\ppr)\ov{\cong}{\ltc}\Ebb(-,A\ti A\ppr) \]
of functors $(h\C)\op\to\Sets$.
\item For any $C,C\ppr\in\C_0$, there is a natural isomorphism
\[ (\ups_{C,C\ppr})\uas\co\Ebb(C\ti C\ppr,-)\ov{\cong}{\ltc}\Ebb(C\am C\ppr,-) \]
of functors $h\C\to\Sets$.
\item For any $A,A\ppr,C,C\ppr\in\C_0$, there is an isomorphism 
\[ \vt=\vt_{A\ppr,C\ppr}^{A,C}\co\Ebb(C\am C\ppr,A\am A\ppr)\ov{\cong}{\ltc}\Ebb(C\ti C\ppr,A\ti A\ppr) \]
given by $\vt_{A\ppr,C\ppr}^{A,C}=((\ups_{C,C\ppr})\uas)^{-1}\ci(\ups_{A,A\ppr})\sas$. This sends $\del\am\del\ppr=\und{\Sf\am\Sf\ppr}$ to $\del\ti\del\ppr=\und{\Sf\ti\Sf\ppr}$ for any $\del=\und{\Sf}\in\Ebb(C,A)$ and $\del\ppr=\und{\Sf\ppr}\in\Ebb(C\ppr,A\ppr)$.
\end{enumerate}
\end{prop}
\begin{proof}
Since $\ovl{\ups_{A,A\ppr}}$ is an isomorphism in $h\C$, {\rm (1)} is a formal consequence of the functoriality of $\Ebb$. Dually for {\rm (2)}.

The former part of {\rm (3)} is immediate from {\rm (1)} and {\rm (2)}. The latter part follows from Corollary~\ref{CorProdExSeq1}, since it implies $\und{(\ups_{A,A\ppr})\sas(\Sf\am\Sf\ppr)}=\und{(\ups_{C,C\ppr})\uas(\Sf\ti\Sf\ppr)}$ by Lemma~\ref{LemET2}.
\end{proof}

\begin{rem}\label{RemIdentify}
In the rest we will identify $A\oplus A\ppr=A\am A\ppr$ with $A\ti A\ppr$ through the isomorphism $\ovl{\ups_{A,A\ppr}}\co A\oplus A\ppr\ov{\cong}{\lra}A\ti A\ppr$ in $h\C$.

Also we put $\del\oplus\del\ppr=\del\am\del\ppr\in\Ebb(C\oplus C\ppr,A\oplus A\ppr)$. By the above identification and Proposition~\ref{PropSumExtension} {\rm (3)}, we will often identify $\del\oplus\del\ppr$ with $\del\ti \del\ppr=\vt(\del\oplus\del\ppr)$.
This agrees with the one in {\rm (iii)} of Definition~\ref{DefAddReal} via the above identification.
\end{rem}

\begin{lem}\label{LemForAddExtension}
We have the following.
\begin{enumerate}
\item For any $\del\in\Ebb(C,A)$, we have
\begin{equation}
(\nabla_A)\sas(\del\oplus\del)=(\nabla_C)\uas\del
\quad\text{and}\quad
(\Delta_A)\sas\del=(\Delta_C)\uas(\del\oplus\del).
\end{equation}
\item For any $\del_k\in\Ebb(C_k,A_k),\afr_k\in(h\C)(A_k,A_k\ppr),\cfr_k\in(h\C)(C_k,C_k\ppr)$ $(k=1,2)$, we have
\[ (\afr_1\oplus \afr_2)\sas(\del_1\oplus\del_2)=\afr_{1\ast}\del_1\oplus \afr_{2\ast}\del_2\ \ \text{and}\ \ 
(\cfr_1\oplus \cfr_2)\uas(\del_1\oplus\del_2)=\cfr_1\uas\del_1\oplus \cfr_2\uas\del_2.
 \]
\end{enumerate}
\end{lem}
\begin{proof}
{\rm (1)}
This follows from Lemma~\ref{LemET2} applied to the morphisms in Example~\ref{ExProdExSeq}.

{\rm (2)}
Take representatives $\del_k=\und{\Sf_k}$, $\afr_k=\ovl{a_k}$ $(k=1,2)$. By the definition of $a_{k\ast}\Sf_k$, there are morphisms of exact sequences
\[ {}_{a_k}(\Df_k)_{1_{C_k}}\co {}_{A_k}(\Sf_k)_{C_k}\to {}_{A_k\ppr}(a_{k\ast}\Sf_k)_{C_k} \]
for $k=1,2$. By Proposition~\ref{PropProdExSeq1}, we have a morphism ${}_{\ups}\Cf_{\ups\ppr}\co\Sf_1\am\Sf_2\to\Sf_1\ti\Sf_2$ in which, under the identifications in Remark~\ref{RemIdentify}, the morphisms $\ups\in\C_1(A_1\am A_2,A_1\ppr\ti A_2\ppr)$ and $\ups\ppr\in\C_1(C_1\am C_2,C_1\ti C_2)$ induce
\begin{eqnarray*}
\ovl{\ups}=\ovl{a_1}\oplus\ovl{a_2}=\afr_1\oplus\afr_2&\in&(h\C)(A_1\oplus A_2,A_1\ppr\oplus A_2\ppr),\\
\ovl{\ups\ppr}=\ovl{1_{C_1}}\oplus\ovl{1_{C_2}}=\id_{C_1\oplus C_2}&\in&(h\C)(C_1\oplus C_2,C_1\oplus C_2)
\end{eqnarray*}
in the homotopy category, thus implies
\[
(\afr_1\oplus\afr_2)\sas(\del_1\oplus\del_2)
=\und{\ups\sas(\Sf_1\am\Sf_2)}
=\und{\ups^{\prime\ast}(a_{1\ast}\Sf_1\ti a_{2\ast}\Sf_2)}
=\afr_{1\ast}\del_1\oplus\afr_{2\ast}\del_2
\]
by Lemma~\ref{LemET2}. Similarly for $(\cfr_1\oplus\cfr_2)\uas(\del_1\oplus\del_2)=\cfr_1\uas\del_1\oplus\cfr_2\uas\del_2$.
\end{proof}

\begin{dfn}\label{DefAddExtension}
For any $\del_1,\del_2\in\Ebb(C,A)$, define $\del_1+\del_2\in\Ebb(C,A)$ by
\[ \del_1+\del_2=(\nabla_A)\sas(\Delta_C)\uas(\del_1\oplus\del_2). \]
Obviously this operation is commutative.
\end{dfn}

\begin{prop}\label{PropAddExtMor}
Let $\del\in\Ebb(C,A)$ be any element.
\begin{enumerate}
\item  For any pair of endomorphisms $\afr,\afr\ppr\in\End_{h\C}(A)$ of $A$ in $h\C$, we have $(\afr+\afr\ppr)\sas\del=\afr\sas\del+\afr\ppr\sas\del$.
\item  For any pair of endomorphisms $\cfr,\cfr\ppr\in\End_{h\C}(C)$ of $C$ in $h\C$, we have $(\cfr+\cfr\ppr)\uas\del=\cfr\uas\del+\cfr^{\prime\ast}\del$.
\end{enumerate}
\end{prop}
\begin{proof}
{\rm (1)} By Lemma~\ref{LemForAddExtension}, we have
\begin{eqnarray*}
\afr\sas\del+\afr\ppr\sas\del&=&(\nabla_A)\sas(\Delta_C)\uas(\afr\sas\del\oplus\afr\ppr\sas\del)
\ =\ (\nabla_A)\sas(\Delta_C)\uas(\afr\oplus\afr\ppr)\sas(\del\oplus\del)\\
&=&(\nabla_A)\sas(\afr\oplus\afr\ppr)\sas(\Delta_C)\uas(\del\oplus\del)
\ =\ (\nabla_A)\sas(\afr\oplus\afr\ppr)\sas(\Delta_A)\sas\del\\
&=&(\nabla_A\circ (\afr\oplus\afr\ppr)\circ\Delta_A)\sas\del\ =\ (\afr+\afr\ppr)\sas\del.
\end{eqnarray*}
Similarly for {\rm (2)}.
\end{proof}

In the following corollary, remark that $h\C$ is an additive category by definition, hence there are $-\id_A\in (h\C)(A,A)$ and $-\id_C\in (h\C)(C,C)$. 
\begin{cor}\label{CorAddExtMor}
For any $\del\in\Ebb(C,A)$, the following holds in $\Ebb(C,A)$.
\begin{enumerate}
\item $\del+{}_A0_C=\del$.
\item $\del+(-\id_C)\uas\del=\del+(-\id_A)\sas\del={}_A0_C$.
\end{enumerate}
In particular it follows that $(-\id_C)\uas\del=(-\id_A)\sas\del$ gives the additive inverse of $\del$.
\end{cor}
\begin{proof}
This follows from Remark~\ref{RemZeroElement} and Proposition~\ref{PropAddExtMor}.
\end{proof}

\begin{prop}\label{PropEAddFtr}
The following holds.
\begin{enumerate}
\item For any $A_1,A_2,C\in\C_0$, 
\begin{equation}\label{CanBij}
([1\ 0]\sas,[0\ 1]\sas)\co\Ebb(C,A_1\oplus A_2)\to\Ebb(C,A_1)\ti\Ebb(C,A_2)
\end{equation}
is bijective.
\item For any $A,C_1,C_2\in\C_0$, 
\[ (\spmatrix{1}{0}\uas,\spmatrix{0}{1}\uas)\co\Ebb(C_1\oplus C_2,A)\to\Ebb(C_1,A)\ti\Ebb(C_2,A) \]
is bijective.
\end{enumerate}
\end{prop}
\begin{proof}
{\rm (1)} Let us denote the map $(\ref{CanBij})$ by $F$ for simplicity. It suffices to show that 
\[ G\co\Ebb(C,A_1)\ti\Ebb(C,A_2)\to\Ebb(C,A_1\oplus A_2)\ ;\ (\del_1,\del_2)\mapsto \spmatrix{1}{0}\sas\del_1+\spmatrix{0}{1}\sas\del_2 \]
gives its inverse. $G\ci F=\id$ follows from Proposition~\ref{PropAddExtMor}, since we have
\[ \spmatrix{1}{0}\sas[1\ 0]\sas\del+\spmatrix{0}{1}\sas[0\ 1]\sas\del=(\spmatrix{1}{0}[1\ 0]+\spmatrix{0}{1}[0\ 1])\sas\del=(\id_{A_1\oplus A_2})\sas\del=\del \]
for any $\del\in\Ebb(C,A_1\oplus A_2)$.

For any $(\del_1,\del_2)\in\Ebb(C,A_1)\ti\Ebb(C,A_2)$, we have
\[ [1\ 0]\sas(\del_1\oplus\del_2)=[1\ 0]\uas\del_1 \]
by Proposition~\ref{PropProdExSeq1} applied to $\If\co\Sf_1\to\Sf_1$ and $\Sf_2\to\Of$. Thus by Lemma~\ref{LemForAddExtension}, we obtain
\begin{eqnarray*}
[1\ 0]\sas(\spmatrix{1}{0}\sas\del_1+\spmatrix{0}{1}\sas\del_2)
&=&[1\ 0]\sas(\nabla_{A_1\oplus A_2})\sas(\Delta_C)\uas(\spmatrix{1}{0}\sas\del_1\oplus\spmatrix{0}{1}\sas\del_2)\\
&=&[1\ 0]\sas(\nabla_{A_1\oplus A_2})\sas(\Delta_C)\uas(\spmatrix{1}{0}\oplus\spmatrix{0}{1})\sas(\del_1\oplus\del_2)\\
&=&(\Delta_C)\uas\Big([1\ 0]\ci\nabla_{A_1\oplus A_2}\ci(\spmatrix{1}{0}\oplus\spmatrix{0}{1})\Big)\sas(\del_1\oplus\del_2)\\
&=&(\Delta_C)\uas[1\ 0]\sas(\del_1\oplus\del_2)
\ = (\Delta_C)\uas[1\ 0]\uas\del_1\ =\ \del_1.
\end{eqnarray*}
Similarly we have $[0\ 1]\sas(\spmatrix{1}{0}\sas\del_1+\spmatrix{0}{1}\sas\del_2)=\del_2$, and thus $F\ci G((\del_1,\del_2))=(\del_1,\del_2)$.
{\rm (2)} can be shown dually.
\end{proof}

\begin{cor}\label{CorEAddFtr}
$\Ebb(C,A)$ has a natural structure of additive group, whose addition is given by Definition~\ref{DefAddExtension}, with the zero element ${}_A0_C$.
Moreover,
\[ \Ebb\co(h\C)\op\ti h\C\to\Ab \]
becomes a biadditive functor.
\end{cor}
\begin{proof}
This is a formal consequence of the functoriality of $\Ebb\co(h\C)\op\ti h\C\to\Sets$, Corollary~\ref{CorAddExtMor} and Proposition~\ref{PropEAddFtr}.
\end{proof}

\subsection{External triangulation of $h\C$}

\begin{dfn}\label{DefReal}
Let $A,C\in\C_0$ be any pair of objects. For each $\del\in\E(C,A)$, take any
\[
\Sf=\ 
\xy
(-7,7)*+{A}="0";
(7,7)*+{B}="2";
(-7,-7)*+{O}="4";
(7,-7)*+{C}="6";
{\ar^{x} "0";"2"};
{\ar_{i} "0";"4"};
{\ar^{y} "2";"6"};
{\ar_{j} "4";"6"};
\endxy
\]
such that $\del=\und{\Sf}$, and put $\sfr(\del)=[A\ov{\ovl{x}}{\lra}B\ov{\ovl{y}}{\lra}C]$, where the equivalence class of the sequences is the one in Definition~\ref{DefEquivSeq}. Well-definedness follows from Proposition~\ref{PropInvert}.
\end{dfn}

\begin{lem}\label{LemET}
$\sfr$ is an additive realization of $\Ebb$.
\end{lem}
\begin{proof}
Let us confirm the conditions {\rm (i),(ii),(iii)} in Definition~\ref{DefAddReal}.

{\rm (i)} This follows from Lemma~\ref{LemET2}.

{\rm (ii)}
Let $A,C\in\C_0$ be any pair of objects. Since ${}_A0_C={}_A\und{\Nf\ppr}_C$
, we have
\[ \sfr({}_A0_C)=[A\ov{\ovl{i}_A}{\lra}A\am C\ov{\ovl{j}_C}{\lra}C]=[A\ov{\left[\bsm 1\\0\esm\right]}{\lra}A\am C\ov{[0\ 1]}{\lra}C]. \]

{\rm (iii)}
Let $\del=\und{\Sf}\in\Ebb(C,A),\del\ppr=\und{\Sf\ppr}\in\Ebb(C\ppr,A\ppr)$ be any pair of elements, with
\[
\Sf=\ 
\xy
(-6,6)*+{A}="0";
(6,6)*+{B}="2";
(-6,-6)*+{O}="4";
(6,-6)*+{C}="6";
{\ar^{x} "0";"2"};
{\ar_{i} "0";"4"};
{\ar^{y} "2";"6"};
{\ar_{j} "4";"6"};
\endxy
\quad
\text{and}
\quad
\Sf\ppr=\ 
\xy
(-6,6)*+{A\ppr}="0";
(6,6)*+{B\ppr}="2";
(-6,-6)*+{O\ppr}="4";
(6,-6)*+{C\ppr}="6";
{\ar^{x\ppr} "0";"2"};
{\ar_{i\ppr} "0";"4"};
{\ar^{y\ppr} "2";"6"};
{\ar_{j\ppr} "4";"6"};
\endxy.
\]
By definition we have 
\[ \sfr(\del)=[A\ov{\xfr}{\lra}B\ov{\yfr}{\lra}C],\ \ \sfr(\del\ppr)=[A\ppr\ov{\xfr\ppr}{\lra}B\ppr\ov{\yfr\ppr}{\lra}C\ppr], \]
and $\del\oplus\del\ppr=\und{\Sf\am\Sf\ppr}$ as in Remark~\ref{RemIdentify}.
It follows
\[ \sfr(\del\oplus\del\ppr)=[A\am A\ppr\ov{\ovl{x\am x\ppr}}{\lra}B\am B\ppr\ov{\ovl{y\am y\ppr}}{\lra}C\am C\ppr]=[A\oplus A\ppr\ov{\ovl{x}\oplus\ovl{x\ppr}}{\lra}B\oplus B\ppr\ov{\ovl{y}\oplus\ovl{y\ppr}}{\lra}C\oplus C\ppr ]. \]

\end{proof}

\begin{thm}\label{ThmET}
$(h\C,\Ebb,\sfr)$ is an extriangulated category.
\end{thm}
\begin{proof}
{\rm (ET1)} and {\rm (ET2)} are shown in Corollary~\ref{CorEAddFtr} and Lemma~\ref{LemET}, respectively.
{\rm (ET3)} follows from Proposition~\ref{PropMorphExSeq} and Lemma~\ref{LemET2}. Dually for {\rm (ET3)$\op$}. It remains to show {\rm (ET4)} and {\rm (ET4)$\op$}.

Since {\rm (ET4)$\op$} can be shown dually, it is enough to show {\rm (ET4)}.
Let $\del=\und{\Sf}\in\E(D,A)$ and $\del\ppr=\und{\Sf\ppr}\in\E(F,B)$ be any pair of elements given by the following exact sequences.
\[
\Sf=\ 
\xy
(-7,7)*+{A}="0";
(7,7)*+{B}="2";
(-7,-7)*+{O}="4";
(7,-7)*+{D}="6";
{\ar^{f} "0";"2"};
{\ar_{i} "0";"4"};
{\ar|*+{_z} "0";"6"};
{\ar^{f\ppr} "2";"6"};
{\ar_{j} "4";"6"};
(3,7)*+{}="00";
(7,3)*+{}="01";
{\ar@/_0.2pc/@{-}_{^{\xi}} "00";"01"};
(-3,-7)*+{}="10";
(-7,-3)*+{}="11";
{\ar@/_0.2pc/@{-}_{_{\eta}} "10";"11"};
\endxy,\quad
\Sf\ppr=\ 
\xy
(-7,7)*+{B}="0";
(7,7)*+{C}="2";
(-7,-7)*+{O\ppr}="4";
(7,-7)*+{F}="6";
{\ar^{g} "0";"2"};
{\ar_{i\ppr} "0";"4"};
{\ar|*+{_{z\ppr}} "0";"6"};
{\ar^{g\ppr} "2";"6"};
{\ar_{j\ppr} "4";"6"};
(3,7)*+{}="00";
(7,3)*+{}="01";
{\ar@/_0.2pc/@{-}_(0.4){^{\xi\ppr}} "00";"01"};
(-3,-7)*+{}="10";
(-7,-3)*+{}="11";
{\ar@/_0.2pc/@{-}_(0.4){_{\eta\ppr}} "10";"11"};
\endxy
\]
By definition, we have $\sfr(\del)=[A\ov{\ovl{f}}{\lra}B\ov{\ovl{f\ppr}}{\lra}D]$ and $\sfr(\del\ppr)=[B\ov{\ovl{g}}{\lra}C\ov{\ovl{g\ppr}}{\lra}F]$.
Since $f,g$ are ingressive, we may take a $2$-simplex
\[
\xy
(-7,7)*+{A}="0";
(7,7)*+{B}="4";
(7,-7)*+{C}="6";
{\ar^{f} "0";"4"};
{\ar_{h} "0";"6"};
{\ar^{g} "4";"6"};
(3,7)*+{}="10";
(7,3)*+{}="11";
{\ar@/_0.2pc/@{-}_{_{\chi}} "10";"11"};
\endxy
\]
in which $h$ is ingressive by the definition of exact quasi-category.
As in \cite[Remark~3.2]{B1}, $h$ should appear as the fiber of some exact sequence as below.
\[
\Sf\pprr=\ 
\xy
(-7,7)*+{A}="0";
(7,7)*+{C}="2";
(-7,-7)*+{O\pprr}="4";
(7,-7)*+{E}="6";
{\ar^{h} "0";"2"};
{\ar_{i\pprr} "0";"4"};
{\ar|*+{_{z\pprr}} "0";"6"};
{\ar^{h\ppr} "2";"6"};
{\ar_{j\pprr} "4";"6"};
(3,7)*+{}="00";
(7,3)*+{}="01";
{\ar@/_0.2pc/@{-}_(0.3){^{\xi\pprr}} "00";"01"};
(-3,-7)*+{}="10";
(-7,-3)*+{}="11";
{\ar@/_0.2pc/@{-}_(0.3){_{\eta\pprr}} "10";"11"};
\endxy
\]
Put $\del\pprr=\und{\Sf\pprr}$. We have $\sfr(\del\pprr)=[A\ov{\ovl{h}}{\lra}C\ov{\ovl{h\ppr}}{\lra}E]$ by definition.
Applying Proposition~\ref{PropMorphExSeq} to $\Sf,\Sf\pprr$ and
\[
\Tf=\SQ{s_0(h)}{\chi}=\ 
\xy
(-7,7)*+{A}="0";
(7,7)*+{B}="2";
(-7,-7)*+{A}="4";
(7,-7)*+{C}="6";
{\ar^{f} "0";"2"};
{\ar_{1_A} "0";"4"};
{\ar|*+{_h} "0";"6"};
{\ar^{g} "2";"6"};
{\ar_{h} "4";"6"};
(3,7)*+{}="00";
(7,3)*+{}="01";
{\ar@/_0.2pc/@{-}_{^{\chi}} "00";"01"};
(-3,-7)*+{}="10";
(-7,-3)*+{}="11";
{\ar@/_0.2pc/@{-}_(0.3){_{s_0(h)}} "10";"11"};
\endxy
\]
we obtain a morphism of the form
\[
{}_{1_A}\Cf_d=\CB{s_0(\xi\pprr)}{\Ycal_{\ff}}{\Zcal_{\ff}}{s_0(\eta\pprr)}{\Ycal_{\bb}}{\Zcal_{\bb}}=\ 
\xy
(-8,9)*+{A}="0";
(8,9)*+{B}="2";
(1,3)*+{O}="4";
(17,3)*+{D}="6";
(-8,-7)*+{A}="10";
(8,-7)*+{C}="12";
(1,-13)*+{O\pprr}="14";
(17,-13)*+{E}="16";
{\ar^{f} "0";"2"};
{\ar_{i} "0";"4"};
{\ar^{f\ppr} "2";"6"};
{\ar_(0.3){j} "4";"6"};
%
{\ar_{1_A} "0";"10"};
{\ar^(0.7){g}|!{(9,3);(13,3)}\hole "2";"12"};
{\ar_(0.3){o} "4";"14"};
{\ar^{d} "6";"16"};
{\ar^(0.3){h}|!{(1,-3);(1,-7)}\hole "10";"12"};
{\ar_{i\pprr} "10";"14"};
{\ar^{h\ppr} "12";"16"};
{\ar_{j\pprr} "14";"16"};
%
\endxy\ \ \co\Sf\to\Sf\pprr.
\]
By the dual of Lemma~\ref{LemForSym} {\rm (1)}, the square
\[
\Pf=\big(\Cf|_{\Delta^1\ti\{1\}\ti\Delta^1}\big)^t=\ 
\xy
(-7,7)*+{B}="0";
(7,7)*+{C}="2";
(-7,-7)*+{D}="4";
(7,-7)*+{E}="6";
{\ar^{g} "0";"2"};
{\ar_{f\ppr} "0";"4"};
{\ar^{h\ppr} "2";"6"};
{\ar_{d} "4";"6"};
%
%
\endxy
\]
is a push-out. By Lemma~\ref{LemExSeqPO}, using this push-out we obtain a morphism of exact sequences
\[
{}_{f\ppr}\Df_{1_F}=\CB{\Xcal_{\ff}\ppr}{\Ycal_{\ff}\ppr}{s_2(\xi\ppr)}{\Xcal_{\bb}\ppr}{s_1(\eta\ppr)}{s_2(\eta\ppr)}=\ 
\xy
(-8,9)*+{B}="0";
(8,9)*+{C}="2";
(1,3)*+{O\ppr}="4";
(17,3)*+{F}="6";
(-8,-7)*+{D}="10";
(8,-7)*+{E}="12";
(1,-13)*+{O\ppr}="14";
(17,-13)*+{F}="16";
{\ar^{g} "0";"2"};
{\ar_{i\ppr} "0";"4"};
{\ar^{g\ppr} "2";"6"};
{\ar_(0.3){j\ppr} "4";"6"};
%
{\ar_{f\ppr} "0";"10"};
{\ar^(0.7){h\ppr}|!{(9,3);(13,3)}\hole "2";"12"};
{\ar_(0.3){1_{O\ppr}} "4";"14"};
{\ar^{1_F} "6";"16"};
{\ar^(0.3){d}|!{(1,-3);(1,-7)}\hole "10";"12"};
{\ar_{i_0} "10";"14"};
{\ar^{e} "12";"16"};
{\ar_{j\ppr} "14";"16"};
%
\endxy\ \ \co\Sf\ppr\to f\ppr\sas\Sf\ppr.
\]
By the construction so far, we have a commutative diagram in $h\C$ 
\[
\xy
(-21,14)*+{A}="0";
(-7,14)*+{B}="2";
(7,14)*+{D}="4";
(21,14)*+{}="6";
(-21,0)*+{A}="10";
(-7,0)*+{C}="12";
(7,0)*+{E}="14";
(21,0)*+{}="16";
(-7,-14)*+{F}="22";
(7,-14)*+{F}="24";
%
%
{\ar^{\ovl{f}} "0";"2"};
{\ar^{\ovl{f^{\prime}}} "2";"4"};
%
{\ar@{=} "0";"10"};
{\ar_{\ovl{g}} "2";"12"};
{\ar^{\ovl{d}} "4";"14"};
{\ar_{\ovl{h}} "10";"12"};
{\ar_{\ovl{h^{\prime}}} "12";"14"};
%
{\ar_{\ovl{g^{\prime}}} "12";"22"};
{\ar^{\ovl{e}} "14";"24"};
{\ar@{=} "22";"24"};
%
%
{\ar@{}|\circlearrowright "0";"12"};
{\ar@{}|\circlearrowright "2";"14"};
{\ar@{}|\circlearrowright "12";"24"};
\endxy
\]
satisfying $\sfr(\ovl{f\ppr}\sas\del\ppr)=[D\ov{\ovl{d}}{\lra}E\ov{\ovl{e}}{\lra}F]$.
Besides, the existence of $\Cf$ shows $\del=\ovl{d}\uas\del\pprr$ by Lemma~\ref{LemET2}.

It remains to show $\ovl{f}\sas\del\pprr=\ovl{e}\uas\del\ppr$. By Lemma~\ref{LemET2}, it suffices to show the existence of a morphism ${}_{f}\Cf\ppr_{e}\co\Sf\pprr\to\Sf\ppr$. Remark that we have cubes
\begin{eqnarray*}
&\CB{\Ycal_{\ff}}{s_0(\xi\pprr)}{s_0(\eta\pprr)}{\Zcal_{\ff}}{\Zcal_{\bb}}{\Ycal_{\bb}}=\ 
\xy
(-8,9)*+{A}="0";
(8,9)*+{A}="2";
(1,3)*+{O}="4";
(17,3)*+{O\pprr}="6";
(-8,-7)*+{B}="10";
(8,-7)*+{C}="12";
(1,-13)*+{D}="14";
(17,-13)*+{E}="16";
{\ar^{1_A} "0";"2"};
{\ar_{i} "0";"4"};
{\ar^{i\pprr} "2";"6"};
{\ar_(0.3){o} "4";"6"};
%
{\ar_{f} "0";"10"};
{\ar^(0.7){h}|!{(9,3);(13,3)}\hole "2";"12"};
{\ar_(0.3){j} "4";"14"};
{\ar^{j\pprr} "6";"16"};
{\ar^(0.3){g}|!{(1,-3);(1,-7)}\hole "10";"12"};
{\ar_{f\ppr} "10";"14"};
{\ar^{h\ppr} "12";"16"};
{\ar_{d} "14";"16"};
%
\endxy&,\\
&\CB{s_2(\eta\ppr)}{s_2(\xi\ppr)}{\Ycal_{\ff}\ppr}{s_1(\eta\ppr)}{\Xcal_{\bb}\ppr}{\Xcal_{\ff}\ppr}=\ 
\xy
(-8,9)*+{B}="0";
(8,9)*+{C}="2";
(1,3)*+{D}="4";
(17,3)*+{E}="6";
(-8,-7)*+{O\ppr}="10";
(8,-7)*+{F}="12";
(1,-13)*+{O\ppr}="14";
(17,-13)*+{F}="16";
{\ar^{g} "0";"2"};
{\ar_{f\ppr} "0";"4"};
{\ar^{h\ppr} "2";"6"};
{\ar_(0.3){d} "4";"6"};
%
{\ar_{i\ppr} "0";"10"};
{\ar^(0.7){g\ppr}|!{(9,3);(13,3)}\hole "2";"12"};
{\ar_(0.3){i_0} "4";"14"};
{\ar^{e} "6";"16"};
{\ar^(0.3){j\ppr}|!{(1,-3);(1,-7)}\hole "10";"12"};
{\ar_{1_{O\ppr}} "10";"14"};
{\ar^{1_F} "12";"16"};
{\ar_{j\ppr} "14";"16"};
%
\endxy&
\end{eqnarray*}
as transpositions of $\Cf$ and $\Df$. 
As in Proposition~\ref{PropComposeCubes}, we obtain a map $\Delta^1\ti\Delta^1\ti\Delta^2\to\C$ 
\[
\xy
(-8,16)*+{A}="0";
(8,16)*+{A}="2";
(1,10)*+{O}="4";
(17,10)*+{O\pprr}="6";
(-8,0)*+{B}="10";
(8,0)*+{C}="12";
(1,-6)*+{D}="14";
(17,-6)*+{E}="16";
(-8,-16)*+{O\ppr}="20";
(8,-16)*+{F}="22";
(1,-22)*+{O\ppr}="24";
(17,-22)*+{F}="26";
{\ar^{1_A} "0";"2"};
{\ar_{i} "0";"4"};
{\ar^{i\pprr} "2";"6"};
{\ar_(0.3){o} "4";"6"};
%
{\ar_{f} "0";"10"};
{\ar^(0.7){h}|!{(9,10);(13,10)}\hole "2";"12"};
{\ar_(0.3){j} "4";"14"};
{\ar^{j\pprr} "6";"16"};
{\ar^(0.3){g}|!{(1,4);(1,0)}\hole "10";"12"};
{\ar_{f\ppr} "10";"14"};
{\ar^{h\ppr} "12";"16"};
{\ar_(0.7){d} "14";"16"};
%
{\ar_{i\ppr} "10";"20"};
{\ar^(0.7){g\ppr}|!{(9,-6);(13,-6)}\hole "12";"22"};
{\ar_(0.3){i_0} "14";"24"};
{\ar^{e} "16";"26"};
{\ar^(0.3){j\ppr}|!{(1,-12);(1,-16)}\hole "20";"22"};
{\ar_{1_{O\ppr}} "20";"24"};
{\ar^{1_F} "22";"26"};
{\ar_{j\ppr} "24";"26"};
%
\endxy
\]
compatibly with these cubes. Especially it contains $4$-simplices
\begin{eqnarray*}
\Thh=\ \Penta{A}{B}{C}{E}{F}{f}{g}{h\ppr}{e}&,&
\Xi=\ \Penta{A}{O}{D}{O\ppr}{F}{i}{j}{i_0}{j\ppr},\\
\Phi=\ \Penta{A}{B}{D}{O\ppr}{F}{f}{f\ppr}{i_0}{j\ppr}&,&
\Psi=\ \Penta{A}{O}{O\pprr}{E}{F}{i}{o}{j\pprr}{e}
\end{eqnarray*}
such that $\Sf\pprr=\SQ{d_1d_4(\Psi)}{d_1d_4(\Thh)}, \Sf\ppr=\SQ{d_0d_2(\Phi)}{d_0d_3(\Thh)}$ and
\[ d_2d_3(\Thh)=d_2d_2(\Phi),\ d_1d_1(\Thh)=d_1d_1(\Psi),\ d_1d_2(\Xi)=d_1d_2(\Phi). \]
Applying Lemma~\ref{LemReplZero} to $d_2(\Xi)$ and $d_3(\Psi)$, we obtain a $3$-simplex $\mu$ such that
$d_1(\mu)=d_1d_2(\Xi)$ 
and
$d_2(\mu)=d_1d_3(\Psi)$. 
If we put
\[ \Cf\ppr=\CB{d_3(\Thh)}{s_1d_1d_3(\Thh)}{d_1(\Thh)}{d_2(\Phi)}{\mu}{d_1(\Psi)}, \]
then this give a morphism ${}_f\Cf\ppr_e\co\Sf\pprr\to\Sf\ppr$ as desired.
\end{proof}

\begin{dfn}
 Inspired from~\cite{S}, we call \emph{topological} any extriangulated category which is equivalent to the homotopy category of an exact quasi-category.
\end{dfn}

\begin{lem}\label{LemExtClosed}
Let $\C$ be an exact quasi-category, and let $\D$ be a full subcategory of $\C$ closed by homotopy equivalences of objects in $\C$.
Assume that $\D$ is extension-closed in $\C$, meaning that, $\D$ contains the zero objects of $\C$ and that, for any exact sequence $_A\Sf_C$ in $\C$ with $A,C\in\D_0$, the middle term $B$ also belongs to $\D_0$.
Define $(\D_{\dag})_1$ by considering those morphisms in $\D$ that are ingressive in $\C$ and for which some cofiber belongs to $\D$. (For $n\ge 2$, we define $(\D_{\dag})_n\se(\C_{\dag})_n\cap\D_n$ to be the subset consisting of $n$-simplices whose edges are in $(\D_{\dag})_1$. The face and degeneracy maps in $\D_{\dag}$ are restrictions of those in $\D$.)
Define $\D^{\dag}$ dually.
Then $(\D,\D_{\dag},\D^{\dag})$ is an exact quasi-category.
\end{lem}
\begin{proof}
Stability under (split) extensions implies that $\D$ is additive.
To see that $\D_{\dag}$ is a subcategory, let
\[
\TwoSP{X}{Y}{Z}{f}{g}{h}{}\]
any $2$-simplex in $\D$ in which $f,g$ belong to $(\D_{\dag})_1$.
Consider the diagram of ambigressive push-outs
\[
\xy
(-18,12)*+{X}="0";
(-6,12)*+{Y}="2";
(6,12)*+{Z}="4";
(-18,0)*+{O}="10";
(-6,0)*+{Z\ppr}="12";
(6,0)*+{Y\ppr}="14";
(-6,-12)*+{O\ppr}="22";
(6,-12)*+{X\ppr}="24";
{\ar^{f} "0";"2"};
{\ar^{g} "2";"4"};
{\ar "0";"10"};
{\ar "2";"12"};
{\ar "4";"14"};
{\ar "10";"12"};
{\ar "12";"14"};
{\ar "12";"22"};
{\ar "14";"24"};
{\ar "22";"24"};
\endxy
\]
where $Z\ppr$ and $X\ppr$ belong to $\D$ by assumption.
The bottom square is an exact sequence, hence $Y\ppr$ also belongs to $\D$.
The outer upper rectangle shows that $X\ov{h}{\lra} Z$ belongs to $\D_{\dag}$.
Dually, $\D^{\dag}$ is a subcategory of $\D$.
The remaining axioms readily follow from those for $\C$.
\end{proof}

\begin{prop}
 Every extension-closed full subcategory, containing zero, of a topological extriangulated category is topological.
\end{prop}
\begin{proof}
 Let $\C$ be an exact quasi-category and let $\mathcal{D}$ be a full subcategory of $h\C$ that we assume extension-closed.
 By \cite[Remark 2.18]{NP}, $\mathcal{D}$ inherits an extriangulated structure from that of $h\C$.
 Define $\D$ to be the full subcategory of $\C$ given by $\mathcal{D}$.
 More explicitly, $\D$ is given by the pull-back, in the category of simplicial sets:
 \[
\xy
(-10,6)*+{\D}="0";
(10,6)*+{\C}="2";
(-10,-6)*+{N(\mathcal{D})}="4";
(10,-6)*+{N(h\C)}="6";
{\ar "0";"2"};
{\ar "0";"4"};
{\ar "2";"6"};
{\ar "4";"6"};
\endxy
\]
Then $\D$ is extension-closed in $\C$.
Indeed, any exact sequence $X\to Y\to Z$ in $\C$ with $X,Z\in\D_0$ induce, by \cref{ThmET}, an extriangle $X\infl Y\defl Z\dashrightarrow$ in $h\C$.
By assumption, we have $Y\in\mathcal{D}$.
We can thus apply \cref{LemExtClosed}: $\D$ is an exact quasi-category.
The claim follows since $h\D=\mathcal{D}$.
\end{proof}

We plan to investigate the following
\begin{question}\label{QuestionTopoExtricat}
Is every topological extriangulated category equivalent to some extension-closed full subcategory of a triangulated category?
\end{question}
Our next result is an elementary tool that might help answering \cref{QuestionTopoExtricat}.

Let $F\co\C\to\D$ be an exact functor of exact quasi-categories \cite[Definition 4.1]{B1}.
In other words, $F$ is a map of simplicial sets that preserves zero objects, ingressive morphisms, egressive morphisms, push-outs along ingressive morphisms and pull-backs along egressive morphisms.

Because $F$ is a map of simplicial sets, it commutes to faces and degeneracies and thus induces a functor $hF\co h\C\to h\D$ which coincide with $F$ on objects.
It is defined on morphisms by $hF\, \ovl{a} = \ovl{Fa}$.

\begin{dfn}[{\cite[Definition 2.31]{BTS}}]
 A functor $G\co(\mathcal{C},\E,\sfr)\to(\mathcal{D},\E\ppr,\sfr\ppr)$ between extriangulated categories is \emph{exact} (or \emph{extriangulated}) if it is additive, and if there is a natural transformation $\Gamma\co\E\tc\E\ppr(G\op-,G-)$ such that, for any $\sfr$-triangle
 $A\ov{i}{\infl}B\ov{p}{\defl}C\ov{\delta}{\dashrightarrow}$ in $\mathcal{C}$, its image
 $GA\ov{Gi}{\infl}GB\ov{Gp}{\defl}GC\ov{\Gamma_{C,A}(\delta)}{\dashrightarrow}$ is an $\sfr\ppr$-triangle in $\mathcal{D}$.
\end{dfn}

\begin{prop}\label{propExactFunctors}
 Let $F\co\C\to\D$ be an exact functor of exact quasi-categories.
 Then $hF\co h\C\to h\D$ is an exact functor of extriangulated categories.
\end{prop}
\begin{proof}
Because $F$ preserves push-outs of ingressive morphisms, it preserves coproducts, and similarly for products.
It follows that $hF$ is additive.
Define $\Gamma$ as follows: For any $A,C\in\C_0$, and any $\delta=\und{\Sf}\in\E(C,A)$, let $\Gamma_{C,A}(\delta) = \und{F\Sf}$. The map $\Gamma_{C,A}$ is well-defined since if $_a\Cf_c$ is a cube in $\C$ from $\Sf$ to $\Sf\ppr$ with $\ovl{a} = \id_A$ and $\ovl{c}=\id_C$ in $h\C$, then $_{Fa}\Cf\ppr_{Fc}=F\Cf$ is a cube in $\D$ from $F\Sf$ to $F\Sf\ppr$ with $\ovl{Fa}=\id_{FA}$ and $\ovl{Fc}=\id_{FC}$ in $h\D$.
It remains to show that $\Gamma = (\Gamma_{C,A})$ is natural.
Let $\delta=\und{\Sf}\in\E(C,A)$ and let $a\in\C_1(A,A\ppr)$. Consider a cube 
\[
\Df=
\xy
(-8,9)*+{A}="0";
(8,9)*+{B}="2";
(1,3)*+{O}="4";
(17,3)*+{C}="6";
(-8,-7)*+{A\ppr}="10";
(8,-7)*+{B_0}="12";
(1,-13)*+{O}="14";
(17,-13)*+{C}="16";
{\ar^{x} "0";"2"};
{\ar_{i} "0";"4"};
{\ar^{y} "2";"6"};
{\ar^(0.3){j} "4";"6"};
{\ar_{a} "0";"10"};
{\ar^(0.6){b_0}|!{(9,3);(13,3)}\hole "2";"12"};
{\ar^(0.4){1_O} "4";"14"};
{\ar^{1_C} "6";"16"};
{\ar^(0.3){x_0}|!{(1,-3);(1,-7)}\hole "10";"12"};
{\ar_{i_0} "10";"14"};
{\ar^{y_0} "12";"16"};
{\ar_{j} "14";"16"};
\endxy
\]
defining $a\sas\Sf$ as in \cref{LemExSeqPO}.
In particular, the back square is a push-out, with $x$ ingressive.
Since $F$ preserves push-outs along ingressive morphisms, the cube $F\Df$ witnesses the fact that $F(a\sas\Sf)=(Fa)\sas F\Sf$.
Dually, for any $C\ppr\ov{c}{\lra}C$, we have $F(c\uas\Sf)=(Fc)\uas F\Sf$.
Whence $\Gamma_{C\ppr,A\ppr}\circ \ovl{c}\uas \ovl{a}\sas = (hF\ovl{c})\uas(hF\ovl{a})\sas \circ \Gamma_{C,A}$.
\end{proof}

\subsection{Compatibility with the triangulation in stable case}\label{Subsection_StableCase}

In this subsection, let $\C$ be a stable quasi-category, which is 
a particular case of an exact quasi-category, as in \cite[Example~3.3]{B1}. In this case it is known that the homotopy category $h\C$ can be equipped with a structure of a triangulated category (\cite{L2}).
Let us show that the structure of an extriangulated category on $h\C$ obtained in Theorem \ref{ThmET} is compatible with it.

First let us briefly review the construction of the shift functor. For any object $A\in\C_0$, we choose an exact sequence
\[
\Uf_A=\SQ{\eta_A}{\ze_A}=\ 
\xy
(-6,6)*+{A}="0";
(6,6)*+{O_A}="2";
(-6,-6)*+{O_A\ppr}="4";
(6,-6)*+{\Sig A}="6";
{\ar^{} "0";"2"};
{\ar_{} "0";"4"};
{\ar^{} "2";"6"};
{\ar_{} "4";"6"};
\endxy
\quad\Bigg(\text{respectively,}\ 
\Vf_A=\ 
\xy
(-6,6)*+{\Om A}="0";
(6,6)*+{O_A\pprr}="2";
(-6,-6)*+{O_A^{\prime\prime\prime}}="4";
(6,-6)*+{A}="6";
{\ar^{} "0";"2"};
{\ar_{} "0";"4"};
{\ar^{} "2";"6"};
{\ar_{} "4";"6"};
\endxy\Bigg)
\]
in which $O_A,O_A\ppr$ (resp. $O_A\pprr,O_A^{\prime\prime\prime}$) are zero objects.

If we put $\mu_A=\und{\Uf_A}\in\E(\Sig A,A)$ (resp. $\nu_A=\und{\Vf_A}\in\E(A,\Om A)$), then we have an $\sfr$-triangle
\begin{equation}\label{ShiftsTria}
A\to0\to\Sig A\ov{\mu_A}{\dra},\quad (\text{resp.}\ \ \Om A\to0\to A\ov{\nu_A}{\dra}).
\end{equation}
Especially, as in \cite[Proposition~3.3]{NP} we obtain natural isomorphisms of functors induced by the Yoneda's lemma
\[ (\mu_A)\ssh\co(h\C)(-,\Sig A)\ov{\cong}{\ltc}\E(-,A) \]
which assigns $(\mu_A)\ssh(\cfr)=\cfr\uas(\mu_A)$ to each $\cfr\in(h\C)(C,\Sig A)$. Thus for any $\afr\in(h\C)(A,A\ppr)$, there exists a morphism $\Sig\afr\in(h\C)(\Sig A,\Sig A\ppr)$ uniquely so that
\[
\xy
(-14,6)*+{(h\C)(-,\Sig A)}="0";
(14,6)*+{\E(-,A)}="2";
(-14,-6)*+{(h\C)(-,\Sig A\ppr)}="4";
(14,-6)*+{\E(-,A\ppr)}="6";
{\ar@{=>}^(0.56){(\mu_A)\ssh} "0";"2"};
{\ar@{=>}_{(\Sig\afr)\ci-} "0";"4"};
{\ar@{=>}^{\afr\sas} "2";"6"};
{\ar@{=>}_(0.56){(\mu_{A\ppr})\ssh} "4";"6"};
{\ar@{}|\circlearrowright "0";"6"};
\endxy
\]
becomes commutative. Then the correspondence
\begin{itemize}
\item for any object $A\in h\C$, associate $\Sig A\in h\C$ using the chosen $\Uf_A$,
\item for any morphism $\afr\in(h\C)(A,A\ppr)$, associate the morphism $\Sig\afr\in(h\C)(\Sig A,\Sig A\ppr)$ by the above
\end{itemize}
gives the endofunctor $\Sig\co h\C\to h\C$. By Lemma~\ref{LemET2}, morphism $\Sig\ovl{a}$ for $a\in\C_1(A,A\ppr)$ is equal to the one given by $\Sig\ovl{a}=\ovl{s}$ where $s\in\C_1(\Sig A,\Sig A\ppr)$ admits some morphism ${}_a\Cf_{s}\co\Uf_A\to\Uf_{A\ppr}$ of exact sequences. Similarly for the functor $\Om\co h\C\to h\C$, which gives a quasi-inverse of $\Sig$.

Using these structures, we can complete any morphism $\ovl{z}\in(h\C)(C,\Sig A)$ into a triangle diagram of the form $A\to B\to C\ov{\ovl{z}}{\lra}\Sig A$ compatibly with the external triangulation in the following way. Indeed, this is how to relate triangulations with external triangulations (\cite[Proposition~3.22]{NP}).
\begin{itemize}
\item For each $\ovl{z}\in(h\C)(C,\Sig A)$, take
\begin{equation}\label{DistTriET}
\sfr((\mu_A)\ssh(\ovl{z}))=[A\ov{\ovl{x}}{\lra}B\ov{\ovl{y}}{\lra}C]
\end{equation}
to obtain $A\ov{\ovl{x}}{\lra}B\ov{\ovl{y}}{\lra}C\ov{\ovl{z}}{\lra}\Sig A$.
\end{itemize}

\begin{rem}
In fact, the existence of $\sfr$-triangles of the form $(\ref{ShiftsTria})$ is equivalent to that the extriangulated category $(h\C,\E,\sfr)$ admits a compatible triangulation. Indeed, existence of such $\sfr$-triangles is equivalent to that $(h\C,\E,\sfr)$ is \emph{Frobenius} with trivial projective-injectives in the sense of \cite[Definition~7.1]{NP} and \cite[Subsection~3.3]{ZZ}, and thus $(h\C,\Sig,\triangle)$ becomes a triangulated category by \cite[Corollary~7.6]{NP} or \cite[Example~3.15]{ZZ}, in which $\triangle$ is defined to be the class of triangles isomorphic to those $A\ov{\ovl{x}}{\lra}B\ov{\ovl{y}}{\lra}C\ov{\ovl{z}}{\lra}\Sig A$ obtained in the above way.
\end{rem}

On the other hand, the distinguished triangles in the triangulation given in \cite{L2} are those isomorphic to
$A\ov{\ovl{x\ppr}}{\lra}B\ppr\ov{\ovl{y\ppr}}{\lra}C\ov{\ovl{z}}{\lra}\Sig A$ which fits in some rectangle 
\begin{equation}\label{Recta}
\Rf=\RT{\Xcal}{\Ycal}{\Zcal}=\ 
\xy
(-16,7)*+{A}="0";
(0,7)*+{B\ppr}="2";
(16,7)*+{O_A}="4";
(-16,-7)*+{O_A\ppr}="10";
(0,-7)*+{C}="12";
(16,-7)*+{\Sig A}="14";
{\ar^{x\ppr} "0";"2"};
{\ar_{} "2";"4"};
%
{\ar_{} "0";"10"};
{\ar_{y\ppr} "2";"12"};
{\ar^{} "4";"14"};
{\ar^{} "10";"12"};
{\ar^{z} "12";"14"};
%
\endxy
\end{equation}
such that $\Rf_{\mathrm{out}}=\Uf_A$ and that $\Rf_{\mathrm{right}}$ is a pull-back.
Thus to see the compatibility, it is enough to show the following.
\begin{prop}\label{PropCompatStable}
For any morphism $z\in\C_1(C,\Sig A)$ and any rectangle $(\ref{Recta})$ in which $\Rf_{\mathrm{out}}=\Uf_A$ holds and $\Rf_{\mathrm{right}}$ is a pull-back, we have $\sfr((\mu_A)\ssh(\ovl{z}))=[A\ov{\ovl{x\ppr}}{\lra}B\ppr\ov{\ovl{y\ppr}}{\lra}C]$.
\end{prop}
\begin{proof}
In the rectangle $\Rf$, we see that $\Sf\ppr=\Rf_{\mathrm{left}}$ is an exact sequence, and satisfies $\sfr(\und{\Sf\ppr})=[A\ov{\ovl{x\ppr}}{\lra}B\ppr\ov{\ovl{y\ppr}}{\lra}C]$ by definition. Moreover, $\Rf$ gives a morphism of exact sequences
\[ {}_{1_A}\Cf_{z}=\CB{s_0(\ze_A)}{\Zcal}{\Ycal}{s_0(\eta_A)}{s_1(\eta_A)}{\Xcal}\co\Sf\ppr\to\Uf_A, \]
which shows $(\mu_A)\ssh(\ovl{z})=(\ovl{z})\uas\und{\Uf_A}=\und{\Sf\ppr}$ by Lemma~\ref{LemET2}.
\end{proof}

\end{document}